\documentclass[onefignum,onetabnum]{siamart190516}

\usepackage{lipsum}
\usepackage{amsfonts}
\usepackage{graphicx}
\usepackage{epstopdf}
\usepackage{verbatim,amscd}
\usepackage{latexsym, bm, float, color}
\usepackage{mathtools}
\usepackage{tikz}
\usepackage{enumitem}
\usepackage{subfig}
\usepackage{listings}
\lstdefinestyle{jaxpy}{
  language=Python,
  basicstyle=\footnotesize\ttfamily,
  keywordstyle=\color[rgb]{0,0.5,0}\bfseries,
  commentstyle=\color[rgb]{0.25,0.5,0.5}\itshape,
  stringstyle=\color[rgb]{0.73,0.13,0.13},
  showstringspaces=false,
  columns=fullflexible,
  keepspaces=true,
  breaklines=true,
  frame=single,
  mathescape=true
}
\usepackage{siunitx}
\usepackage{algpseudocode}
\ifpdf
  \DeclareGraphicsExtensions{.eps,.pdf,.png,.jpg}
\else
  \DeclareGraphicsExtensions{.eps}
\fi

\newsiamremark{remark}{Remark}
\newsiamremark{hypothesis}{Hypothesis}
\crefname{hypothesis}{Hypothesis}{Hypotheses}
\newsiamthm{claim}{Claim}

\newsiamthm{assumption}{Assumption}

\headers{Randomized Forward Mode AD for optimization}{K. Shukla and Y. Shin}

\title{Randomized forward mode of automatic differentiation for optimization algorithms
	\thanks{Submitted to the editors DATE.}}

\author{Khemraj Shukla\thanks{Division of Applied Mathematics, Brown University, Providence, RI 02912-9056 USA
  (\email{khemraj\_shukla@brown.edu}).}
  \and 
  Yeonjong Shin\thanks{Department of Mathematics, North Carolina State University, Raleigh, NC 27695-8205 USA
  (\email{yeonjong\_shin@ncsu.edu}). Corresponding author.}
}

\usepackage{amsopn}

\newcommand{\bx}{\bm{x}}
\newcommand{\bz}{\bm{z}}

\makeatletter

\@ifundefined{cref@override@label@type}{}{%
  \let\cref@ovr@orig\cref@override@label@type
  \long\def\cref@override@label@type#1\@nil#2{%
    \cref@ovr@probe#1[][][]\@nil@@{#2}}%
  \long\def\cref@ovr@probe#1[#2][#3][#4]#5\@nil@@#6{%
    \ifx\relax#1\relax [#6][#3][#4]#2\else [#6][][]#1\fi}%
}
\makeatother

\begin{document}
\maketitle
\begin{abstract}
Reverse mode automatic differentiation (AD) computes a gradient in a single sweep, but at a memory cost that grows with the length of the computational graph. That cost becomes the binding constraint when the graph is long, as in parameter identification governed by differential equations or in training on memory-limited hardware. We study a randomized forward mode gradient (RFG) as an alternative to backpropagation in precisely this regime. RFG is a random gradient estimator built from the directional derivative along a random vector, a quantity that forward mode AD computes without storing a tape.
By replacing the gradient with RFG, a class of RFG-based optimization algorithms is obtained. In particular, we consider gradient descent (GD) and Polyak's heavy ball (PHB) methods, and establish convergence results for the RFG-based GD and PHB algorithms on quadratic functions.
The analysis reveals that using random vectors from a probability distribution with the smallest kurtosis yields the fastest rate of convergence.
Computational experiments demonstrate the performance of the proposed algorithms and verify the theoretical findings, including beyond quadratic functions.
We further characterize when RFG is computationally preferable to backpropagation.
Since the variance of RFG grows with the parameter dimension $d$ while the memory required by the reverse mode grows with the length $\ell$ of the computational graph, the relevant figure of merit is the ratio $\ell/d$.
We confirm this by measurement against reverse mode \emph{with checkpointing}. For a physics-informed neural network, where $d$ is large and $\ell$ short, RFG is not competitive. For parameter identification governed by differential equations, where a handful of coefficients is recovered from a long time integration, RFG requires one to three orders of magnitude less memory. The price is a larger iteration count, which leaves its time to solution comparable to or worse than a well-tuned reverse mode.
\end{abstract}

 
	
	
\section{Introduction}
Modern computational problems continue to grow in size, and with them the demand for efficient methods for large-scale, high-dimensional optimization.
First-order methods that use gradients are widely employed for their theoretical guarantees, simple implementation, and strong empirical performance across scientific and engineering applications.
Their use, in turn, places a premium on computing gradients quickly and with little memory.

Automatic differentiation (AD) is a central computational technique that meets this need.
It plays a pivotal role in many research fields where derivative computations are essential.
This includes but is not limited to deep learning \cite{baydin2018automatic}, optimization \cite{dunning2017jump}, scientific computing \cite{innes2019differentiable}
and more recently, scientific machine learning \cite{karniadakis2021physics}.
For gradient-based optimization of real-valued functions, the reverse mode of AD computes gradients efficiently via backpropagation.
However, there are doubts about backpropagation-based training, stemming mainly from a neuroscience perspective \cite{lillicrap2020backpropagation,hinton2022forward}---%
if the neural network were a model of the human brain,
it should be trained in a similar way to how the cortex learns.
In that sense, backpropagation is biologically implausible as the brain does not work in that way \cite{scellier2017equilibrium}.
In addition, there is a need for alternatives to backpropagation that reduce the computational time and energy costs of neural network training, and that thereby enable efficient hardware design tailored to deep learning \cite{baydin2022gradients}.

On the other hand, forward-mode AD is a type of AD algorithm that computes directional derivatives by means of only a single forward evaluation (without backpropagation). This feature allows one to compute gradients efficiently through the Jacobian-Vector Product (JVP), especially when the number of outputs is much greater than the number of inputs.
The two modes differ in what they return and in what they cost.
Reverse-mode AD returns the full gradient in a single sweep, at the price of a tape whose size grows in proportion to the length of the computational graph; forward-mode AD returns a single directional derivative, but never builds a tape, so its memory footprint is independent of graph length.
This trade---one directional derivative per evaluation, at a memory cost that does not grow with depth---is the property this paper exploits.
It is what makes forward-mode AD attractive whenever memory, rather than time, is the binding constraint: long rollouts of time-stepping solvers, PDE-constrained parameter identification, and training on hardware with a hard memory ceiling such as edge devices.
We stress at the outset that this is not a claim of a general speed advantage. As our measurements in section~\ref{sec:efficiency} show, the randomized forward-mode methods studied here are slower to a given tolerance than a well-tuned reverse-mode baseline, and the case for them rests on memory.



A pioneering work \cite{baydin2022gradients} proposed an unbiased gradient estimator defined by a standard normal random vector multiplied by the directional derivative along the random vector.
While some promising empirical results were demonstrated in \cite{baydin2022gradients}, further investigations are needed.

The present work considers the gradient estimator of \cite{baydin2022gradients} with a general probability distribution
and investigates its statistical properties.
Because the estimator allows flexibility in choosing a probability distribution, we call it a randomized forward mode gradient (RFG).
Through the second-moment analysis, we prove in Theorem~\ref{thm:2nd-moment-RFG} that the smallest expected relative error of RFG is achieved with a probability distribution having the minimum kurtosis $\kappa_4$ (the fourth standardized moment) and variance $\frac{1}{d+\kappa_4-1}$, where $d$ is the dimension of the problem.
The resulting estimator is deliberately biased. It satisfies $\mathbb{E}[\nabla_{\bz,0}^\text{FM}f(\bx)] = \frac{1}{d+\kappa_4-1}\nabla f(\bx)$, in contrast with the unbiased choice made in \cite{baydin2022gradients}.
The bias is benign in the optimization setting, since a constant rescaling of the estimator is indistinguishable from a rescaling of the learning rate. What governs the performance of the resulting algorithms is the \emph{relative} variance, and it is this quantity that the minimum-kurtosis choice minimizes.

By replacing the gradient with RFG, one obtains a class of RFG-based optimization algorithms. To give concrete algorithms, we consider gradient descent (GD) and Polyak's heavy ball (PHB) methods. By focusing on quadratic objective functions, we present a convergence analysis for the RFG-based GD and PHB methods. 
We found that, unlike the vanilla PHB, the RFG-based PHB converges even with a negative momentum parameter.

Beyond the analysis, we examine when RFG is computationally preferable to
backpropagation. The two modes stand in a trade-off that the complexity
analysis of section~\ref{sec:complexity} makes precise. The forward mode
removes from the memory requirement a factor equal to the length of the
computational graph, while the analysis of section~\ref{sec:analysis} shows
that it pays a factor $d+\kappa_4-1$ in the number of iterations. Which effect
dominates is a property of the problem rather than of the method. We find that
for parameter identification problems governed by differential equations, in
which a handful of coefficients is recovered from a long integration, RFG
requires between one and three orders of magnitude less memory than reverse
mode with checkpointing, and that its memory is independent of the length of the
integration where that of the reverse mode grows linearly. It is correspondingly
slower to reach a given accuracy, and where the parameter dimension is large and
the graph short, as in neural network training, it is not competitive.
Computational examples are provided to demonstrate the performance of RFG-based optimization algorithms for six different probability distributions (Bernoulli, Uniform, Wigner, Gaussian, Laplace, and the discrete distribution recovering randomized coordinate descent) and verify theoretical findings.

The proposed RFG-based optimization encompasses a class of existing randomized optimization schemes as special cases, obtained by particular choices of the direction distribution $\text{P}$.
In particular, taking $\text{P}$ to be the uniform distribution over the scaled signed coordinate vectors $\{\pm\alpha\bm{e}_i\}_{i=1}^d$ recovers the randomized coordinate descent (RCD) method \cite{nesterov2012efficiency,wright2015coordinate}, as noted in Remark~\ref{rmk:RCD}. The block-coordinate variant is obtained analogously by drawing $\bz$ uniformly from the scaled indicator vectors of a partition of the coordinates.
Choosing $\text{P}$ to be a scaled Bernoulli mask reproduces the direction structure used by dropout-style updates.
This unification lets the analysis of section~\ref{sec:analysis} be read as a statement about a family of randomized first-order methods rather than about a single algorithm.

\textbf{Related work.}
Estimating a gradient from directional information along random directions has a long history in stochastic approximation and derivative-free optimization, and we position the present contribution against it explicitly.
Simultaneous perturbation stochastic approximation \cite{spall1992multivariate} constructs a gradient estimate from finite differences along Rademacher directions, and already advocates the symmetric Bernoulli choice on variance grounds. The estimator analyzed here coincides with that construction when $h > 0$, and Theorem~\ref{thm:2nd-moment-RFG} may be read as identifying the precise sense in which that choice is optimal, namely minimality of the kurtosis $\kappa_4$ within the class of Assumption~\ref{assumption:prob-dist}.
Random gradient-free methods with Gaussian directions were analyzed by Nesterov and Spokoiny \cite{nesterov2017random}, who established the $\mathcal{O}(d)$ degradation in iteration complexity relative to full-gradient methods. The factor $(\kappa_4 + d -1)^{-1}$ appearing in the rate of Theorem~\ref{thm:RFG-GD-convg} is the analogue of that penalty in the present setting.
Optimal rates for two-point zeroth-order schemes are due to \cite{duchi2015optimal}, and nonconvex stochastic variants to \cite{ghadimi2013stochastic}.
The RCD special case of Remark~\ref{rmk:RCD} connects to \cite{nesterov2012efficiency,richtarik2014iteration}.
What distinguishes the present work from this literature is the use of \emph{exact} directional derivatives supplied by forward-mode AD in place of finite differences, which removes the truncation error entirely at $h=0$, and the resulting memory profile analyzed in section~\ref{sec:complexity}.

A second, more recent line of work is closer to ours.
Within the AD-based line initiated by \cite{baydin2022gradients}, subsequent studies have examined directional descent \cite{silver2022learning} and the scalability of forward gradients, with \cite{ren2023scaling} recovering competitive performance only by combining forward gradients with local losses, and \cite{fournier2023can} reporting that the plain estimator degrades as the parameter dimension grows.
Our analysis is consistent with these empirical findings and identifies their source. The variance of the RFG grows proportionally to $d + \kappa_4 - 1$, so the dimension penalty is intrinsic to the estimator rather than an artifact of a particular architecture.
Section~\ref{sec:efficiency} accordingly characterizes the regime in which the memory advantage of the forward mode outweighs this penalty, instead of claiming a uniform improvement.

The rest of the paper is organized as follows.
Upon introducing the preliminaries on AD in section~\ref{sec:setup},
the RFG and the RFG-based optimization algorithms are presented in section~\ref{sec:RFG}
along with the second-moment analysis.
Section~\ref{sec:analysis} is devoted to the convergence analysis of RFG-based GD and PHB algorithms for quadratic objective functions. 
Computational examples are presented in section~\ref{sec:example}:
sections~\ref{sec:example}.1--\ref{sec:example}.3 verify the theoretical
findings and examine the behavior of the method on standard test problems and
on three scientific machine learning tasks, section~\ref{sec:inverse} introduces
four parameter identification problems governed by differential equations, and
section~\ref{sec:efficiency} uses them to compare the computational cost of RFG
against backpropagation in memory and in time to solution.

\section{Preliminaries on automatic differentiation} \label{sec:setup}



This section reviews AD in the detail that the later sections require. 
We discuss the forward and the reverse modes and elaborate on the differences between them.
Pedagogical examples are also included along with the snippets of JAX \cite{jax2018github} code.

\subsection{Forward mode AD or Jacobian-vector product}
In forward mode AD, the derivative of a function is calculated by evaluating both the function and its derivative simultaneously. It proceeds in a forward direction from the input to the output, tracking derivatives at each step.
A single forward sweep propagates the derivative along one input direction, so this approach is efficient for functions with few inputs and many outputs; recovering a full Jacobian requires one sweep per input coordinate. Forward mode AD utilizes the concept of dual numbers.
The dual numbers are expressions of the form 
\begin{equation}\label{dual_number}
    a+b\epsilon,
\end{equation}
where $a, b \in \mathbb{R}$ and the symbol $\epsilon$ satisfies $\epsilon^2 = 0$ with $\epsilon \ne 0$.
Dual numbers can be added component-wise and multiplied by the formula
\begin{equation*}
    (a+b\epsilon)(c+d\epsilon) = ac + (ad+bc)\epsilon.
\end{equation*}
Note that any real number $a$ can be identified with the corresponding dual number of $a + 0\epsilon$.
Let $f$ be a scalar-valued differentiable function defined on $\mathbb{R}$.
Using \eqref{dual_number} and the Taylor expansion,  
$f$ at the dual number $a+b\epsilon$ is expressed as 
\begin{align}\label{fd_dual}
f(a+b\epsilon)=f(a)+f^{\prime}(a) b\epsilon.
\end{align}
If we set $b = 1$ in \eqref{fd_dual}, the leading coefficient of $\epsilon$ gives the derivative of $f$ at $a$. 
For example, suppose $f(x) = 5x + 3$ and we want to compute $f(4)$ and $f'(4)$ using the forward mode AD.
Evaluating $f$ at the dual number $4 + 1\epsilon$ gives not only $f(4)=23$ but also $f'(4) = 5$:
\begin{align*}
    f(4+1 \epsilon) =(5+0 \epsilon)(4+1\epsilon) + (3+0\epsilon) 
    =23 + 5 \epsilon.
\end{align*}
\eqref{fd_dual} can be easily extended for the composition of multiple functions by using the chain rule:
\begin{equation*} 
    \begin{split}
        f(g(a+b\epsilon)) &=f\big( g(a)+g^{\prime}(a) b\epsilon \big) \\
                    &=f(g(a))+f^{\prime}(g(a)) g^{\prime}(a) b\epsilon.
    \end{split}
\end{equation*}

The univariate forward mode AD is generalized to the multivariate one
by defining the dual vectors by $\bm{a} + \bm{b}\epsilon$. It can be checked that 
an argument similar to the one used above gives the directional derivative of $f$ at $\bm{a}$ along the vector $\bm{b}$:
\begin{align*}
    f(\bm{a}+\bm{b}\epsilon) = f(\bm{a}) + \nabla f(\bm{a})^\top \bm{b} \cdot \epsilon.
\end{align*}

Lastly, the forward mode AD is extended to the multivariate vector-valued functions,
resulting in the JVP.
Let $\bm{f}:\mathbb{R}^n \to \mathbb{R}^m$
and $\bm{J}_{\bm{f}}$ be the Jacobian of $\bm{f}$ given by
\[
\bm{J}_{\bm{f}}(\bx)=\left[\begin{array}{cccc}
\frac{\partial f_1}{\partial x_1}(\bx) & \frac{\partial f_1}{\partial x_2}(\bx) & \cdots & \frac{\partial f_1}{\partial x_n}(\bx) \\
\frac{\partial f_2}{\partial x_1}(\bx) & \frac{\partial f_2}{\partial x_2}(\bx) & \cdots & \frac{\partial f_2}{\partial x_n}(\bx) \\
\vdots & \vdots & \ddots & \vdots \\
\frac{\partial f_m}{\partial x_1}(\bx) & \frac{\partial f_m}{\partial x_2}(\bx) & \cdots & \frac{\partial f_m}{\partial x_n}(\bx)
\end{array}\right]=
\begin{bmatrix}
    \nabla f_1(\bx)^\top \\ \vdots \\ \nabla f_m(\bx)^\top
\end{bmatrix}
\in \mathbb{R}^{m \times n},
\]
where $\bx = (x_1,\dots,x_n)$ and $\bm{f}(\bx) = (f_1(\bx),\dots,f_m(\bx))$.
The forward mode AD or JVP serves as a mapping defined by 
\begin{align}\label{JVP_Def}
\begin{aligned}
\operatorname{JVP}[\bm{f}](\bx): \mathbb{R}^n & \rightarrow \mathbb{R}^m, \\
\bm{v} & \mapsto \bm{J}_{\bm{f}}(\bx) \bm{v}.
\end{aligned}
\end{align}
We note that while the forward mode AD provides the function evaluation $\bm{f}(\bm{x})$ as well, we omit it in the above expression for simplicity.
By applying the chain rule, the JVP for the composition of two functions $\bm{f} \circ \bm{g}$ is obtained as
\begin{align}\label{JVP_chain_rule}
\begin{aligned}
\operatorname{JVP}[\bm{f} \circ \bm{g}](\bx)(\bm{v}) & 
=\bm{J}_{\bm{f} \circ \bm{g}}(\bx)\bm{v} \\
& =(\bm{J}_{\bm{f}} \circ \underbrace{\bm{g}(\bx)}_{\bm{y}=\bm{g}(\bx)}) \bm{J}_{\bm{g}}(\bx) \bm{v} \\
& =\bm{J}_{\bm{f}}(\bm{y}) \bm{J}_{\bm{g}}(\bm{x})\bm{v} \\
& =\operatorname{JVP}[\bm{f}](\bm{y})(\operatorname{JVP}[\bm{g}](\bx)(\bm{v})).
\end{aligned}
\end{align}
Since the evaluation of $\bm{g}(\bx)$ and $\operatorname{JVP}[\bm{g}](\bx)(\bm{v})$ are performed simultaneously,
the chain rule can be effectively performed.

\begin{remark} \label{rmk:JVP-cost}
    From \eqref{JVP_Def}, it can be seen that
    a single evaluation of $\operatorname{JVP}[\bm{f}](\bx)(\bm{v})$ costs
    $\mathcal{O}(1) \times \text {cost of } \bm{f}(\bx)$,
    \emph{independently of} $n$;
    the constant is bounded by $1+1.5p$, where $p$ is the column rank of the seed \cite{griewank2008evaluating}.
    Assembling the full Jacobian requires $n$ such evaluations, one per input coordinate.
\end{remark}
\begin{remark}
     \eqref{JVP_chain_rule} indicates that the Jacobian is constructed column-wise, one column per seed vector, and therefore JVP becomes efficient if $m \gg n$.
\end{remark}

\textbf{Pedagogical example of JVP.}
To demonstrate how easily JVP can be implemented, we present a pedagogical example 
using the JAX framework \cite{jax2018github}.
Let $f: \mathbb{R}^n \rightarrow \mathbb{R}$ be defined by $f(\bx) = \frac{1}{2}\|\alpha \bx\|^2$
where $\alpha \in \mathbb{R}$.
The directional derivative of $f$ at $\bx \in \mathbb{R}^n$ along $\bm{v} \in \mathbb{R}^n$ is $\nabla_{\bm{v}} f(\bx) = \alpha^2 \bx^\top \bm{v}$.
Figure~\ref{code:JVP} shows 
a JAX code for implementing 
the forward mode AD or JVP
of $f$ at $\bx=(0, 4, 6)$ along $\bm{v}=(1, 1, 1)$.
\begin{figure}[!ht]
\centering
\begin{lstlisting}[style=jaxpy]
import jax
import jax.numpy as jnp
alpha = 2.0
x = jnp.array([0.0, 4.0, 6.0]) ## Evaluation point
v = jnp.array([1.0, 1.0, 1.0]) ## Direction for derivative
f = lambda z: jnp.sum((alpha * z)**2)/2 ## Objective function
f_x, df_v = jax.jvp(f, (x,), (v,)) ## Evaluate $f(\bx)$ and $\nabla_{\bm{v}}f(\bx)$
print(f"f(x): {f_x}, directional derivative: {df_v}")
## f(x): 104.0, directional derivative: 40.0
\end{lstlisting}
\caption{A JAX code for implementing JVP of $f(\bx)=2\|\bx\|^2$
at $\bx=(0,4,6)$ along $\bm{v}=(1, 1,1)$.}\label{code:JVP}
\end{figure}

\subsection{Reverse mode AD or vector-Jacobian product}

Reverse mode AD calculates the derivative of a function by first computing the function's value and then working backward from the output to the inputs, propagating the derivatives through the computation graph.
A single reverse sweep propagates sensitivities from one output back to every input, so this approach is particularly efficient for functions with many inputs and few outputs. This is the common case in many machine learning models, where a scalar loss is differentiated with respect to a high-dimensional parameter vector. The efficiency is paid for in memory. The intermediate values produced during the forward pass must be retained for use in the backward sweep.

For $\bm{f}:\mathbb{R}^n \to \mathbb{R}^m$,
let $\bm{J}_{\bm{f}}$ be the Jacobian of $\bm{f}$. Reverse mode AD is then defined as Vector-Jacobian Product (VJP) as follows:
\begin{align*}
\begin{aligned}
\operatorname{VJP}[\bm{f}](\bx): \mathbb{R}^m & \rightarrow \mathbb{R}^n, \\
\bm{w} & \mapsto \bm{J}_{\bm{f}}(\bx)^\top \bm{w}.
\end{aligned}
\end{align*}
Similar to JVP, 
it follows from the chain rule that 
the VJP of the composition of two functions
is readily expressed as follows:
\begin{align}\label{chain_VJP}
\begin{aligned}
\operatorname{VJP}[\bm{f} \circ \bm{g}](\bx)(\bm{w}) & =\bm{J}_{\bm{f} \circ \bm{g}}(\bx)^{\top} \bm{w} \\
& =\bm{J}_{\bm{g}}(\bx)^{\top}(\bm{J}_{\bm{f}} \circ \underbrace{\bm{g}(\bx)}_{\underline{\bm{y}=\bm{g}(\bx)}})^{\top} \bm{w} \\
& =\bm{J}_{\bm{g}}(\bx)^{\top} \bm{J}_{\bm{f}}(\bm{y})^{\top} \bm{w} \\
& =\operatorname{VJP}[\bm{g}](\bx)(\operatorname{VJP}[\bm{f}](\bm{y})(\bm{w})).
\end{aligned}
\end{align}
Equation~\eqref{chain_VJP} shows that reverse mode AD proceeds in two stages. The VJP must first evaluate $\bm{g}(\bx)$ and store it. This is the forward pass, represented by the underlined term. It then computes the gradients by following the chain rule.
A single VJP evaluation costs $\mathcal{O}(1) \times \text { cost of } \bm{f}(\bx)$, \emph{independently of} $m$, with the constant bounded by $1.5+2.5p$ for a seed of column rank $p$ \cite{griewank2008evaluating}; assembling the full Jacobian requires $m$ such evaluations. This also shows that the VJP constructs the Jacobian row-wise, one row per seed vector, and therefore VJP will be a more efficient approach if $n \gg m$.
Crucially, the two modes differ in their \emph{memory} profile rather than in their arithmetic cost per sweep. As \eqref{chain_VJP} makes explicit, the reverse mode must retain the intermediate quantity $\bm{y}=\bm{g}(\bx)$, and by induction every intermediate value along the computational graph, until the backward sweep consumes it. This asymmetry is quantified in section~\ref{sec:complexity} and is the source of the computational trade-off studied in section~\ref{sec:efficiency}.

\textbf{Pedagogical example of VJP.}
Here, we demonstrate a use case of VJP to compute the derivative of 
$f: \mathbb{R}^n \rightarrow \mathbb{R}$, defined as
\begin{align}\label{vjp_fn}
f(\bx)=\frac{1}{2}\|\bx\|_2^2.
\end{align}
The Jacobian of \eqref{vjp_fn} is 
$\textbf{J}_f(\bx) \in \mathbb{R}^{1 \times n}, ~\text{with}~\nabla f(x)=\textbf{J}_f(\bx)^{\top} 1$.
Therefore the gradient of $f$ at $\bx$ is  $\nabla f(\bx)=\bx \cdot 1$.
 Figure~\ref{code:VJP} shows a code snippet implementing the VJP for the function defined in \eqref{vjp_fn}. 
\begin{figure}[!ht]
\centering
\begin{lstlisting}[style=jaxpy]
import jax
import jax.numpy as jnp
f = lambda x: jnp.sum(x**2)/2 ## Objective function
x = jnp.array([0.0, 1.0, 2.0]) ## Evaluation point
f_x, pullback = jax.vjp(f, x) ## $f(\bx)$ and the closure $\bm{w} \mapsto \bm{J}_f(\bx)^\top \bm{w}$
grad_f, = pullback(1.0) ## Seed with $\bm{w}=1$ to recover the full gradient
print(f"f(x): {f_x}")
print(f"grad f(x): {grad_f}")
## f(x): 2.5
## grad f(x): [0. 1. 2.]
\end{lstlisting}
\caption{JAX code for evaluating the gradient of a quadratic function using VJP}\label{code:VJP}
\end{figure}

\subsection{Computational complexity of JVP and VJP} \label{sec:complexity}
We summarize the time and space complexities of the two modes, as they are the basis of the computational comparison presented in section~\ref{sec:efficiency}.
Let $\mathcal{O}(\bm{f})$ denote the operation count of evaluating $\bm{f}$,
let $\mathcal{M}(\bm{f})$ denote the number of floating-point values that must be held
to evaluate $\bm{f}$, and let $p$ be the column rank of the seed.

For the forward mode, the time complexity satisfies
$\mathcal{O}[\bm{f}(\bx), \bm{J}_{\bm{f}}\bm{v}] \le \omega\, \mathcal{O}(\bm{f})$
with $\omega \in [1+p, 1+1.5p]$,
and the space complexity satisfies
$\mathcal{M}[\bm{f}(\bx), \bm{J}_{\bm{f}}\bm{v}] \le (1+p)\mathcal{M}(\bm{f})$
\cite{griewank2008evaluating}.
Both are bounded by a constant multiple of the cost of evaluating $\bm{f}$ itself
(Remark~\ref{rmk:JVP-cost}),
and neither depends on the depth of the computational graph:
the dual number carried alongside each intermediate value is consumed as soon as it is produced.

For the reverse mode, the time complexity satisfies
$\mathcal{O}[\bm{f}(\bx), \bm{v}^\top \bm{J}_{\bm{f}}] \le \omega\, \mathcal{O}(\bm{f})$
with $\omega \in [1+2p, 1.5+2.5p]$ \cite{griewank2008evaluating},
which is again a constant multiple of $\mathcal{O}(\bm{f})$.
The space complexity, however, is fundamentally different.
Every intermediate value generated during the forward pass must be retained until the
backward sweep consumes it.
This storage is termed the sequential access memory (SAM) in \cite{griewank2008evaluating}, and it grows in proportion to the length of the computational graph,
\begin{equation} \label{eqn:SAM}
    \mathcal{M}_{\text{SAM}} \propto \ell \cdot \mathcal{M}(\bm{f}),
\end{equation}
where $\ell$ is the number of elemental operations recorded on the tape.
For a feed-forward network of depth $L$ one has $\ell \propto L$, while
for an objective defined through a $T$-step numerical integrator, as in
section~\ref{sec:efficiency}, one has $\ell \propto LT$.

The bound \eqref{eqn:SAM} is not the last word, because the tape may be traded
against recomputation. \emph{Checkpointing}, also called rematerialization,
stores only a subset of the intermediate states, namely the checkpoints, and
discards the rest, recovering each discarded segment by re-running the forward
computation from the nearest checkpoint when the backward sweep reaches it. For
a $T$-step integration with $K$ evenly spaced checkpoints the retained storage is
$\mathcal{O}(T/K + K)$, which is minimized at $K \sim \sqrt{T}$ and yields
$\mathcal{O}(\sqrt{T})$ memory for one additional forward evaluation. Recursive
schedules reduce the storage further, to $\mathcal{O}(\log T)$, at a
correspondingly larger recomputation cost \cite{griewank2008evaluating}. What
checkpointing does not do is remove the dependence on $\ell$ altogether: it
lowers the exponent, from $T$ to $\sqrt{T}$, but the memory still grows with the
length of the graph, whereas the forward mode's does not grow at all. For this
reason we take reverse mode \emph{with} checkpointing, rather than an
unmodified tape, as the baseline throughout section~\ref{sec:efficiency}. The
comparison against an unmodified tape would be considerably more favorable to
the forward mode and would overstate the case.

Checkpointing is also not uniformly effective. It recovers memory only when the
body of the loop holds substantial interior state that can be discarded and
recomputed. Where each step is a single inexpensive update, there is no interior
to discard and the boundary state must be retained regardless. Section~\ref{sec:efficiency}
reports both situations.

Two consequences are worth stating explicitly, as they delimit precisely when the forward mode is advantageous.
First, the advantage is one of \emph{memory}, not of arithmetic. Both modes cost a constant multiple of $\mathcal{O}(\bm{f})$ per sweep, and that constant is larger for a JVP-based gradient estimator once the cost of drawing the random direction is included.
Second, the forward mode is not free of memory cost.
Carrying the dual component doubles the storage associated with the primal trajectory, and the RFG additionally materializes a random vector $\bz \in \mathbb{R}^d$ at every iteration.
The forward mode is therefore cheaper precisely when the term $\ell \cdot \mathcal{M}(\bm{f})$ in \eqref{eqn:SAM} dominates the parameter storage, that is, when the computational graph is long relative to the size of the parameter vector.
This prediction is tested against measurements in section~\ref{sec:efficiency}.

\section{Forward mode AD-based gradients} \label{sec:RFG}
Forward mode AD provides an efficient way of computing 
the directional derivative of $f$ along a given vector. 
However, if $f$ is not differentiable at $\bx$, the forward mode AD is not applicable.
This occurs, for example, in the training of ReLU neural networks. 
If this is the case, one can approximate the directional derivative using, e.g., a forward difference.
More precisely, for a given vector $\bm{z} \in \mathbb{R}^d$ and a small positive number $h \ll 1$,
let us consider an approximation to the directional derivative by the forward difference:
\begin{equation*}
    \nabla_{\bm{z},h} f(\bm{x}) := \frac{f(\bm{x}+h\bm{z}) - f(\bm{x})}{h},
\end{equation*}
which converges, as $h \to 0$, to $\nabla_{\bm{z}} f(\bm{x}) = \bm{z}^\top \nabla f(\bm{x})$ 
assuming $\nabla f(\bm{x})$ exists.
For notational completeness, let $\nabla_{\bm{z},0} f(\bm{x}):=\nabla_{\bm{z}} f(\bm{x})$.

\begin{definition} \label{def:FWG}
    The forward mode AD-based gradient of $f$ at $\bx$ is defined by
    \begin{equation*}
    \nabla_{\bm{z},h}^\textup{FM} f(\bm{x}) := \nabla_{\bm{z},h} f(\bm{x}) \cdot \bm{z},
    \end{equation*}
    where $h \ge 0$ and $\bm{z}$ is a given vector. 
    If the vector $\bm{z}$ is chosen randomly from a probability distribution $\text{P}$, we call $\nabla_{\bm{z},h}^\textup{FM} f(\bm{x})$ the RFG of $f$ at $\bm{x}$ along $\bm{z}$.
\end{definition}

\begin{remark}
    A special case of RFG was proposed in \cite{baydin2022gradients}, where it was named the ``forward gradient'',
    which uses the standard normal distribution for $\text{P}$, i.e., 
    $\text{P} = \mathcal{N}(0, I)$.
\end{remark}

\begin{remark} \label{rmk:RFG-scaling}
    If the exact directional derivative is available,
    the RFG is equivariant under scaling of the random vector $\bm{z}$ by a positive constant. That is, let $h=0$ and $\sigma > 0$.
    Then,
    \begin{align*}
        \nabla_{\sigma\bm{z},0}^\text{FM} f(\bm{x}) = \sigma^2 \nabla_{\bm{z},0}^\text{FM} f(\bm{x}).
    \end{align*}
    This indicates that if the RFG is used in place of the gradient for optimization, rescaling the random vector by $\sigma$ is equivalent to multiplying the learning rate by $\sigma^2$.
    This implies that 
    the RFG with any mean zero Gaussian distribution is equivalent to the ``forward gradient'' of \cite{baydin2022gradients}
    as it is defined through the standard normal distribution, i.e., variance 1.
\end{remark}

\begin{remark}
    If $\bm{z}$ is set to one of the standard basis vectors $\pm\bm{e}_i$ (up to sign), 
    we have $\nabla_{\pm\bm{e}_i,0}^\text{FM} f(\bm{x}) = \partial_i f(\bm{x}) \bm{e}_i$.
\end{remark}

\subsection{RFG-based optimization algorithms}
We now consider a family of RFG-based optimization algorithms by replacing the standard gradients with the RFGs.
Let $f:\mathbb{R}^d \to \mathbb{R}$ be a real-valued function defined on $\mathbb{R}^d$.
We are concerned with the unconstrained minimization problem of 
\begin{equation*}
    \min_{\bx \in \mathbb{R}^d} f(\bx).
\end{equation*}
The first-order optimization method may be written as follows:
For an initial point $\bx^{(0)}$, the $(k+1)$th iterated solution is obtained 
according to 
\begin{align*}
    \bx^{(k+1)} = \bx^{(k)} + \Phi \left(\{\nabla f(\bx^{(j)}) : j=0,\dots,k\} \cup \{\bx^{(j)} : j=0,\dots,k\}\right),
\end{align*}
where $\Phi$ represents the direction of the update, which may depend on 
all or some of the previous gradients and points. 
For example, gradient descent is recovered if
\begin{align*}
    \Phi_{\text{GD}} = -\eta \nabla f(\bx^{(k)}),
\end{align*}
and Polyak's heavy ball method is obtained if
\begin{align*}
    \Phi_{\text{PHB}} = -\eta \nabla f(\bx^{(k)}) + \mu(\bx^{(k)} - \bx^{(k-1)}),
\end{align*}
where $\mu$ is the momentum factor.

By replacing $\nabla f$ with $\nabla_{\bz,h}^\text{FM} f$,
one obtains a family of RFG-based optimization algorithms, that is,
\begin{equation} \label{def:algorithm}
    \bx^{(k+1)} = \bx^{(k)} + \Phi \left(\{\nabla_{\bz_j,h}^\text{FM}f(\bx^{(j)}) : j=0,\dots,k\} \cup \{\bx^{(j)} : j=0,\dots,k\}\right),
\end{equation}
where $\bz_j$'s are independent random vectors.
A pseudo-algorithm of the RFG-based GD is shown in Algorithm~\ref{alg:RFWGs}.

\begin{remark} \label{rmk:RCD}
    If the distribution $\text{P}$ is the uniform distribution over $\{\pm\bm{e}_i\}_{i=1}^d$,
    the algorithm \eqref{def:algorithm} with $\Phi_{\text{GD}}$ recovers the randomized coordinate descent algorithm \cite{nesterov2012efficiency,wright2015coordinate}.
\end{remark}

\begin{algorithm}
\caption{RFG-based GD Algorithm}\label{alg:RFWGs}
\begin{algorithmic}
\Require $f:$ Objective function 
\Require $\eta:$ Learning rate
\Require $\bm{x}^{(0)}:$ Initial trainable parameters
\Require $\textup{P}:$ Probability distribution
\State $k \gets 0$
\While{$\bm{x}^{(0)}$ not converged}
    \State $k \gets k + 1$
    \State $\bm{z}_k \sim \textup{P}$ 
     \State $\langle \nabla f(\bm{x}^{(k)}), \bm{z}_k \rangle \gets \text{Forward-mode AD with }f(\cdot), \bm{x}^{(k)}, \text{ and } \bm{z}_k$ 
     \State $\bm{x}^{(k+1)} \gets \bm{x}^{(k)} - \eta \cdot \langle \nabla f(\bm{x}^{(k)}), \bm{z}_k \rangle \bm{z}_k$
\EndWhile
\end{algorithmic}
\end{algorithm}

\subsection{Second-moment analysis of the RFG}
Suppose that the probability distribution $\text{P}$ satisfies $\mathbb{E}[\bm{z}\bm{z}^\top] = I$ where $\bm{z} \sim \text{P}$.
Assuming $\nabla f(\bm{x})$ exists, it can be checked that 
\begin{align*}
    \mathbb{E}[\nabla_{\bm{z},0}^\text{FM} f(\bm{x})] = \nabla f(\bm{x}).
\end{align*}
Note that there are infinitely many probability distributions that satisfy the above property.
For example, if 
$\bm{z} = (z_1,\dots,z_d)^\top$ and $z_i$'s are i.i.d. from a probability distribution $\text{p}$
having zero mean and unit standard deviation, we have $\mathbb{E}[\bm{z}\bm{z}^\top] = I$.
It would be natural to ask which probability distribution to use for the RFG.

To answer the question, we first consider the case where the exact directional gradient is available.
For any differentiable objective function, the second moment of the RFG is as follows.


\begin{theorem} \label{thm:2nd-moment-general}
    Suppose that $f$ is continuously differentiable.
    Let $\bm{z}$ be a random vector.
    Then,
    \begin{equation} \label{eqn:h0-2ndmoment-estimate}
       \mathbb{E}\left[\|\nabla_{\bm{z},0}^\textup{FM} f(\bm{x}) -\nabla f(\bm{x})\|^2\right]
       \le \|\mathbb{E}[(I-\bm{z}\bm{z}^\top)^2] \|\cdot \|\nabla f(\bm{x})\|^2.
    \end{equation}
\end{theorem}
\begin{proof}
    Since $\|\nabla_{\bm{z},0}^\textup{FM} f(\bm{x}) -\nabla f(\bm{x})\|^2
        = \|(I-\bm{z}\bm{z}^\top)\nabla f(\bm{x})\|^2$,
    we may expand the square as
    \begin{align*}
        \|(I-\bm{z}\bm{z}^\top)\nabla f(\bm{x})\|^2
        = (\nabla f(\bm{x}))^\top (I-\bm{z}\bm{z}^\top)^2 \nabla f(\bm{x}),
    \end{align*}
    using that $I-\bm{z}\bm{z}^\top$ is symmetric.
    Taking expectations and using linearity gives
    \begin{align*}
        \mathbb{E}\left[\|\nabla_{\bm{z},0}^\textup{FM} f(\bm{x}) -\nabla f(\bm{x})\|^2\right]
        = (\nabla f(\bm{x}))^\top \mathbb{E}[(I-\bm{z}\bm{z}^\top)^2] \nabla f(\bm{x}).
    \end{align*}
    Since $\mathbb{E}[(I-\bm{z}\bm{z}^\top)^2]$ is symmetric positive semi-definite,
    the Rayleigh quotient bound $\bm{u}^\top B\bm{u} \le \|B\|\,\|\bm{u}\|^2$
    applied with $B=\mathbb{E}[(I-\bm{z}\bm{z}^\top)^2]$ and $\bm{u}=\nabla f(\bm{x})$
    yields \eqref{eqn:h0-2ndmoment-estimate}.
\end{proof}

If the directional gradient is approximated by the forward difference, the error caused by a small step size $h$ should be taken into account. 
By focusing on a classical convex and strongly smooth function,
we obtain the following result.
\begin{theorem} \label{thm:2nd-moment-RFG}
    Suppose $f$ is continuously differentiable.
    Also, suppose that $f$ is convex,
    and $L$-strongly smooth, i.e.,
    \begin{equation*}
         \nabla f(\bm{x})^\top (\bm{y} - \bm{x}) 
         \le f(\bm{y}) - f(\bm{x}) 
         \le \nabla f(\bm{x})^\top (\bm{y} - \bm{x}) + \frac{L}{2}\|\bm{y}-\bm{x}\|^2,
    \end{equation*}
    for any $\bm{x},\bm{y} \in \mathbb{R}^d$.
    Let $\bm{z}$ be a random vector.
    Then, for any $h \ge 0$,
    \begin{align*}
        \mathbb{E}\left[ \|\nabla_{\bm{z},h}^\textup{FM} f(\bm{x}) - \nabla f(\bm{x})\|^2\right] 
        &\le 
        \frac{h^2}{4}L^2\cdot \mathbb{E}[\|\bm{z}\|^6]
        +\|\mathbb{E}[(I-\bm{z}\bm{z}^\top)^2] \|\cdot \|\nabla f(\bm{x})\|^2
        \\
        &\qquad+hL \mathbb{E}\left[\|\bm{z}\|^3\max\{1,|\|\bm{z}\|^2-1|\}
        \right]\cdot \|\nabla f(\bm{x})\|,
    \end{align*}
    where 
    $\|\cdot\|$ is the Euclidean norm.
\end{theorem}
\begin{proof}
    The proof can be found in Appendix~\ref{app:thm:2nd-moment-RFG}.
\end{proof}

If $h=0$, Theorem~\ref{thm:2nd-moment-RFG} recovers the result of
Theorem~\ref{thm:2nd-moment-general}, in which case the inequality holds with equality.

\begin{remark} \label{rmk:roundoff}
    Theorem~\ref{thm:2nd-moment-RFG} is a statement about exact arithmetic.
    In finite precision the difference quotient
    $\left(f(\bm{x}+h\bm{z}) - f(\bm{x})\right)/h$
    additionally carries a rounding error of order $\epsilon_{\text{mach}}/h$,
    which the bound does not model, so that decreasing $h$ does not reduce the
    error indefinitely.
    We note, however, that the practical significance of this term depends on
    the working precision and on the magnitude of $f$ near the minimizer.
    In the experiments of section~\ref{sec:efficiency}, carried out in double
    precision on an objective whose value tends to zero at the solution, the
    two evaluations being differenced are themselves small, their rounding
    error scales with their own magnitude, and no degradation was observed for
    $h$ as small as $10^{-12}$.
    In single precision the effect is visible.
\end{remark}

\begin{remark}
    Let $\bm{z}$ be a random vector whose components are i.i.d. from a probability distribution $\text{p}$ whose first and third moments are zeros and whose second moment is finite and denoted by $\sigma^2$, and whose fourth standardized moment is denoted by $\kappa_4$. 
    Then, the estimate \eqref{eqn:h0-2ndmoment-estimate} holds with equality, yielding 
    \begin{align*}
       \mathbb{E}\left[\|\nabla_{\bm{z},0}^\textup{FM} f(\bm{x}) -\nabla f(\bm{x})\|^2\right]
       = \bigg((d+\kappa_4-1)\sigma^4 -2\sigma^2 +1 \bigg) \|\nabla f(\bm{x})\|^2.
    \end{align*}
    In particular, the constant term is minimized when $\sigma^2 = \frac{1}{d+\kappa_4-1}$, which gives the expected squared error of
    $\mathbb{E}\left[\|\nabla_{\bm{z},0}^\textup{FM} f(\bm{x}) -\nabla f(\bm{x})\|^2\right] = (1-\frac{1}{d+\kappa_4-1})\|\nabla f(\bm{x})\|^2$.
\end{remark}

\begin{remark}
    If $\bm{z}$ is uniformly sampled from $\{\pm \alpha\bm{e}_i\}_{i=1}^d$,
    since $\mathbb{E}[(I-\bm{z}\bm{z}^\top)^2]=(1+\frac{\alpha^2(\alpha^2-2)}{d})I$, 
    the estimate \eqref{eqn:h0-2ndmoment-estimate} holds with equality, giving 
    \begin{align*}
       \mathbb{E}\left[\|\nabla_{\bm{z},0}^\textup{FM} f(\bm{x}) -\nabla f(\bm{x})\|^2\right]
       = \bigg(1-\frac{1}{d} +\frac{(\alpha^2-1)^2}{d}\bigg) \|\nabla f(\bm{x})\|^2.
    \end{align*}
    In particular, the constant term is minimized when $\alpha^2=1$ and the expected squared error is 
    $\mathbb{E}\left[\|\nabla_{\bm{z},0}^\textup{FM} f(\bm{x}) -\nabla f(\bm{x})\|^2\right] =(1-\frac{1}{d})\|\nabla f(\bm{x})\|^2$.
\end{remark}

Both cases imply that 
to minimize the relative squared error of $\nabla_{\bm{z},0}^\text{FM} f$,
a biased estimate for the gradient of $f$ should be used
as $\mathbb{E}[\nabla_{\bm{z},0}^\text{FM} f(\bx)] = \frac{1}{d+\kappa_4-1} \nabla f(\bx)$ (former)
and $\mathbb{E}[\nabla_{\bm{z},0}^\text{FM} f(\bx)] = \frac{1}{d} \nabla f(\bx)$ (latter).

We emphasize that this bias is of no consequence for the optimization algorithms studied in the sequel.
In both cases the estimator is an unbiased estimate of a fixed positive multiple of $\nabla f(\bx)$, and by the scaling relation of Remark~\ref{rmk:RFG-scaling} the multiplicative constant may be absorbed into the learning rate $\eta$ without altering the iterates.
Consequently the quantity that governs the behavior of RFG-based algorithms is not the bias but the relative variance, and it is precisely this quantity that is minimized by choosing the distribution of least kurtosis.
The optimal learning rate \eqref{RFG-GD-opt-lr} of Theorem~\ref{thm:RFG-GD-convg} carries exactly the compensating factor $\beta^2/\gamma_d$, which makes this absorption explicit.

\section{Convergence analysis for quadratic functions} \label{sec:analysis}
In this section, by focusing on quadratic objective functions, 
we present a convergence analysis of RFG-based optimization algorithms.
Specifically, we focus on the GD and PHB methods.

For a matrix $A \in \mathbb{R}^{m\times d}$
and a vector $b \in \mathbb{R}^m$,
let us consider
\begin{equation} \label{def:quadratic}
    \min_{\bx} f(\bx) \quad \text{where} \quad f(\bm{x}) = \frac{1}{2}\|A\bm{x}-b\|^2,
\end{equation}
where $\|\cdot\|$ is the Euclidean norm.
Note that $\nabla f(\bm{x}) = A^\top(A\bm{x}-b)$
and the minimizer of $f$ is explicitly expressed as $\bm{x}^*=(A^\top A)^{-1}A^\top b$
assuming the invertibility of $A^\top A$.

Since random vectors are used for the computation of the RFGs, we introduce a class of probability distributions satisfying the following assumptions.
\begin{assumption} \label{assumption:prob-dist}
    Let $\bm{z}$ be a random vector whose components are i.i.d. from a probability distribution $\emph{\text{p}}$ whose first, third, and fifth moments are zeros, whose second moment is denoted by $\sigma^2$,
    and whose sixth moment is finite.
    Let $Z \sim \emph{\text{p}}$.
    The standardized $k$th moment of $Z$ is denoted by
    \begin{equation*}
        \kappa_k := \frac{\mathbb{E}[(Z-\mu)^k]}{\sigma^k},
    \end{equation*}
    where $\mu = \mathbb{E}[Z]$ and $\sigma^2 = \mathbb{E}[(Z-\mu)^2]$.
\end{assumption}

There are many probability distributions that satisfy Assumption~\ref{assumption:prob-dist}.
In Table~\ref{table:prob-dist}, we present some well-known distributions along with their kurtosis $\kappa_4$
and $\kappa_6$.
We remark that while the mean and variance of a probability distribution can be altered by shifting or scaling, the standardized moments remain unchanged,
which may be viewed as an intrinsic property of the distribution.

\begin{table}[htp!] 
\begin{center}
    \begin{tabular}{l|l|l|l} 
          & pmf or pdf                                                       & Kurtosis ($\kappa_4$) & $\kappa_6$     \\ \hline
Bernoulli & $P(z=-r)=P(z=r)=0.5$                                            & $1$      & $1$            \\
Uniform   & $f(z) = \frac{1}{2r}\mathbb{I}_{[-r,r]}(z)$                      & $1.8$    & $27/7$ \\
Wigner    & $f(z) = \frac{2}{\pi r^2}\sqrt{r^2-z^2} \mathbb{I}_{[-r,r]}(z)$  & $2$      & $5$            \\
Gaussian  & $f(z) = \frac{1}{\sqrt{2\pi \sigma^2}}e^{-\frac{z^2}{2\sigma^2}}$ & $3$      & $15$           \\
Laplace   & $f(z) = \frac{1}{2b}e^{-|z|/b}$                                  & $6$      & $90$
\end{tabular}
\end{center} 
\caption{The list of some probability distributions that satisfy Assumption~\ref{assumption:prob-dist}.}
\label{table:prob-dist}
\end{table}

\begin{proposition} \label{prop:RFGquadratic}
    Let $f$ be the quadratic function \eqref{def:quadratic}.
    Suppose $\bm{z}$ is a random vector such that 
    $\mathbb{E}[\bz\bz^\top] = \beta^2 I$ for some $\beta > 0$, $\mathbb{E}[\|M\bm{z}\|^2\bm{z}] = 0$ for any matrix $M$, and the following quantities are well-defined: 
    For any matrix $M$ and vector $c$, 
    \begin{align*}
        \mathbb{E}[|\bz^\top c|^2\bz\bz^\top], \quad  
        \mathbb{E}[\|M\bz\|^4\bz\bz^\top], \quad 
        \mathbb{E}[\|M\bz\|^2(\bz^\top c) \bz\bz^\top].
    \end{align*}
    Then, the covariance matrix of $\nabla_{\bm{z},h}^{\emph{\text{FM}}} f(\bm{x})$ is given by
    \begin{equation*}
        \begin{split}
            &\emph{\text{Cov}}[\nabla_{\bm{z},h}^{\emph{\text{FM}}} f(\bm{x})] \\
            &=
        \mathbb{E}\left[(|\bm{z}^\top \nabla f(\bx)|^2 + h\|A\bm{z}\|^2\bm{z}^\top \nabla f(\bx) + \frac{h^2}{4}\|A\bm{z}\|^4) \bm{z}\bm{z}^\top\right] - \beta^4 \nabla f(\bx) (\nabla f(\bx))^\top. 
        \end{split}
    \end{equation*}
\end{proposition}
\begin{proof}
    It can be checked that 
    \begin{align*}
        \nabla_{\bm{z},h}^\text{FM} f(\bm{x}) = \bm{z}\bm{z}^\top A^\top (A\bm{x}-b) + \frac{h}{2}\|A\bm{z}\|^2\bm{z}.
    \end{align*}
    It follows from the assumptions on $\bm{z}$ that $\mathbb{E}[\|A\bm{z}\|^2\bm{z}] = 0$ and $\mathbb{E}[\bm{z}\bm{z}^\top] = \beta^2 I$.
    Since $\nabla f(\bm{x}) = A^\top (A\bm{x}-b)$ and $\mathbb{E}[\nabla_{\bm{z},h}^\text{FM} f(\bm{x})] = \mu := \beta^2 \nabla f(\bx)$, we have
    \begin{align*}
        \text{Cov}[\nabla_{\bm{z},h}^\text{FM} f(\bm{x})] 
        &=\mathbb{E}\left[\nabla_{\bm{z},h}^\text{FM} f(\bm{x})(\nabla_{\bm{z},h}^\text{FM} f(\bm{x}))^\top \right]- \mu \mu^\top 
        \\
        &=
        \mathbb{E}\left[(|\bm{z}^\top \nabla f(\bx)|^2 + h\|A\bm{z}\|^2\bm{z}^\top \nabla f(\bx) + \frac{h^2}{4}\|A\bm{z}\|^4) \bm{z}\bm{z}^\top\right] - \mu \mu^\top,
    \end{align*}
    which completes the proof.
\end{proof}

\begin{remark}
    If $\bm{z}$ satisfies Assumption~\ref{assumption:prob-dist}, it follows from Lemma~\ref{app:lemma-equality} that $\beta^2 = \sigma^2$, 
    \begin{align*}
        \mathbb{E}[|\bz^\top \nabla f(\bx)|^2\bz\bz^\top] &= \sigma^4 \Big[
        \text{diag}\big((\kappa_4-3)(\partial_{x_i}f(\bx))^2 + \|\nabla f(\bx)\|^2\big)
        + 2\nabla f(\bx)(\nabla f(\bx))^\top\Big], \\  
        \mathbb{E}[\|A\bz\|^2(\bz^\top c)\bz\bz^\top] &= 0, \\
        \mathbb{E}[\|A\bz\|^4\bz\bz^\top] &=: \Gamma, \qquad
        \text{tr}\,\Gamma = \sigma^6\mathcal{F}(d,\kappa_4,\kappa_6,A),
    \end{align*}
    where $\kappa_q$ is the $q$th standardized moment of $\emph{\text{p}}$,
    $\Gamma$ is symmetric positive semidefinite, and
    \begin{align*}
    \mathcal{F}(d,\kappa_4,\kappa_6,A)=
        \sum_{k,l}
        \left[\alpha_d \sum_{i=1}^d A_{k,i}^2A_{l,i}^2
        + 
        \beta_d
        \sum_{i\ne j} (A_{k,i}^2A_{l,j}^2 + 2A_{k,i}A_{l,i}A_{k,j}A_{l,j})
        \right],
    \end{align*}
    with $\alpha_d= \kappa_6 + (d-1)\kappa_4$ and 
    $\beta_d = d+2(\kappa_4-1)$, so that
    $\sigma^6\mathcal{F}(d,\kappa_4,\kappa_6,A) = \mathbb{E}[\|A\bz\|^4\|\bz\|^2]$
    by Lemma~\ref{app:lemma-equality}. We have
    \begin{align*}
        \text{Cov}[\nabla_{\bm{z},h}^\text{FM} f(\bm{x})]
        &=\sigma^4 \Big[\text{diag}\big((\kappa_4-3)(\partial_{x_i}f(\bx))^2 + \|\nabla f(\bx)\|^2\big)
        + \nabla f(\bx)(\nabla f(\bx))^\top \Big]
        \\
        &\quad + \frac{h^2}{4}\Gamma.
    \end{align*}
    Neither $\mathbb{E}[|\bz^\top \nabla f(\bx)|^2\bz\bz^\top]$ nor $\Gamma$ is
    diagonal: the off-diagonal entries of the former are
    $2\sigma^4\partial_{x_i}f(\bx)\,\partial_{x_j}f(\bx)$. Only the trace of
    $\Gamma$ enters the convergence analysis, and it is supplied by
    Lemma~\ref{app:lemma-equality}.
    In particular, it can be seen that if $h=0$, 
the variance of the $i$th component is
    \begin{align*}
    \text{Var}[(\nabla_{\bm{z},0}^{{\text{FM}}} f(\bm{x}))_i]
    = \sigma^4\Big( (\kappa_4-2)(\partial_{x_i} f(\bx))^2 + \|\nabla f(\bx)\|^2 \Big).
    \end{align*}
In particular, if the components of $\nabla f(\bx)$ are of comparable magnitude,
so that $\|\nabla f(\bx)\|^2 \approx d\,(\partial_{x_i} f(\bx))^2$, this reduces to
$\sigma^4 (d + \kappa_4 - 2)(\partial_{x_i} f(\bx))^2$, and the variance of each
component grows proportionally with $d-1+(\kappa_4-1)$.
This again suggests one should utilize the probability distribution with the smallest kurtosis for the smallest variance of the RFG.

The rank-one term $2\nabla f(\bx)(\nabla f(\bx))^\top$ is not incidental, and
it is worth saying what it means. At $h=0$,
\begin{align*}
    \langle \nabla_{\bz,0}^\text{FM} f(\bx), \nabla f(\bx)\rangle
    = \langle (\bz^\top\nabla f(\bx))\bz, \nabla f(\bx)\rangle
    = |\bz^\top \nabla f(\bx)|^2 \ge 0,
\end{align*}
so the RFG is a descent direction on \emph{every} draw, not merely in
expectation. The rank-one term is the second-moment signature of that fact, and
a diagonal covariance would be inconsistent with it. Decomposing the error
covariance into the gradient direction and its orthogonal complement, the
variance along $\nabla f(\bx)$ is $2\sigma^4\|\nabla f(\bx)\|^2$ up to a term of
order $1/d$, while the perpendicular variance is of order $d\sigma^4\|\nabla
f(\bx)\|^2$. The off-diagonal correction is therefore $\mathcal{O}(1)$ against an
$\mathcal{O}(d)$ background and does not affect the dimension penalty of
Theorem~\ref{thm:RFG-GD-convg}. What it does say is that the excess variance sits
in the one direction along which error costs least. We caution against the
natural inference that this component is removable by a step-size rule: the
exact step along the RFG direction for the quadratic \eqref{def:quadratic} is
$1/\|A\bz\|^2$, in which the factor $|\bz^\top\nabla f(\bx)|^2$ cancels, so no
line search can reach it. It is the luck of the draw rather than a mismatch that
may be tuned away.

For $h>0$ the descent property is lost, since
$\langle \nabla_{\bz,h}^\text{FM} f(\bx), \nabla f(\bx)\rangle
= |\bz^\top\nabla f(\bx)|^2 + \frac{h}{2}\|A\bz\|^2(\bz^\top\nabla f(\bx))$
may be negative when $\bz^\top \nabla f(\bx)<0$. This is a further reason to
prefer exact forward mode AD to a finite difference, beyond the accuracy
considerations of Remark~\ref{rmk:roundoff}.
\end{remark}

\begin{remark}
    If $\bm{z}$ is uniformly sampled from $\{\pm \alpha\bm{e}_i\}_{i=1}^d$,
    it follows from 
    \begin{align*}
        \mathbb{E}[|\bz^\top \nabla f(\bx)|^2\bz\bz^\top] &= \frac{\alpha^4}{d} \cdot \text{diag}\big((\partial_{x_i}f(\bx))^2\big), \\  
        \mathbb{E}[\|A\bz\|^2(\bz^\top c)\bz\bz^\top] &= 0, \\
        \mathbb{E}[\|A\bz\|^4\bz\bz^\top] &= \frac{\alpha^6}{d} \cdot \text{diag}\big(\|A\bm{e}_i\|^4\big),
    \end{align*}
    that we have 
    \begin{align*}
        \text{Cov}[\nabla_{\bm{z},h}^\text{FM} f(\bm{x})]
        &=\frac{\alpha^4}{d}\text{diag}\big((\partial_{x_i}f(\bx))^2\big)
        + \frac{h^2\alpha^6}{4d}\text{diag}\big(\|A\bm{e}_i\|^4\big)
        - \frac{\alpha^4}{d^2} \nabla f(\bx) (\nabla f(\bx))^\top.
    \end{align*}
    In addition, if $h=0$, the normalized variance is  
    \begin{align*}
    \frac{\text{Var}[(\nabla_{\bm{z},0}^{{\text{FM}}} f(\bm{x}))_i]}{\frac{\alpha^4}{d^2}(\partial_{x_i} f(\bx))^2} = d-1,
    \end{align*}
    which indicates that the variance of each component of $\nabla_{\bm{z},0}^{{\text{FM}}} f(\bm{x})$ grows proportionally with $d-1$.
\end{remark}

\begin{remark} \label{rmk:norm-estimate}
The dimension penalty afflicts the \emph{direction} of the RFG but not its
magnitude, which is worth recording because an RFG-based method never observes
$\nabla f(\bx)$ and may need $\|\nabla f(\bx)\|$ for a stopping rule. By the
first identity of Lemma~\ref{app:lemma-equality},
$\mathbb{E}[\|\nabla_{\bz,0}^\textup{FM} f(\bx)\|^2]
= \sigma^4(\kappa_4+d-1)\|\nabla f(\bx)\|^2$, so for independent
$\bz_1,\dots,\bz_k$ the statistic
\begin{align*}
    \widehat{\|\nabla f(\bx)\|^2}
    = \frac{1}{k\,\sigma^4(\kappa_4+d-1)}
      \sum_{j=1}^k \|\nabla_{\bz_j,0}^\textup{FM} f(\bx)\|^2
\end{align*}
is unbiased, and its relative variance does not grow with $d$. The obvious
alternative is not: the squared norm of the $k$-sample average satisfies
\begin{align*}
\mathbb{E}\Big\|\tfrac{1}{k}\textstyle\sum_{j} \nabla_{\bz_j,0}^\textup{FM} f(\bx)\Big\|^2
= \sigma^4\|\nabla f(\bx)\|^2\Big(1 + \tfrac{d+\kappa_4-2}{k}\Big),
\end{align*}
inflated by a factor that grows with $d/k$. At $d=256$ and $k=8$ the latter
overstates $\|\nabla f(\bx)\|$ by a factor of $5.6$, while the former returns
$0.976$ of it. The estimator costs nothing, since
$\|\nabla_{\bz_j,0}^\textup{FM} f(\bx)\|^2$ is already formed. Its relative
standard deviation is approximately $\sqrt{2/k}$, inherited from the $\chi_1^2$
factor $|\bz^\top\nabla f(\bx)|^2$, so at small $k$ it is centered rather than
precise. The point is that its error is noise rather than a dimension-dependent
bias.
\end{remark}

\subsection{RFG-based gradient descent}
We now derive the rate of convergence of the RFG-based GD method, together with the learning rate that optimizes it.
That is,
for an initial point $\bx^{(0)}$, 
the $(k+1)$th iterated solution to the RFG-based GD method is obtained according to
\begin{equation} \label{def:RFG-GD}
    \bx^{(k+1)} = \bx^{(k)} - \eta \nabla_{\bz_k,h}^\text{FM} f(\bx^{(k)}),
\end{equation}
where $\bz_k$'s are independent random vectors
and $\eta$ is the learning rate.

\begin{theorem} \label{thm:RFG-GD-convg}
    Let $f$ be the quadratic objective function 
    and let $\bm{x}^*$ be the optimal solution to 
    \eqref{def:quadratic}.
    Let $\bz$ be the random vector satisfying 
    $\mathbb{E}[\bz\bz^\top] = \beta^2 I$, 
    $\mathbb{E}[\|\bz\|^2\bz\bz^\top] = \gamma_dI$ and 
    \begin{align*}
    \mathbb{E}[\|A\bz\|^2 \langle \bz, c \rangle] &= 0, \quad
    \mathbb{E}[\|A\bz\|^2 \|\bz\|^2 \langle \bz, c \rangle] = 0 \quad  \forall c \in \mathbb{R}^d.
    \end{align*}
    Let $\bx^{(k)}$ 
    be the $k$th iterated solution to
    the RFG-based GD method \eqref{def:RFG-GD}
    with the constant learning rate of 
    \begin{equation} \label{RFG-GD-opt-lr}
        \eta = \frac{\beta^2}{\gamma_d}\cdot \frac{2}{\lambda_{\max}+\lambda_{\min}},
    \end{equation}
    where $\lambda_{\min}$ and $\lambda_{\max}$ are the smallest and largest eigenvalues of $A^\top A$, respectively.
    Then, 
    \begin{align*}
    \mathbb{E}[\|\bx^{(k)} - \bx^*\|^2]
    \le 
    r_\emph{\text{rate}}^k
    \|\bx^{(0)} - \bx^*\|^2
    + 
    \frac{h^2}{4}\frac{\mathbb{E}[\|A\bz\|^4\|\bz\|^2]}{(\frac{\lambda_{\max}+\lambda_{\min}}{2})^2\gamma_d\left[1-(\frac{\kappa_A-1}{\kappa_A+1})^2\right]},
    \end{align*}
    with $r_\text{rate}=1 - \frac{\beta^4}{\gamma_d}\left[1-(\frac{\kappa_A - 1}{\kappa_A +1})^2\right]$, where the expectation is taken over all random vectors and $\kappa_A$ is the condition number of $A^\top A$.
\end{theorem}
\begin{proof}
    The proof can be found in Appendix~\ref{app:thm:RFG-GD-convg}.
\end{proof}

\begin{remark} \label{rmk:GDrate}
    If $\bz$ satisfies Assumption~\ref{assumption:prob-dist}, we have
    $\frac{\beta^4}{\gamma_d} = \frac{1}{\kappa_4+d-1}$.
    If $\bz$ follows the uniform distribution on $\{\pm\alpha\bm{e}_i\}_{i=1}^d$,
    we have
    $\frac{\beta^4}{\gamma_d} = \frac{1}{d}$, which is identical to the former case with $\kappa_4=1$.
    In particular, Theorem~\ref{thm:RFG-GD-convg} indicates that to achieve the optimal rate of convergence, one needs to use the probability distribution having the smallest kurtosis $\kappa_4$ if $\bz$ follows Assumption~\ref{assumption:prob-dist}.
    As shown in Table~\ref{table:prob-dist},
    the Bernoulli distribution achieves the smallest kurtosis of 1.
\end{remark}

\subsection{RFG-based Polyak's heavy ball method}
Let us consider the RFG-based PHB method.
That is, starting from the two initial points $\bx^{(0)},\bx^{(-1)}$,
the $(k+1)$th iterated solution of the RFG-based PHB method is obtained according to 
\begin{equation} \label{def:RFG-PHB}
    \bx^{(k+1)} = \bx^{(k)}- \eta \nabla_{\bz_k,h}^\text{FM} f(\bx^{(k)}) + \mu (\bx^{(k)} - \bx^{(k-1)}).
\end{equation}

For a full rank matrix $A$ of size $m\times d$,
let
$A^\top A =  U_A\Sigma_A U_A^\top$
be a spectral decomposition
where $U_A \in \mathbb{R}^{d\times d}$ is an orthogonal matrix
and $\Sigma_A = \text{diag}(\lambda_i)$ is diagonal whose entries are the square of singular values of $A$.
Let
$\bm{U} = \begin{bmatrix}
        U_A & \\ & U_A
\end{bmatrix} \in \mathbb{R}^{2d \times 2d}$.
Assuming $\mathbb{E}[\bz\bz^\top]=\beta^2I$,
let us define a mapping $\Phi:\text{Sym}_{2d} \to \text{Sym}_{2d}$ by
\begin{equation} \label{def:PBH-mapping}
    \Phi(\bm{S}) = \begin{bmatrix}
        H_1 & H_2^\top \\
        H_2 & H_3
\end{bmatrix}, \qquad \forall \bm{S} = \begin{bmatrix}
    S_1 & S_2^\top \\ S_2 & S_3
\end{bmatrix}\in \text{Sym}_{2d},
\end{equation}
where $\text{Sym}_n$ represents the set of all symmetric real $n\times n$ matrices.
$H_i$'s are defined as functions of $\bm{S}$ by
\begin{align*}
    H_1 &:= (1+\mu)^2 S_1 - \eta(1+\mu)\beta^2\left(S_1\Sigma_A + \Sigma_A S_1\right)
        \\
        &\qquad + \eta^2
        \Sigma_A U_A^\top 
        H_4U_A \Sigma_A + S_2V + (S_2V)^\top + S_3, \\
    H_2 &:= -\mu (S_1V + S_2^\top), \\
    H_3 &:= \mu^2 S_1, \\
    H_4 &:= 
        \mathbb{E}[
        (\bz^\top U_AS_1U_A^\top \bz)\bz
        \bz^\top ],
\end{align*}
where 
$V:= (1+\mu)I - \eta \beta^2 \Sigma_A$.
Note that the first two terms of $H_1$ may be written compactly as
$\tfrac{1}{2}\big[(V^\top V - \eta^2\beta^4\Sigma_A^2)S_1 + S_1(V^\top V - \eta^2\beta^4\Sigma_A^2)\big]$;
the symmetrization is needed because $S_1$ and $\Sigma_A$ need not commute,
and it is what guarantees $\Phi$ maps $\text{Sym}_{2d}$ into itself.

In particular, if $\bz$ satisfies Assumption~\ref{assumption:prob-dist}, we have
$\mathbb{E}[\bz\bz^\top]=\sigma^2I$ and 
$H_4=\sigma^4 T$ with
\begin{align*}
    T_{ij} =\begin{cases}
        \text{tr}(S_1)+ 
        (\kappa_4 - 1)(U_AS_1U_A^\top)_{ii} & \text{if } i =j, \\
        (U_AS_1U_A^\top)_{ij} + (U_AS_1U_A^\top)_{ji}& \text{if } i \ne j
    \end{cases}.
\end{align*}
If $\bz$ follows the uniform distribution on $\{\pm\alpha \bm{e}_j\}_{j=1}^d$, we have
$\mathbb{E}[\bz\bz^\top]=\frac{\alpha^2}{d}I$ and 
$H_4=\frac{\alpha^4}{d} \text{diag}((U_AS_1U_A^\top)_{ii})$.

We are now in a position to present 
the error analysis of the RFG-based PHB method in terms of the mapping defined by \eqref{def:PBH-mapping}.
\begin{theorem} \label{thm:RFG-PHB-convg}
    Let $f$ be the quadratic objective function 
    and let $\bm{x}^*$ be the optimal solution to 
    \eqref{def:quadratic}.
    Let
    $\bx^{(k)}$ 
    be the $k$th iterated solution to
    the RFG-based PHB method \eqref{def:RFG-PHB}, 
    let $\mathcal{E}_k = [
    \bx^{(k)} - \bx^*;\bx^{(k-1)} - \bx^*
    ] \in \mathbb{R}^{2d}$.
    Suppose $\bz$ saitsifes 
    \begin{align*}
        \mathbb{E}[\bz\bz^\top] = \beta^2I, \quad 
        \mathbb{E}[\|A\bz\|^2 \langle \overline{\bm{M}}(\bz)\bm{c}, \bm{U}^\top \bm{Z}\rangle] = 0, \quad \forall \bm{c} \in \mathbb{R}^{2d},
    \end{align*}
    where $\bm{Z}=[\bz; 0] \in \mathbb{R}^{2d}$ and $\overline{\bm{M}}$ is defined in \eqref{def:Moverline}.
    Then, 
    \begin{equation} \label{eqn:full-recursion}
        \mathbb{E}[\|\mathcal{E}_{k}\|^2]
        = 
        \|\bm{U}^\top\mathcal{E}_{0}\|^2_{\Phi^k}
        + \frac{\eta^2h^2}{4}\sum_{i=0}^{k-1} \mathbb{E}[\|A\bz\|^4\|\bm{U}^\top \bm{Z}\|^2_{\Phi^i}],
    \end{equation}
    where 
    $\Phi^k$ is the $k$-fold composition of $\Phi$,
    and 
    $\|\bm{U}^\top\bm{Z}\|^2_{\Phi^i} = (\bm{U}^\top\bm{Z})^\top \Phi^i(I_{2d}) \bm{U}^\top\bm{Z}$.
    Here the expectation is taken over all the random vectors.
\end{theorem}
\begin{proof}
    The proof can be found in Appendix~\ref{app:thm:RFG-PHB-convg}.
\end{proof}

Theorem~\ref{thm:RFG-PHB-convg} shows that 
the rate of convergence for the RFG-based PHB method is determined by
the eigenvalues of $\Phi^k(I_{2d})$
as long as the second term in the right-hand side of \eqref{eqn:full-recursion}
remains negligible.
This is typically the case as $\eta$ and $h$ are chosen sufficiently small in practice.
While an explicit expression of the rate of convergence may be out of reach for general $\mu, \eta$ and probability distributions,
we attempt to study a specialized case of 
$U_A = I$.

\begin{proposition} \label{prop:rate-PHB-k4=1}
    Suppose $\bz$ satisfies the conditions in 
    Theorem~\ref{thm:RFG-PHB-convg},
    $U_A = I$,
    $H_4 = \gamma_d T$,
    and $\bm{S} = [S_1, S_2^\top; S_2, S_3]$ where $T$ and $S_i$'s are diagonal.
    Then, the eigenvalues of $\Phi(\bm{S})$
    are the collection of the eigenvalues of $\Psi_i(\bm{S})$ defined by
    \begin{align*}
        \Psi_i(\bm{S}) := 
        \begin{bmatrix}
            h_{1,i} & h_{2,i} \\ 
            h_{2,i} & h_{3,i}
        \end{bmatrix} \in \mathbb{R}^{2\times 2}, \quad 1 \le i \le d,
    \end{align*}
    where
    $h_{1,i} 
        = v_i^2s_{1,i} + 2v_is_{2,i} + s_{3,i} + (\eta\beta^2\lambda_i)^2
        (\frac{\gamma_d}{\beta^4}T_{ii} - s_{1,i})$,
    $h_{2,i} = -\mu(v_is_{1,i} + s_{2,i})$,
    $h_{3,i} = \mu^2 s_{1,i}$,
    $v_i = 1+\mu - \eta\beta^2\lambda_i$, 
    $s_{k,i}$ is the $(i,i)$-entry of $S_k$,
    $v_i$ is the $(i,i)$-entry of $V$
    and 
    $T_{ii}$ is the $(i,i)$-entry of $T$.
\end{proposition}
\begin{proof}
    Since $U_A=I$,
    all $H_i$'s are diagonal.
    For $\ell=2,\dots,d$, let $P_\ell$ be an orthogonal matrix that interchanges $\ell$-th and $(d+\ell-1)$-th columns.
    Let $P=P_2\cdots P_d$.
    It then can be checked that 
    \begin{align*}
        P^\top \Phi(\bm{S})P = \begin{bmatrix}
            \Psi_1(\bm{S}) &  &  \\
            & \ddots & \\
            &        & \Psi_d(\bm{S})
        \end{bmatrix},
    \end{align*}
    where $h_{k,i}$ is the $(i,i)$-component of $H_k$.
    Since $P$ is orthogonal, $P^\top \Phi(\bm{S})P$ and $\Phi(\bm{S})$ are similar.
    The proof is then completed by observing that the eigenvalues of $P^\top \Phi(\bm{S})P$
    are the collection of the eigenvalues of $\Psi_i(\bm{S})$.
\end{proof}

Proposition~\ref{prop:rate-PHB-k4=1} provides an efficient way to calculate 
the eigenvalues of $\Phi^k(I_{2d})$ for the case of $U_A=I$.
While it is unclear what are the optimal choices of $\mu$ and $\eta$ 
for achieving the fastest rate of convergence,
since the largest eigenvalue of $\Psi_i(\bm{S})$ can be explicitly written as
\begin{align*}
    \lambda^{\Psi_i}_{\max} = \frac{h_{1,i} + h_{3,i}}{2} + \sqrt{\left(\frac{h_{1,i} - h_{3,i}}{2}\right)^2 + h_{2,i}^2},
\end{align*}
and it could be used in finding appropriate $\mu$ and $\eta$
via e.g. grid-search.

\section{Computational examples} \label{sec:example}
In this section, we provide numerical examples to verify our theoretical findings
and demonstrate the performance of the proposed method. The examples are ordered
by the ratio $\ell/d$ of graph length to parameter dimension, namely the quantity
that section~\ref{sec:complexity} identifies as governing the comparison with
backpropagation. Throughout, all computations are carried out in double
precision, and every figure reported for a randomized method is a median over
independent runs.

\subsection{Quadratic functions}
Since the theoretical investigations are based on
the quadratic objective functions \eqref{def:quadratic}, 
we aim to verify our theoretical findings numerically.
To compute the expected squared errors numerically,
we run 100 independent simulations and report the corresponding statistics.
In all simulations, we set $m=d$ and $h=10^{-6}$.

For RFG-based GD, we generate the random synthetic data as follows.
First, a matrix $M$ and a vector $b$ are randomly generated such that their components are independently drawn from the Gaussian distributions
$\mathcal{N}(0,1)$
and $\mathcal{N}(5,1)$, respectively.
Second, we perform the singular value decomposition (SVD) of $M$ to obtain $U, S, V$ matrices, i.e., $M=USV^\top$.
We then modify the singular values $S$ 
so that the condition number of the newly reconstructed matrix $A:=US_{\text{new}}V^\top$ 
is 10.
This ensures that the condition number of $A^\top A$ is 100 regardless of the size of $A$.
The data is generated once and fixed for all experiments.
The learning rate is chosen according to \eqref{RFG-GD-opt-lr}.
In Figure~\ref{fig:Quad_Ex01}, we plot the averaged squared errors versus the number of iterations 
for the six versions of RFG-based GD that use different probability distributions 
(Bernoulli, Uniform, Wigner, Gaussian, Laplace, and the uniform distribution on $\{\pm\sqrt{d}\bm{e}_i\}_{i=1}^d$). See also Table~\ref{table:prob-dist}.
In what follows, we refer to the uniform distribution on $\{\pm\sqrt{d}\bm{e}_i\}_{i=1}^d$ as the RCD distribution.
The top left, the top right, and the bottom left
are the results for $d=5, 10, 20$, respectively.
It can be seen that the fastest convergence is achieved when the Bernoulli and RCD distributions are used for the RFG-based GD, which is expected from Theorem~\ref{thm:RFG-GD-convg} and Remark~\ref{rmk:GDrate}.
On the bottom right,
the results of the Bernoulli distribution are shown
at varying $d=5,10,20,30$.
The rates of convergence from Theorem~\ref{thm:RFG-GD-convg} are shown as black dashed lines.
The shaded region spans one standard deviation about the mean.
The theoretical rates of convergence match the numerical simulations closely. 
Finally, it is also observed that 
the larger the dimension, the slower the convergence.
This is again expected as the rate is negative and inversely proportional to the dimension $d$.

\begin{figure}[!ht]
\centering
{\includegraphics[width=0.49\textwidth]{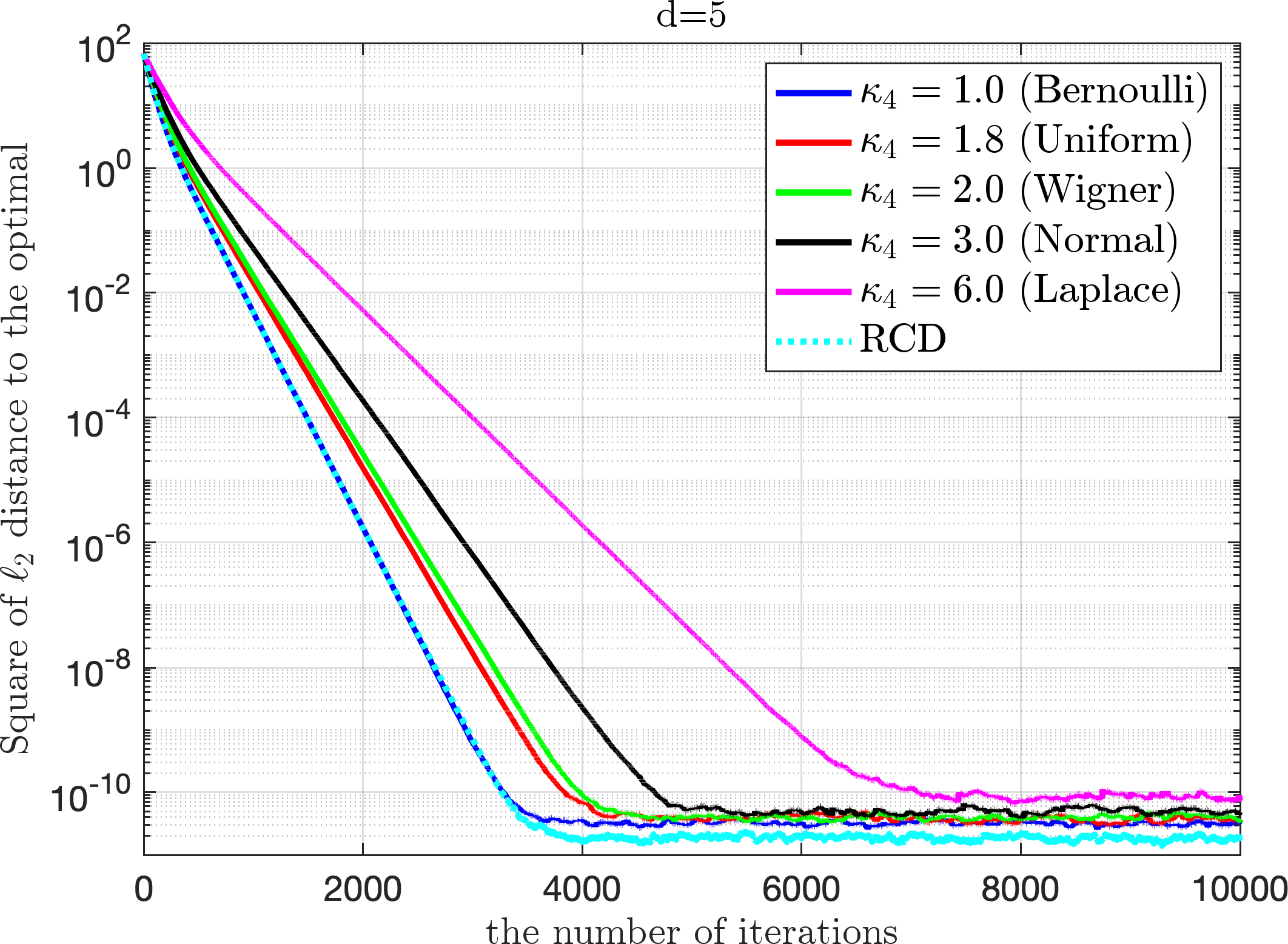}
\includegraphics[width=0.49\textwidth]{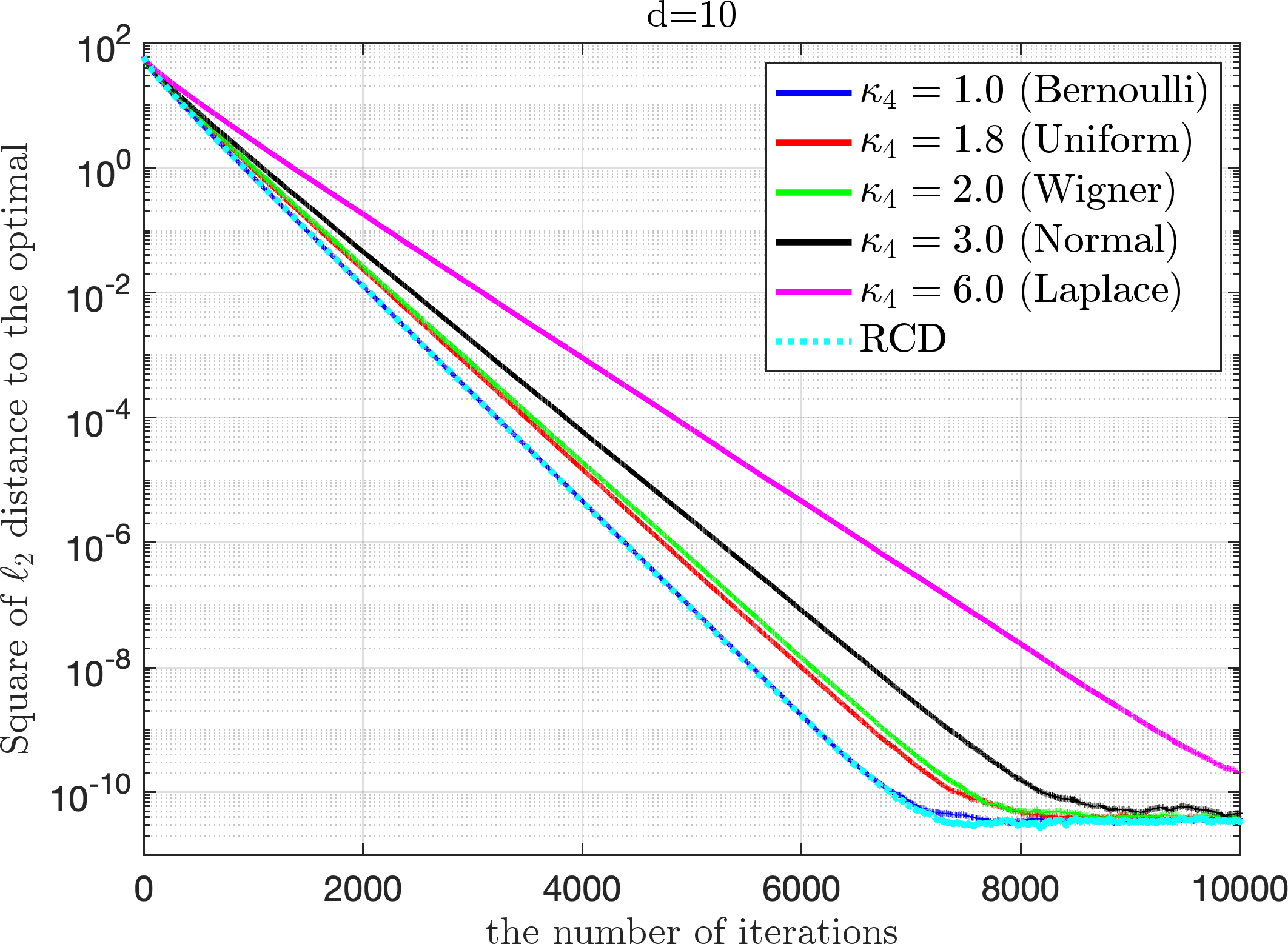}
} \\ \vspace{0.3cm}
{\includegraphics[width=0.49\textwidth]{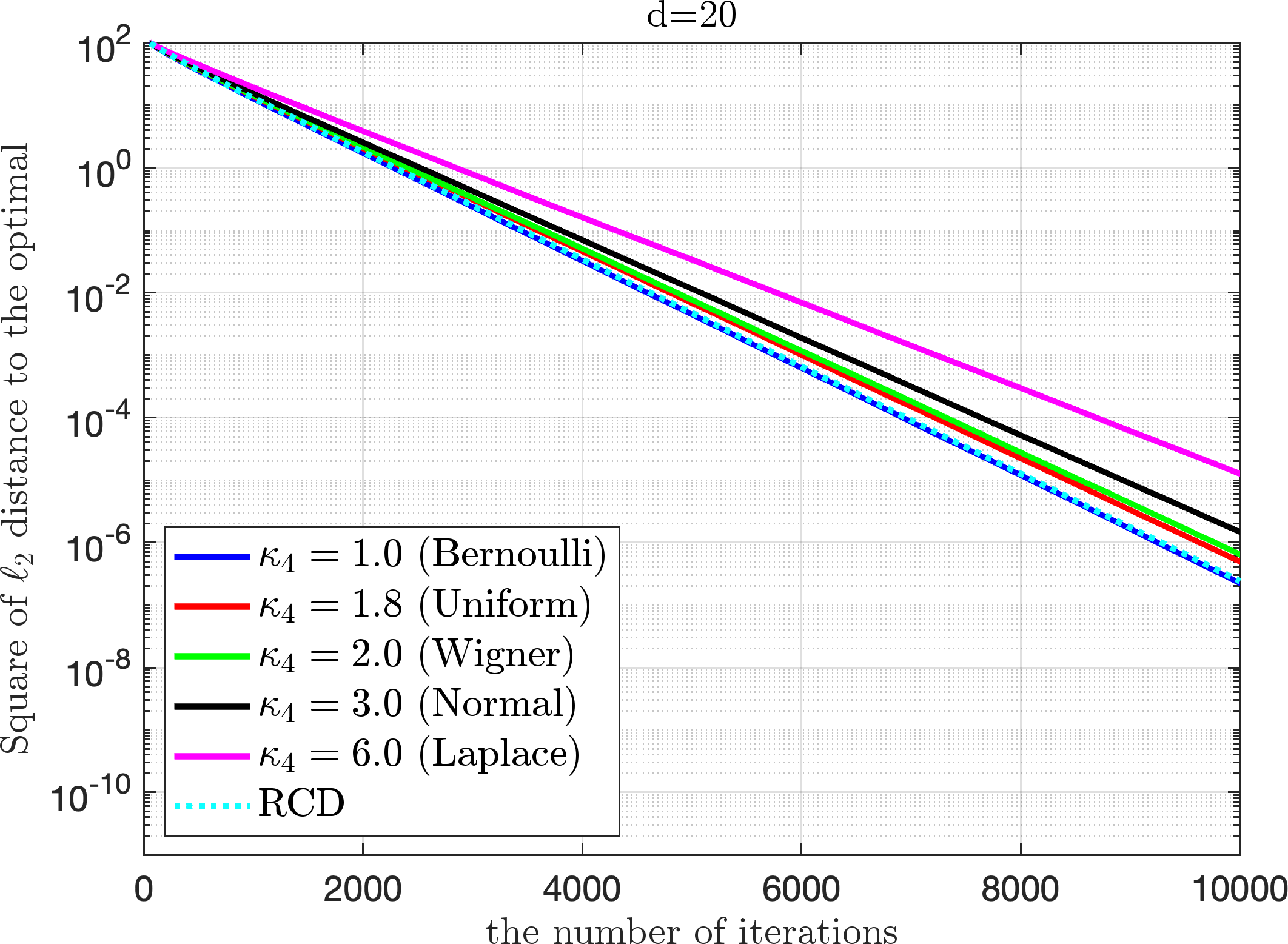}
\includegraphics[width=0.49\textwidth]{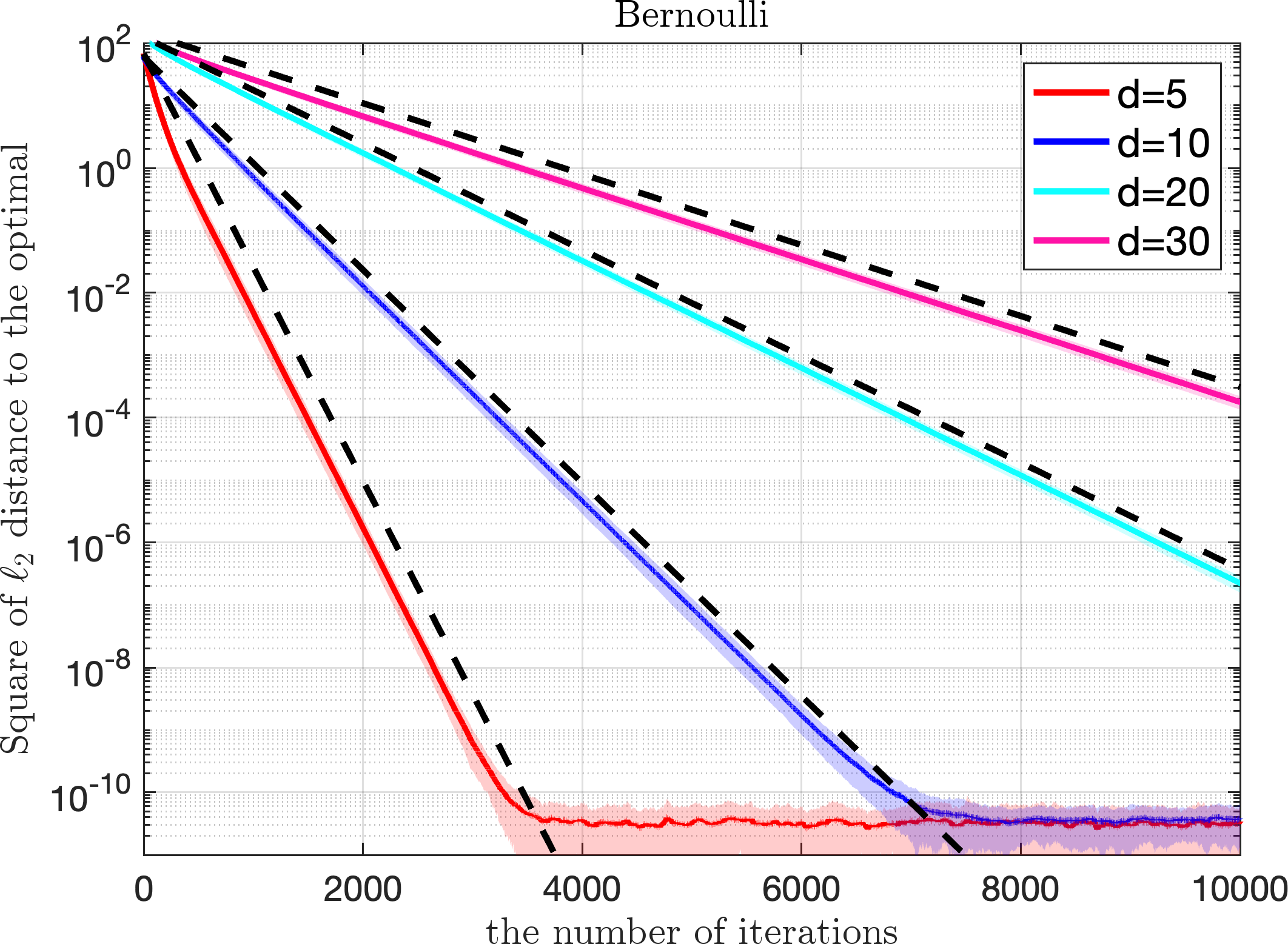}
}
\caption{Top and bottom left: The averaged squared errors versus the number of iterations obtained by the RFG-based GD using the six different probability distributions at varying dimensions $d=5, 10, 20$.
Bottom right: The averaged squared errors 
obtained by the Bernoulli RFG-based GD
along with the upper bounds from Theorem~\ref{thm:RFG-GD-convg} at varying dimensions $d=5, 10, 20, 30$.
The shaded region spans one standard deviation about the mean.
}
\label{fig:Quad_Ex01}
\end{figure}

For RFG-based PHB, 
since an explicit optimal pair of $\mu$ and $\eta$ is not available
from Theorem~\ref{thm:RFG-PHB-convg},
we focus on the specialized case $U_A = I$.
As in the case of RFG-based GD, we generate a matrix $M$ and a vector $b$ randomly from the Gaussian distributions
$\mathcal{N}(0,1)$ and $\mathcal{N}(5,1)$, respectively.
We again perform the SVD of $M$ to obtain $U, S, V$ matrices.
We then set $S_{\text{new}}$ whose first two components are set to 1 and 10, and the remaining components are drawn independently from 
the uniform distribution on $(1,10)$.
The newly reconstructed matrix is given by $A:=US_{\text{new}}$,
whose condition number is exactly 10.
We emphasize that the right singular vectors are deliberately omitted here:
with this choice $A^\top A = S_{\text{new}}^2$ is diagonal, so that $U_A = I$
and Proposition~\ref{prop:rate-PHB-k4=1} applies.
Again, the data is generated once and fixed for all experiments.
We remark that this allows us to find the best $\mu$ and $\eta$ 
from a grid search based on Proposition~\ref{prop:rate-PHB-k4=1}
for general probability distributions.
The grid used for the search is 
\begin{equation*}
    \Omega = \left\{(\mu, \eta) : \mu = i\times 10^{-2}, \eta = 10^{-5 + \frac{j}{100}}, i \in \{-99,\dots,99\}, j \in \{0,\dots,300\} \right\}.
\end{equation*}
which belongs to the domain $[-0.99,0.99] \times [10^{-5},10^{-2}]$.
For any $(\mu, \eta) \in \Omega$, we calculate the largest eigenvalue of $\Phi^{10K}(I_{2d})$
according to Proposition~\ref{prop:rate-PHB-k4=1}
and choose the pair that gives the smallest value.
On the top row and the bottom left of Figure~\ref{fig:Quad_Ex02}, the distribution of the largest eigenvalues of $\Phi^{10K}(I_{2d})$
is plotted on the grid $\Omega$ for the Bernoulli, Laplace, and RCD distributions.
For the purpose of the visualization, we clipped the values to lie between $10^{-11}$ and $10^0$.
The optimal pair found from the grid search is used for the implementation.
On the bottom right of Figure~\ref{fig:Quad_Ex02}, the trajectories of the sum of the consecutive squared errors, $\|\mathcal{E}_k\|^2$ defined in Theorem~\ref{thm:RFG-PHB-convg}, are plotted with respect to the number of iterations by various probability distributions.
Again, we observe that the one with the smallest kurtosis results in the fastest convergence rate.

\begin{figure}[!ht]
\centering
{\includegraphics[width=0.45\textwidth]{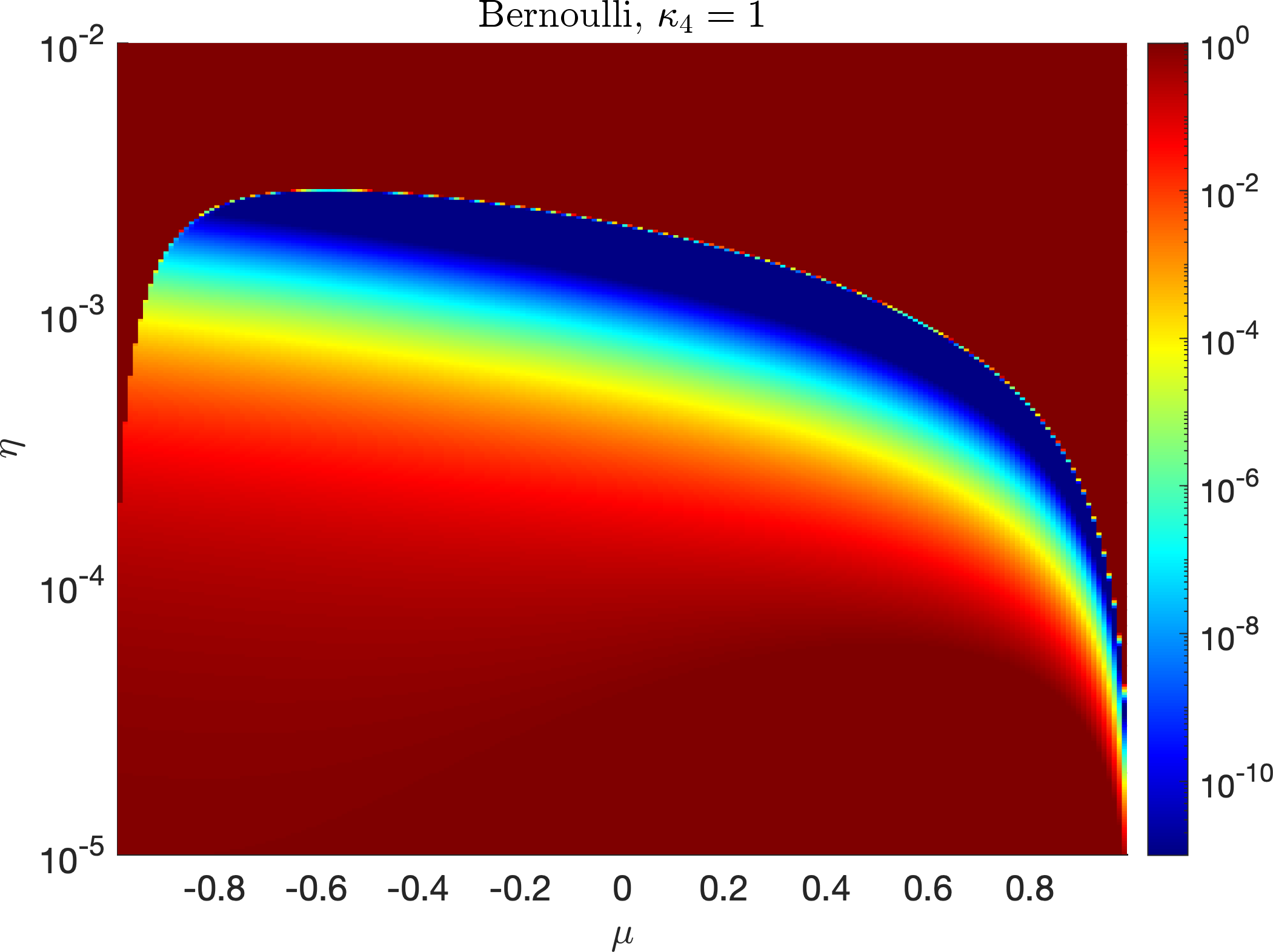}
\includegraphics[width=0.45\textwidth]{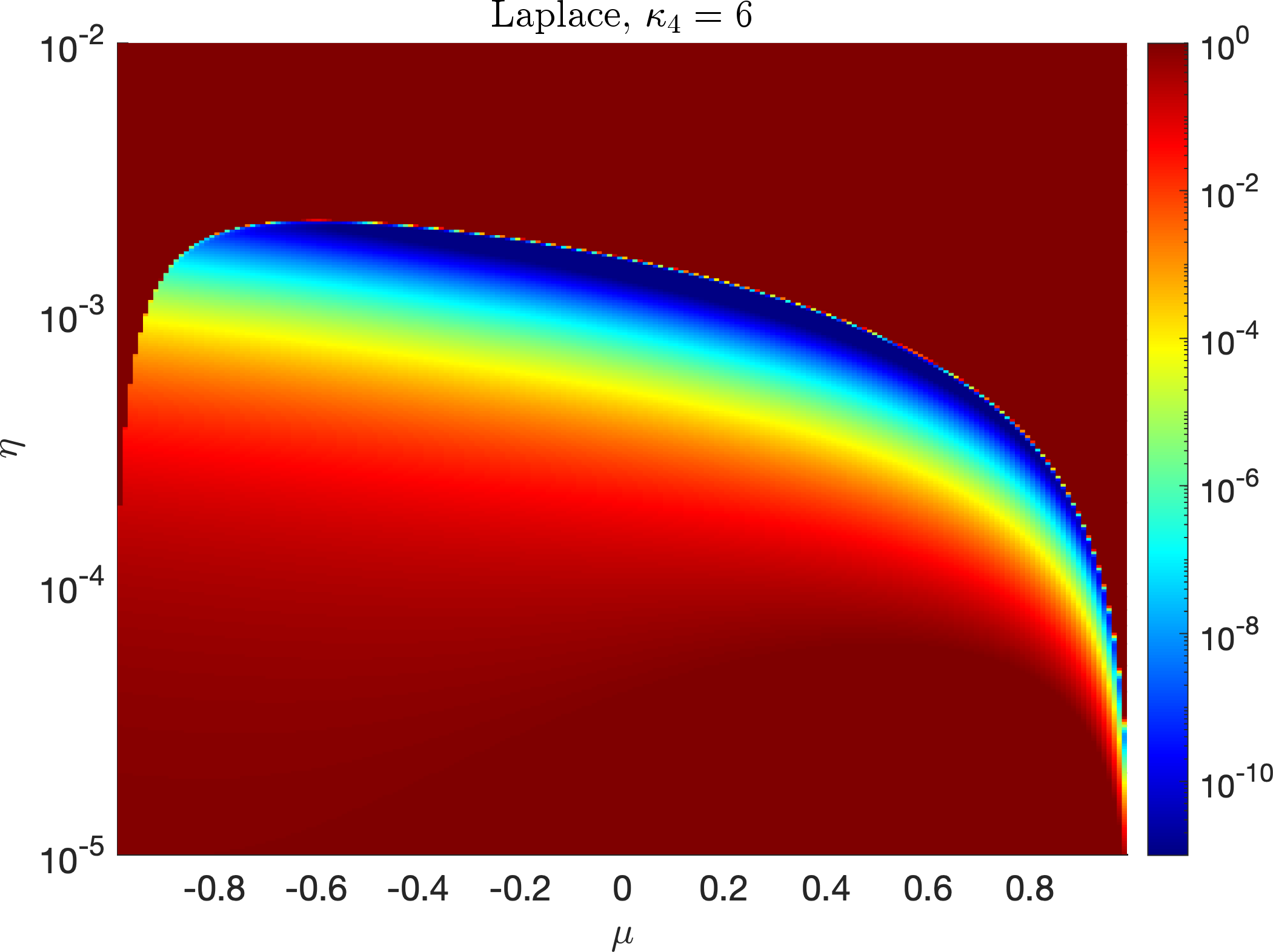}
}
{\includegraphics[width=0.45\textwidth]{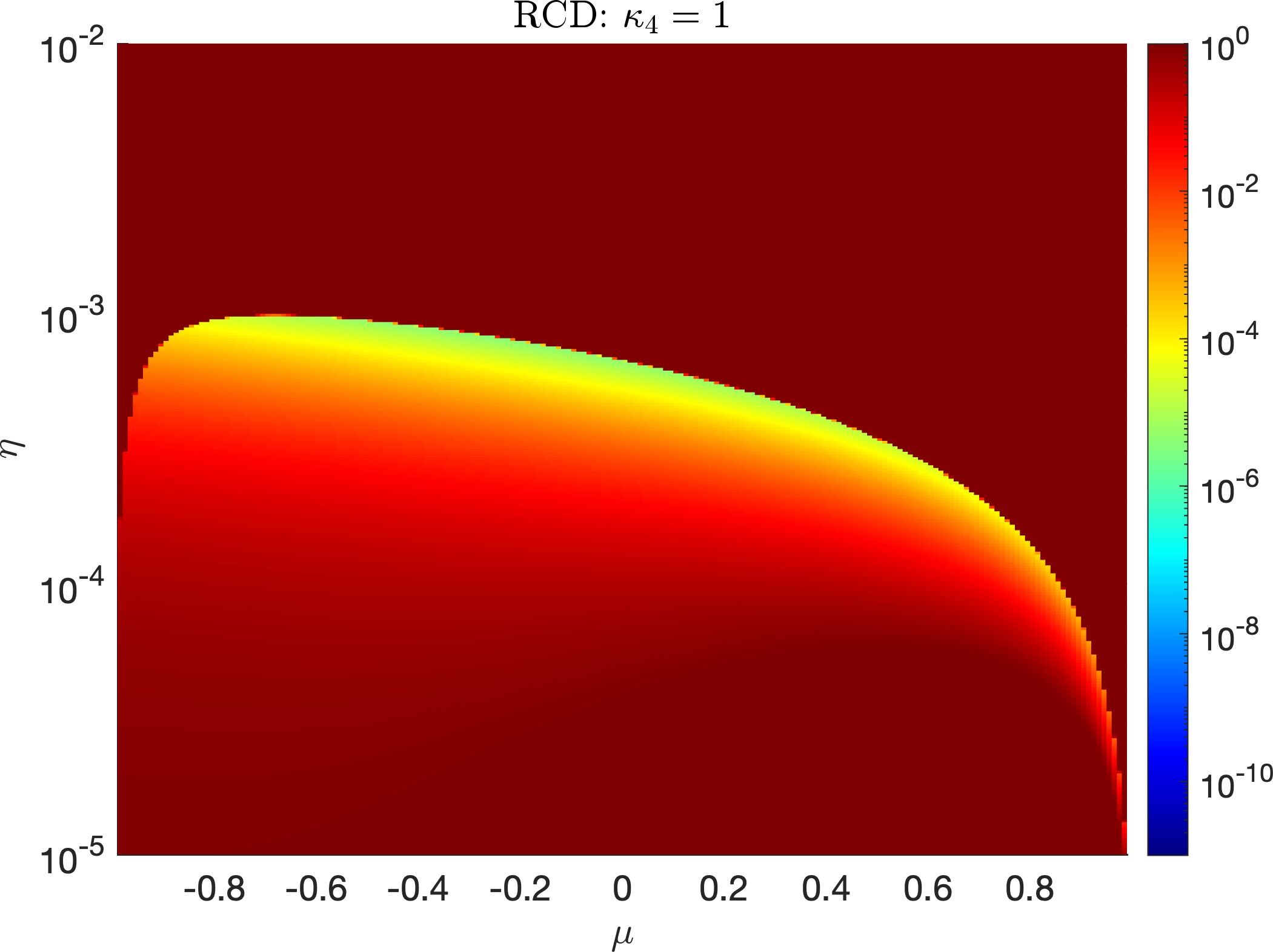}
\includegraphics[width=0.45\textwidth]{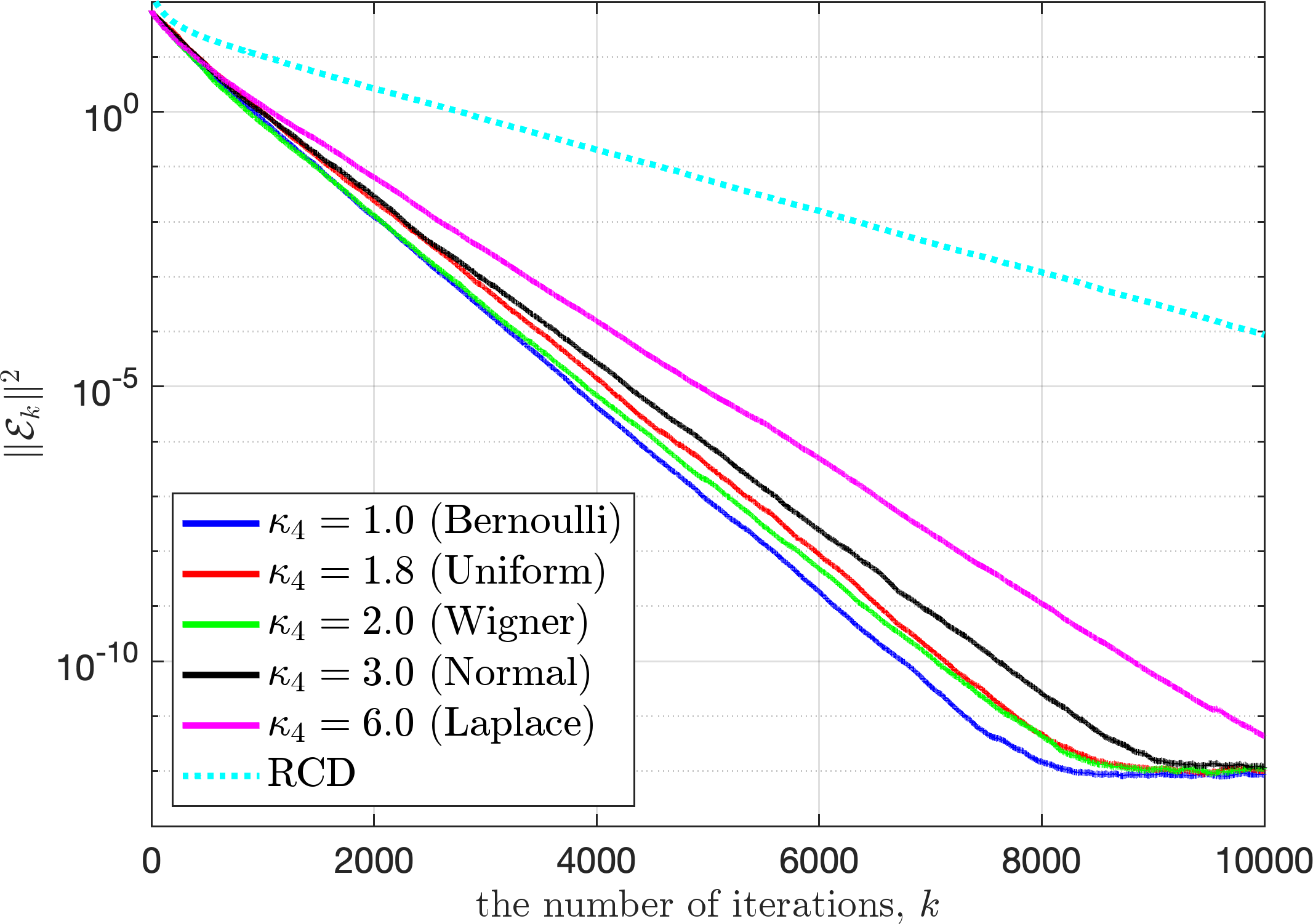}}
\caption{The value map of $\Phi^{10K}(I_{2d})$ on the grid $\Omega$ at $d=30$ for the Bernoulli distribution (top left), the Laplace distribution (top right), and the RCD distribution (bottom left).
Bottom right: The averaged squared errors versus the number of iterations obtained by the RFG-based PHB using the six different probability distributions at $d=30$.
}
\label{fig:Quad_Ex02}
\end{figure}


\subsection{Optimization test problems}
We consider the non-quadratic objective functions from \cite{simulationlib} that are commonly used as a testbed for optimization algorithms.
Specifically, we consider the Rosenbrock and Ackley functions.
\begin{equation}
    \begin{split}
        \text{(Rosenbrock)}&\quad f_{\text{Ros}}(x_1,x_2) = 100(x_2-x_1^2)^2 + (x_1-1)^2,
        \\
        \text{(Ackley)}&\quad f_{\text{Ack}}(x_1,x_2) = -20e^{-0.2\sqrt{0.5(x_1^2 + x_2^2)}} -e^{0.5(\cos2 \pi x_1 + \cos2 \pi x_2)} + e + 20.
    \end{split}
\end{equation}
The global minima for the Rosenbrock and Ackley functions are $\bx^* = (1,1)$ and $\bx^* = (0,0)$, respectively.
The initial point is set to $\bx^{(0)} = (0.5, 0.5)$.
For the Ackley function, the learning rate is set to a constant of $10^{-3}$ for both RFG-based GD and RFG-based PHB, and the momentum parameter $\mu$ for RFG-based PHB is set to $0.70$.
For the Rosenbrock function, the learning rate is chosen based on a grid search ranging from $2\times 10^{-5}$ to $2\times 10^{-3}$, and $\mu$ is set to $0.96$. 
All the results are averaged over 100 independent simulations. 

In Figure~\ref{fig:OTF}, the objective function values are reported with respect to the number of the RFG iterations. Similar to the previous example, the six different probability distributions are employed.
The top and bottom figures correspond to the results for the Rosenbrock and the Ackley functions, respectively.
The left and right figures present the errors by RFG-based GD and RFG-based PHB, respectively.
For RFG-based PHB, we report the minimum error along iterations, which refers to the lowest error achieved across all the iterations.
The rates of convergence differ with the choice of probability distribution. The Bernoulli distribution and RCD produce almost identical results, achieving the fastest convergence for RFG-based GD. Yet, the Bernoulli performs better than RCD for the Ackley case with RFG-based PHB.
In addition, it can be seen that RFG-based PHB produces smaller errors than RFG-based GD for all cases.

\begin{figure}[!ht]
\centering
{\includegraphics[width=0.49\textwidth]{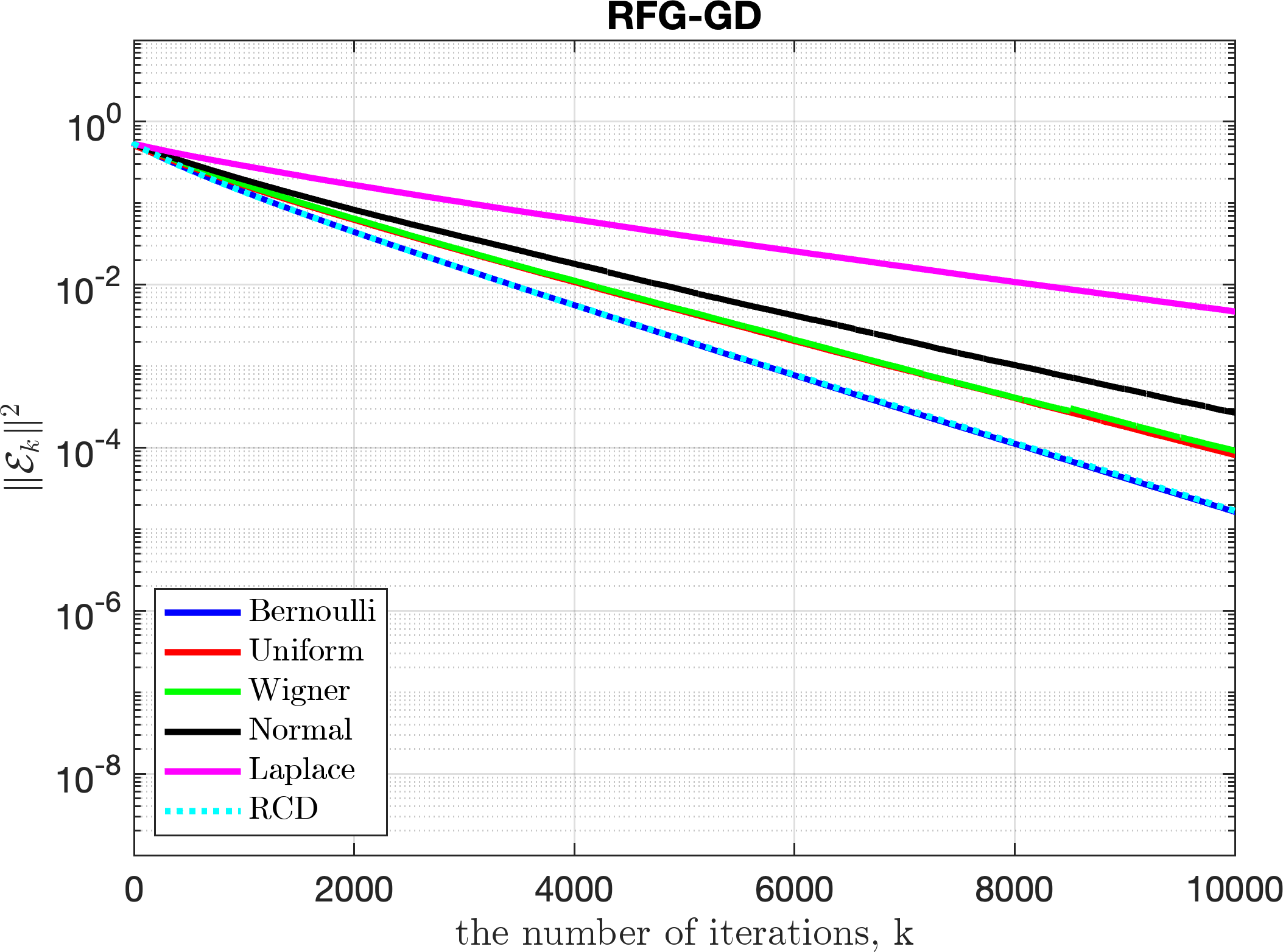}
\includegraphics[width=0.49\textwidth]{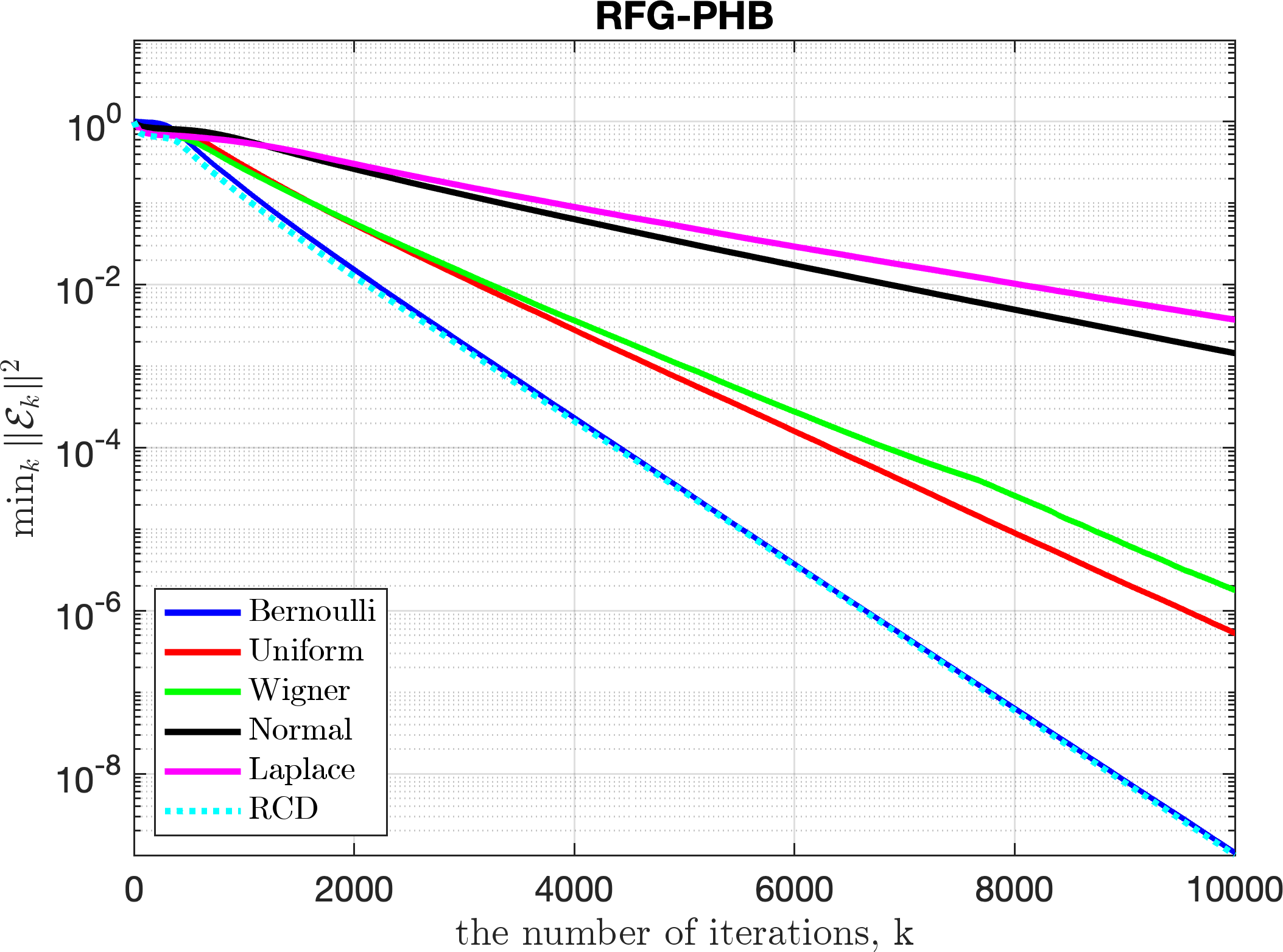}}
{\includegraphics[width=0.49\textwidth]{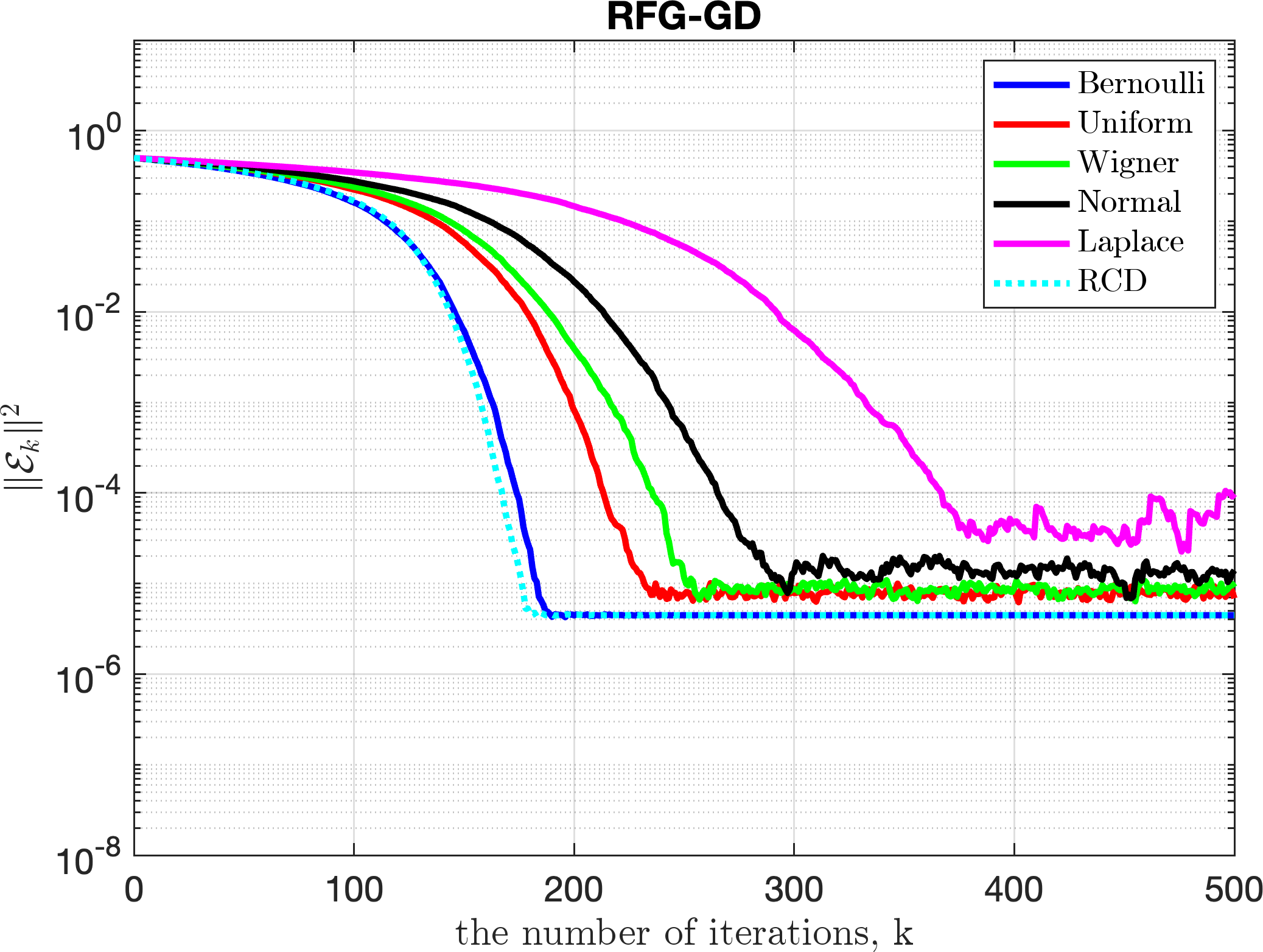}
\includegraphics[width=0.49\textwidth]{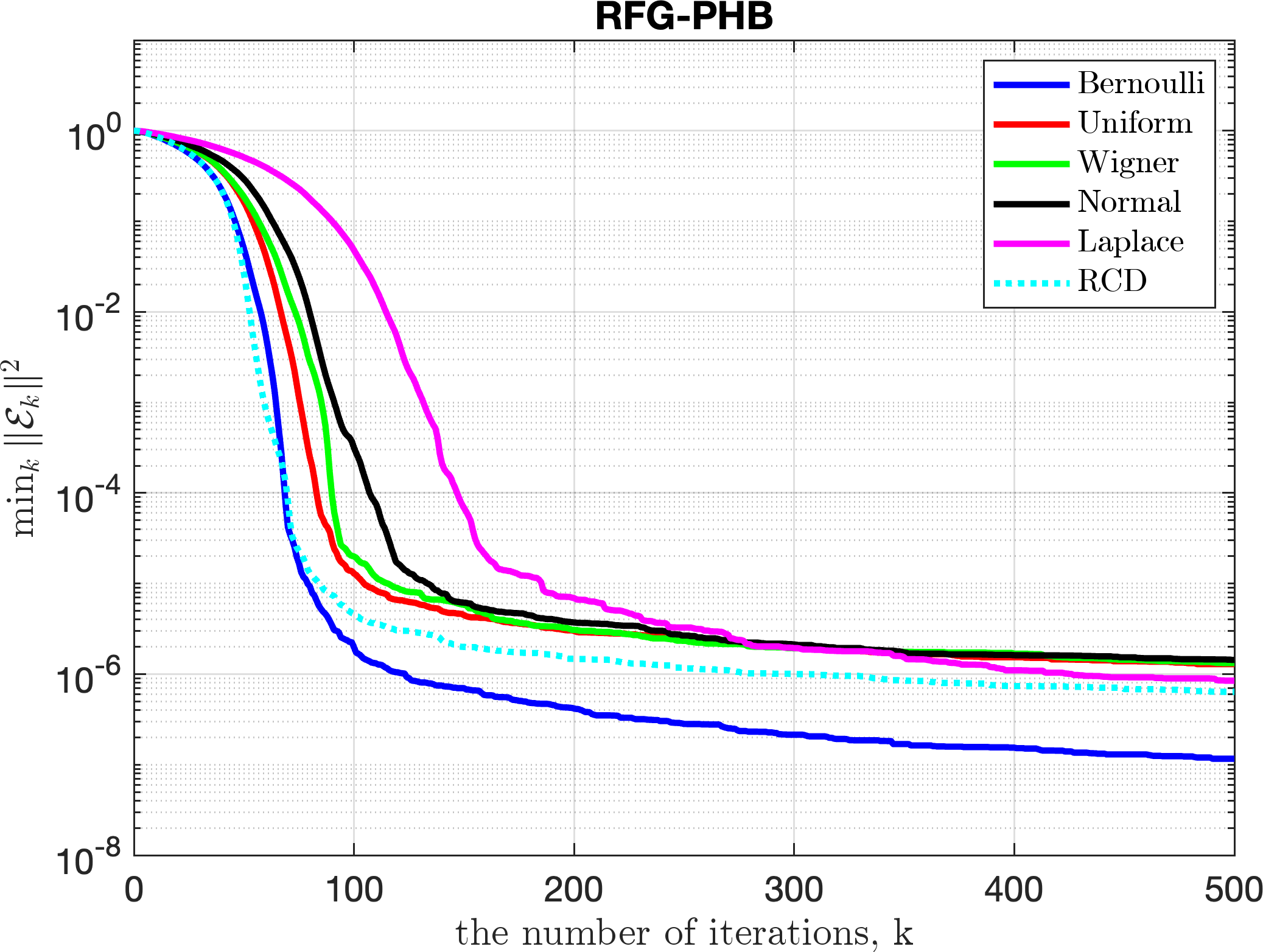}}
\caption{The square of the errors versus the number of iterations by the RFG-based optimization algorithms with six different probability distributions.
Top: Rosenbrock;
Bottom: Ackley;
Left: RFG-based GD;
Right: RFG-based PHB.
}
\label{fig:OTF}
\end{figure}

\subsection{A physics-informed neural network} \label{sec:sciml}
In what follows, we demonstrate the performance of the RFG algorithm on a
scientific machine learning application, a physics-informed neural network (PINN) \cite{shin2020, raissi2019physics, shin2023accelerating, shukla2021parallel}.
For this purpose, we briefly introduce the feed-forward neural network models. 
For $L \in \mathbb{N}_{\ge 2}$ and $N \in \mathbb{N}_{\ge 1}$, 
a $L$-layer feed-forward neural network (NN) of width $N$ is a function 
$$u_{\text{NN}}(\cdot;\theta):\mathbb{R}^{d_\text{in}} \ni x \mapsto u^L(x) \in \mathbb{R}^{d_\text{out}},$$ where $u^L$ is defined recursively as follows: for $x \in \mathbb{R}^d$, let $$u^\ell(x) = W^\ell \phi(u^{\ell-1}(x)) + b^\ell, \quad 2 \le \ell \le L,$$ 
with $u^1(x) = W^1x + b^1$.
Here $W^\ell, b^\ell$ are the weight matrix and bias vector of the $\ell$-th hidden layer, and the collection $\theta = \{W^\ell, b^\ell\}$ of them is called the network parameters. 
$\phi$ is an activation function that applies element-wise.
Here $L$ and $N$ are referred to as the depth and width of the NN, and together they characterize its complexity.

\subsubsection{Training by the RFG} \label{sec:pinn}
The first task is to solve the Poisson equation with PINNs \cite{raissi2019physics} using the RFG algorithm. Specifically, we consider 
\begin{align} \label{Poisson Equation}
-u''(x) = 4\pi^2 \sin(2\pi x),\quad \forall x \in (-1, 1)  \quad  \text{with} \quad u(-1)=u(1)=0, 
\end{align}
whose solution is given by $u(x) = \sin(2\pi x)$.

We employ a two-layer tanh NN, $u_\text{NN}$, of width 10, i.e., $L=2, N=10$, and train it to minimize the physics-informed loss function defined by
\begin{equation*}
    f_\text{PINN}(\theta) = \frac{1}{m}\sum_{i=1}^m \left( u''_\text{NN}(x_i;\theta) + 4\pi^2 \sin(2\pi x_i) \right)^2
    + \frac{1}{2}\left[\left( u_\text{NN}(-1;\theta) \right)^2 
    + \left( u_\text{NN}(1;\theta) \right)^2 
    \right],
\end{equation*}
where $x_i$'s are randomly uniformly sampled from $(-1,1)$.
The goal is then to minimize $f_\text{PINN}$ on which we employ the RFG-based GD algorithm at varying probability distributions. The resulting NN is the approximated solution to the equation, namely, the PINN. 
The RFG-based GD algorithm is employed with a constant learning rate of $0.2$.  
Figure~\ref{sciml_results} shows the relative $\ell_2$ errors of PINNs trained by the RFG-based GD algorithm with five different probability distributions. Notably, the Bernoulli distribution exhibits the fastest convergence while the normal distribution comes in second place. 
On the other hand, the uniform and Wigner distributions are unsuccessful, and the Laplace distribution diverges.

\begin{figure}[!ht]
\centering
{\includegraphics[width=0.6\textwidth]{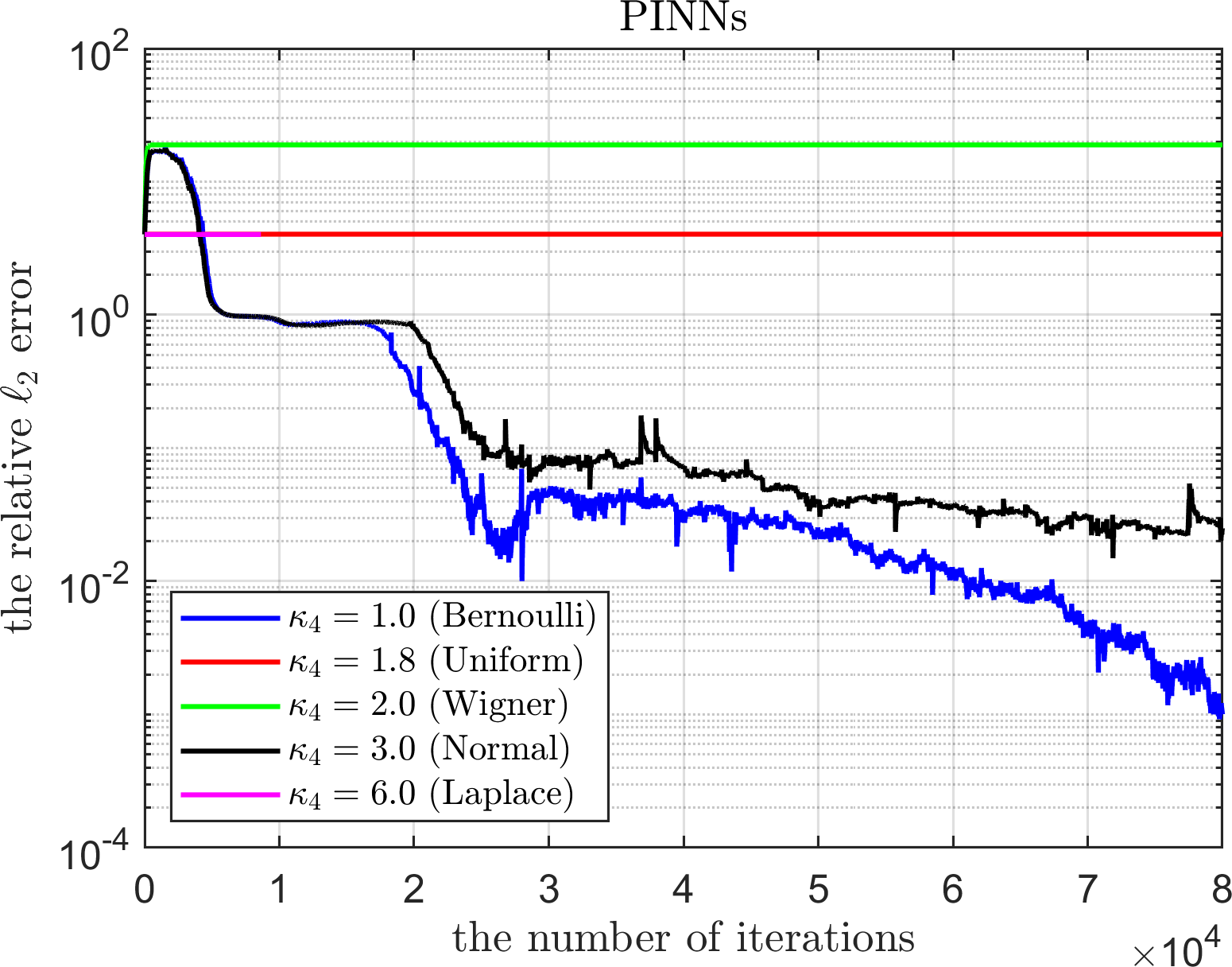}}
\caption{The training results of the RFG-based GD algorithms for the PINN
\eqref{Poisson Equation}. The relative $\ell_2$ errors are reported with respect
to the number of iterations.}
\label{sciml_results}
\end{figure}

\subsubsection{Memory and time}
\label{sec:sciml-cost}
The task above was used to compare direction distributions, which is a
statement about the estimator and not about its cost. We note here what the
comparison against backpropagation costs on this family, because the answer is
negative and the reason is structural. The measurements themselves are deferred
to Table~\ref{table:adam-cost}, where they appear alongside the parameter
identification problems under a single protocol.

This model has a two-layer computational graph, so $\ell = 2$ while $d = 31$. By
the criterion of section~\ref{sec:complexity} the relevant ratio $\ell/d$ is then
of order $10^{-2}$, and little memory advantage should be available. That is what
is observed: the memory factor is $2.3$, against $34$ to $3067$ on the parameter
identification problems of section~\ref{sec:inverse}. Checkpointing is
a no-op here, since a two-layer body has no interior to discard, so the forward
mode is compared against an unmodified tape and still gains little.

The cost in time is correspondingly severe, a factor of $20$ against an
unmodified reverse mode and $17$ against a checkpointed one, and for the same
reason: with no tape to avoid, the forward mode pays the $1/d$ information
penalty and receives nothing back. We regard this as
the negative half of the $\ell/d$ criterion rather than as a defect of the
estimator: these are exactly the problems the analysis predicts it cannot help,
and reporting them is what gives the positive cases their meaning.

\subsection{Parameter identification governed by differential equations}
\label{sec:inverse}
The examples considered so far are typical of machine learning, in that the
number of parameters is large and the computational graph is comparatively
short. We now introduce four problems with the opposite structure, in which a
small number of physical coefficients is recovered from a long numerical
integration. These are the problems used in section~\ref{sec:efficiency}, and
they are chosen so that the tape of the reverse mode grows for three different
reasons.

\subsubsection{A one-dimensional viscous Burgers equation}
\label{sec:P1}
On the periodic interval $[0,2\pi]$ we consider the viscous Burgers equation
with a spatially varying diffusivity,
\begin{equation} \label{eqn:burgers}
    \partial_t u = \nu(x)\,\partial_x^2 u - u\,\partial_x u,
    \qquad
    u(0,x) = \sin x + \tfrac{1}{2}\sin 2x .
\end{equation}
The domain is discretized on a uniform grid of $N=512$ points with spacing
$\Delta x = 2\pi/N$, the first and second derivatives are approximated by the
centred differences
\begin{equation*}
    (\partial_x u)_i \approx \frac{u_{i+1}-u_{i-1}}{2\Delta x},
    \qquad
    (\partial_x^2 u)_i \approx \frac{u_{i+1}-2u_i+u_{i-1}}{\Delta x^2},
\end{equation*}
with periodic wrap-around, and the semi-discrete system is advanced by the
explicit Euler method with $\Delta t = 10^{-4}$ for $T$ steps.

The quantity to be recovered is the diffusivity field, which we parameterize
through a truncated Fourier expansion,
\begin{equation} \label{eqn:nu-param}
    \nu(x;\bm\theta)
    = \nu_0 \exp\Big(\tanh\big(\textstyle\sum_{k=1}^{d}\theta_k\varphi_k(x)\big)\Big),
    \qquad
    \{\varphi_k\} = \{\cos kx,\ \sin kx\}_{k \ge 1},
\end{equation}
with $\nu_0 = 10^{-2}$ and $d=8$ coefficients.
The exponential keeps the diffusivity positive and the hyperbolic tangent keeps
it bounded, so that the explicit scheme remains stable for every $\bm\theta$
encountered during the optimization. This matters because the optimizer is free
to propose large coefficients and an unbounded parameterization would allow it
to leave the stability region.

A reference coefficient vector $\bm\theta^\star$ is drawn once from
$\mathcal{N}(0,\tfrac{1}{4}I_d)$ and held fixed, the corresponding terminal
state $u^{\text{obs}} = u^T(\bm\theta^\star)$ is computed and treated as data,
and the optimization is started from $\bm\theta^{(0)}=0$. The objective is the
mean squared discrepancy at the final time,
\begin{equation} \label{eqn:inv-loss}
    f(\bm\theta) = \frac{1}{N}\big\|u^T(\bm\theta) - u^{\text{obs}}\big\|^2 .
\end{equation}
Observing only the terminal state, rather than the whole trajectory, is
deliberate. It removes any intermediate quantity that a method might exploit and
leaves the entire dependence on $\bm\theta$ inside the $T$-step integration.

In the experiments $T$ is varied over $[5\times10^2,\,2\times10^4]$ with all
other quantities fixed. Since the parameter dimension is $d=8$ while the graph
contains $T$ sequential steps, this problem isolates the dependence of the tape
on the \emph{length} of the computation.

\subsubsection{A two-dimensional diffusion problem}
\label{sec:P2}
On the periodic square $[0,2\pi]^2$ we consider
\begin{equation} \label{eqn:diffusion2d}
    \partial_t u = \kappa(x,y)\,\Delta u,
    \qquad
    u(0,x,y) = \exp\!\big(-2\left[(x-\pi)^2 + (y-\pi)^2\right]\big),
\end{equation}
discretized on an $N \times N$ grid with the five-point Laplacian and advanced
by explicit Euler with $\Delta t = 5\times10^{-5}$.
The diffusivity is parameterized as in \eqref{eqn:nu-param} but with
two-dimensional modes $\{\cos kx\cos ky,\ \sin kx\sin ky\}$ and
$\kappa_0 = 5\times10^{-2}$, again with $d=8$ coefficients.
Data generation, initialization and objective follow
section~\ref{sec:P1} exactly, with \eqref{eqn:inv-loss} averaged over the $N^2$
grid points.

The distinction from section~\ref{sec:P1} is the size of the state rather than
the number of steps. The tape holds $T$ frames of $N^2$ values, so it grows as
$TN^2$, and we vary $N \in \{64,128,256,512\}$ at $T=200$ and $T$ over
$[50,800]$ at $N=256$ in order to separate the two effects. This is the
situation encountered in full-waveform inversion and in variational data
assimilation, where it is the size of each recorded frame, and not their number,
that makes storing the adjoint state impossible.

\subsubsection{A stiff reaction--diffusion system}
\label{sec:P3}
The third problem differs from the first two in that the forward map is
\emph{implicit}. We take the Robertson mechanism, the standard stiff benchmark
for chemical kinetics,
\begin{equation*}
    A \xrightarrow{\;k_1\;} B,
    \qquad
    B + B \xrightarrow{\;k_2\;} C + B,
    \qquad
    B + C \xrightarrow{\;k_3\;} A + C,
\end{equation*}
whose concentrations $\bm{c} = (c_A,c_B,c_C)$ obey
\begin{equation} \label{eqn:robertson}
    \frac{d\bm{c}}{dt} =
    \bm{g}(\bm{c};\bm{k}) :=
    \begin{pmatrix}
        -k_1 c_A + k_3 c_B c_C \\
        k_1 c_A - k_2 c_B^2 - k_3 c_B c_C \\
        k_2 c_B^2
    \end{pmatrix},
\end{equation}
with nominal rate constants $\bm{k}_0 = (4\times10^{-2},\,3\times10^{4},\,10^{4})$.
The system is posed at each of $N=128$ points of a periodic grid and coupled by
diffusion with coefficient $10^{-3}$.
The disparity of several orders of magnitude between the rate constants makes
the reaction term stiff, and an explicit integrator would require a time step far
smaller than the dynamics of interest.

We therefore use operator splitting with an implicit reaction step. Each of the
$T$ time steps applies backward Euler to \eqref{eqn:robertson}, that is, solves
\begin{equation*}
    \bm{F}(\bm{c}) := \bm{c} - \bm{c}^{\,n} - \Delta t\,\bm{g}(\bm{c};\bm{k}) = \bm{0},
\end{equation*}
by $K$ Newton iterations,
\begin{equation} \label{eqn:newton}
    \bm{c}^{(j+1)} = \bm{c}^{(j)}
    - \big[I - \Delta t\,\partial_{\bm{c}}\bm{g}(\bm{c}^{(j)};\bm{k})\big]^{-1}
      \bm{F}(\bm{c}^{(j)}),
\end{equation}
with the $3\times3$ Jacobian $\partial_{\bm{c}}\bm{g}$ formed analytically and
the solve performed pointwise across the grid. Diffusion is then applied
explicitly. We use $\Delta t = 0.2$, which an explicit scheme could not sustain,
and verified that $K=3$ is sufficient by checking that the objective agrees to
four significant figures with the value obtained at $K=6$.

The unknowns are the $d=8$ coefficients of a spatial modulation of the rate
constants, $\bm{k}(x;\bm\theta) = \bm{k}_0\exp(\tanh(\sum_k \theta_k\varphi_k(x)))$,
with data generated and the objective formed as before.

The tape of the reverse mode now grows as $TK$ rather than $T$, since every
outer step contains $K$ inner iterations, each of which is recorded. Varying $T$
and $K$ independently, over $[125,1000]$ and $\{1,2,3,6\}$ respectively, allows
the two factors to be separated, and it is this separation that exposes the
different behavior of checkpointing reported in
section~\ref{sec:efficiency}.

\subsubsection{A neural right-hand side identified from sparse observations}
\label{sec:P4}
The three problems above have a small parameter vector by construction: the
unknown is a handful of coefficients in a prescribed expansion. A natural
objection is that the memory advantage is an artifact of that choice, and
disappears once the unknown is a neural network. The fourth problem is designed
to test exactly that, and to occupy the regime $\ell\gg d$ with a machine
learning model rather than a physical parameterization.

Let $\bm{u}\in\mathbb{R}^2$ evolve under an unknown vector field and let
$\bm{u}_{\text{NN}}(\cdot\,;\theta)$ be a two-layer network of width $8$ that
plays the role of that field,
\begin{equation} \label{eqn:nodeobs}
    \frac{d\bm{u}}{dt} = \bm{u}_{\textup{NN}}(\bm{u};\theta), \qquad
    \bm{u}(0) = \bm{u}_0,
\end{equation}
integrated by a three-stage Runge--Kutta scheme for $T$ steps from a batch of
$32$ initial conditions. The reference is the damped pendulum
$\dot{x}_1 = x_2$, $\dot{x}_2 = -0.1x_2 - 9.8\sin x_1$. The state is compared to
observations at only $n_{\text{obs}}$ times, and the squared error is
\emph{accumulated inside the time loop}, so that no object of size $O(T)$ is
materialized and the reverse-mode tape is the only quantity that grows with the
horizon. The network has $d=42$ parameters, so $\ell/d$ ranges from $24$ to
$1190$ as $T$ runs from $10^3$ to $5\times10^4$.

Two details of the formulation matter and are worth recording, since the obvious
alternatives destroy the measurement. First, comparing the \emph{whole}
trajectory against a stored reference, as is common for neural ODEs, builds a
$T\times 32\times 2$ array that every mode must hold, forward mode included. The
memory ratio is then pinned near $1.3$ at every $T$ and the experiment measures
nothing. Second, the observations must be dense enough for the vector field to
be identifiable: at $n_{\text{obs}}=10$ the training loss falls while the
recovered field gets worse, and only at $n_{\text{obs}}\ge 50$ does minimizing
the objective recover the dynamics. We use $n_{\text{obs}} = 50$.

Checkpointing is applied at $\sqrt{\cdot}$ granularity \emph{inside} each
observation interval rather than only at the $n_{\text{obs}}$ interval
boundaries. This matters: boundary-only checkpointing ties the baseline's memory
to how many observations the data happens to contain, which is a property of the
data and not of the method, and it inflates the apparent advantage of the
forward mode by more than an order of magnitude.

\subsection{Computational efficiency: RFG versus backpropagation} \label{sec:efficiency}
The analysis of section~\ref{sec:complexity} predicts a trade rather than a
uniform improvement. The forward mode removes the factor $\ell$ from the memory
requirement \eqref{eqn:SAM}, but by Theorem~\ref{thm:RFG-GD-convg} it pays a
factor $d+\kappa_4-1$ in the number of iterations. Which effect dominates is a
property of the problem, governed by the ratio $\ell/d$ of graph length to
parameter dimension. We measure both sides.

Measurements were made in double precision, on CPU and on an NVIDIA A100, using
a single implementation for both. Peak memory is taken from the compiled
module's buffer assignment and corroborated by the process resident set size and
the device allocator. Timings exclude compilation and are medians over repeated
blocks after warm-up. The learning rate and the momentum parameter are searched
independently for each method. This matters more than it may appear. Holding the
momentum common at $0.9$ penalizes the reverse mode by a factor of $7.4$ relative
to its own optimum and RFG by only $2.3$, which by itself reverses the apparent
ordering.
Where Adam replaces momentum gradient descent, as in
Table~\ref{table:adam-cost}, only the learning rate is searched, the remaining
coefficients being held at their usual values.
Except where both are reported, the baseline is reverse mode \emph{with
checkpointing}. Comparisons against the naive reverse mode are more favorable to
RFG and would overstate the case.

\textbf{Memory.}
The measured scaling exponents confirm \eqref{eqn:SAM} directly. For the
problems of section~\ref{sec:inverse}, the peak memory of the reverse mode grows
with exponent $1.00$ in the number of steps $T$, while that of RFG is flat, with
exponent $0.00$. In the two-dimensional problem the exponents in the grid size
are $2.00$ for the reverse mode and $2.12$ for RFG, since the state itself scales
that way in both cases. RFG nonetheless carries a far smaller constant. These
exponents are unchanged across single and double precision and across the CPU and
GPU backends. They must be, since they follow from the structure of the tape
rather than the width of a floating-point number. Figure~\ref{fig:eff-regime} shows the two axes for the
one-dimensional problem, and Table~\ref{table:inverse-efficiency} gives the
resulting factors, which range from $34$ to $3067$.

The source of the advantage is worth identifying, because it is not the obvious
one. The operation counts reported by the compiler show that reverse mode
performs only $1.06$ times the arithmetic of RFG on \eqref{eqn:burgers},
uniformly over $T \in [5\times10^{2}, 2\times10^{4}]$, yet it is the slower of
the two per step. At $T=2\times10^{3}$ the factor is $1.8$ on the A100 and $2.6$
on the CPU, the reciprocals of the cost-per-step entry of
Table~\ref{table:adam-cost}. The two modes do essentially the same
floating-point work. The entire difference is the traffic of writing the tape
and reading it back. The memory advantage and the lower cost per step are
therefore the same phenomenon measured twice.

\textbf{A neural unknown in the same regime.}
The problem of section~\ref{sec:P4} places a network, rather than a handful of
coefficients, in the long-horizon regime. Its memory behaves exactly as the
criterion predicts, as Table~\ref{table:nodeobs-memory} shows: RFG's peak is
\emph{identical} at every horizon, while the reverse mode grows linearly and
checkpointing as $\sqrt{T}$.

\begin{table}[htp!]
\begin{center}
\small
\begin{tabular}{c|c|c|c|c|c}
$T$ & $\ell/d$ & reverse & \,+\,ckpt & RFG & RFG vs.\ ckpt \\ \hline
$10^3$            & $23.8$   & \SI{14.2}{\mega\byte}  & \SI{0.112}{\mega\byte} & \SI{0.0151}{\mega\byte} & $7.4\times$ \\
$5\times10^3$     & $119$    & \SI{70.7}{\mega\byte}  & \SI{0.184}{\mega\byte} & \SI{0.0151}{\mega\byte} & $12.2\times$ \\
$2\times10^4$     & $476$    & \SI{282.3}{\mega\byte} & \SI{0.327}{\mega\byte} & \SI{0.0151}{\mega\byte} & $21.7\times$ \\
$5\times10^4$     & $1190$   & \SI{705.6}{\mega\byte} & \SI{0.606}{\mega\byte} & \SI{0.0151}{\mega\byte} & $\mathbf{40.2\times}$
\end{tabular}
\end{center}
\caption{Peak memory for the neural right-hand side of section~\ref{sec:P4},
$d=42$, double precision, CPU, at the observation density $n_{\text{obs}}=50$
used throughout that section. RFG's peak does not change with the horizon. The
advantage over a properly checkpointed reverse mode therefore grows as
$\sqrt{T}$.}
\label{table:nodeobs-memory}
\end{table}

\textbf{When checkpointing helps.}
The factor achieved spans close to two orders of magnitude across the problems of
section~\ref{sec:inverse}, and the reason is structural. Checkpointing discards the interior of a recomputable
segment and retains its boundary, so it is effective exactly when the body of
the loop has substantial interior state. In the stiff system of section~\ref{sec:P3}, whose steps
contain $K$ Newton iterations, it reduces the scaling exponent in $K$ from $1.00$
to $0.04$, removing that factor from the tape entirely. For a feed-forward
network it reduces the exponent in depth from $0.94$ to $0.55$, the expected
$\mathcal{O}(\sqrt{L})$. Figure~\ref{fig:eff-kinetics} displays the two factors of
the stiff system separately. Where the body is a single inexpensive update, as in
sections~\ref{sec:P1} and \ref{sec:P2}, there is no interior to discard: the
boundary state must be retained at every step, and the exponents remain $1.00$ and
$2.00$ on both backends. The forward mode is indifferent to the distinction, since
it stores nothing in either case. This identifies precisely when the forward
mode is not merely cheaper but is the only method that fits: a long
computational graph whose individual steps are too simple for recomputation to
recover anything.

This bears on a setting we have not measured but which motivates the approach.
Where a model must be trained on hardware whose memory is severely constrained,
the binding question is not which method is faster but which method fits at all.
The iteration penalty applies with equal force there, so the argument holds only
when $d$ is small enough for that penalty to be tolerable. The claim is one of
feasibility, not of speed.

\textbf{Time to solution.}
Here the ordering reverses. Across eight independent
problem instances of \eqref{eqn:burgers}, with both methods tuned as described,
RFG reached the target accuracy more slowly by a median factor of $3.2$, and was
faster on none of the eight. The median iteration penalty was $6.5$, in
reasonable agreement with the factor $d+\kappa_4-1 = 8$ predicted by
Theorem~\ref{thm:RFG-GD-convg}. That agreement is the more informative of the
two numbers: once the baseline is tuned as carefully as the method under test,
the dimension penalty established by the analysis is realized in practice rather
than masked by an under-tuned comparison.

Replacing momentum gradient descent by Adam on both
sides, with the learning rate searched independently for each, narrows the
wall-clock gap on \eqref{eqn:burgers} to a factor of $1.50$ against an
unmodified reverse mode and $1.35$ against a checkpointed one, while the peak
memory falls by factors of $946$ and $316$ respectively.
Table~\ref{table:adam-cost} reports the three families under this single protocol.

\begin{table}[htp!]
\begin{center}
\footnotesize
\begin{tabular}{l|r|r|r|r|rr|rr}
 &  &  &  &  & \multicolumn{2}{c|}{time} & \multicolumn{2}{c}{memory} \\
family & $d$ & $\ell/d$ & iter. & cost/step & rev. & ckpt & rev. & ckpt \\ \hline
Burgers, \S\ref{sec:P1}         & $8$    & $2.5\times10^{2}$  & $2.75$  & $0.54$ & $\mathbf{1.50}$ & $\mathbf{1.35}$ & $\mathbf{946}$ & $\mathbf{316}$ \\
neural RHS, \S\ref{sec:P4}$^{\dagger}$ & $42$   & $11.9$             & $12.3$  & $0.30$ & $3.66$ & $3.21$ & $358$ & $5.9$ \\
PINN, \S\ref{sec:pinn}          & $31$   & $6.5\times10^{-2}$ & $25.5$  & $0.78$ & $19.8$ & $17.0$ & $2.3$ & $2.3$
\end{tabular}
\end{center}
\caption{RFG relative to reverse mode on an NVIDIA A100, all families under one
protocol: Adam on both sides, learning rate searched independently for each,
double precision. The iteration and cost-per-step columns are factors by which
RFG is \emph{worse}, and their product is the time column. The time and memory
columns are each given against an unmodified reverse mode (rev.) and against a
checkpointed one (ckpt), the latter being the conservative baseline used
elsewhere in this section. Iteration counts are measured to a fixed target and
the cost per step is timed separately as a median over compiled blocks after
warm-up, so that neither the target check nor the dispatch between blocks enters
the comparison. The ordering is monotone in $\ell/d$.
$^{\dagger}$The neural right-hand side is run here at $T=500$, a shorter horizon
than the range $T \in [10^3, 5\times10^4]$ of Table~\ref{table:nodeobs-memory},
so that the learning-rate search is affordable at every rate.}
\label{table:adam-cost}
\end{table}

The table exhibits the mechanism rather than merely the outcome. On
\eqref{eqn:burgers} the iteration penalty of $2.75$ is largely offset by a
per-step cost of $0.54$, since the reverse mode must write and read a tape of
$2\times10^3$ steps that the forward mode never creates. As $\ell/d$ falls the
tape shrinks until, on the PINN, it is no longer large enough to repay the
$d$-dimensional tangent the forward mode must carry, and the cost-per-step
saving all but disappears. Both effects are visible separately, and their
balance is what the criterion of section~\ref{sec:complexity} predicts.

The balance between them, and therefore the time column, depends on the
hardware. Repeating the measurement on CPU reproduces the iteration counts on
\eqref{eqn:burgers} exactly, namely $40$ reverse-mode and $110$ RFG steps on both
machines, but changes the cost per step from $0.54$ to $0.38$, so that the
wall-clock factor falls from $1.50$ to $1.04$ and the memory factor falls from
$946$ to $657$. The deterministic iteration counts agree on every family, as do
the RFG counts on each of the three reported here. The reason is that what the forward
mode declines to pay is tape traffic, and the A100's memory bandwidth makes that
traffic comparatively cheap. The saving is therefore largest where bandwidth is
scarcest. The iteration penalty, which is a property of the estimator rather than
of the machine, does not move at all.

The iteration penalty of $2.75$ realized under Adam is well below the factor
$d+\kappa_4-1 = 8$ of Theorem~\ref{thm:RFG-GD-convg}, and below the $6.5$
measured above under momentum gradient descent. This is not in conflict with the
analysis: the theorem bounds RFG-based \emph{gradient descent}, so a
better-conditioned optimizer is free to do better, and it is the momentum
measurement, made with the same update the theorem describes, that should be
compared against the bound.

Two statements should be distinguished. The dimension penalty in the
\emph{iteration count} is a property of the estimator: it is present under every
optimizer tested, and under the update the analysis describes its magnitude
tracks $d+\kappa_4-1$. Whether it is visible
in \emph{wall-clock time} is not, because it competes against the per-step saving
from carrying no tape, and on a problem with $\ell/d$ large enough the two
cancel. This is the same trade the memory results describe, now measured on the
axis where it is least favorable to the forward mode.

For neural network problems such as that of section~\ref{sec:pinn}, where $d$ is
large and $\ell$ short, RFG is not competitive at all.
Figure~\ref{fig:eff-memory} reports memory for a feed-forward network and a
neural ODE. The memory advantage is present there too, and the crossover in
depth it exhibits is the prediction of section~\ref{sec:complexity}, since the
forward mode carries a tangent of parameter size and is cheaper only once
$\ell\,\mathcal{M}(f)$ dominates parameter storage.
That advantage is nonetheless of no practical use here: at comparable memory the
checkpointed reverse mode reaches a substantially lower loss within a common
iteration budget. Reaching a fixed target accuracy took RFG
between two and fifteen times longer than the checkpointed reverse mode across
the depths and rollout lengths considered, consistent with the empirical
findings of \cite{fournier2023can}.
The size of this penalty is sensitive to how carefully the step size is chosen:
a coarse search over the learning rate suggests that RFG fails to reach the
target at all on these problems, whereas a search by successive refinement shows
that it converges, only slowly. All figures reported here are tuned, the search
having been verified to bracket the optimum of each method. We note also that the
physics-informed loss of section~\ref{sec:pinn} contains $u''$ and therefore
invokes the reverse mode regardless of the outer differentiation strategy, so
the method offers no advantage there.

\textbf{Why the iteration penalty cannot be tuned away.}
It is natural to ask whether the penalty is an artifact of the step size rather
than of the estimator. It is not, and the reason is an information count. A
reverse sweep returns all $d$ components of $\nabla f$. A single JVP returns the
one scalar $\bm{z}^\top\nabla f$. For a direction drawn as in
Assumption~\ref{assumption:prob-dist} the alignment satisfies
$\mathbb{E}[\cos^2(\bm{z},\nabla f)] = 1/d$, which we measured as $0.0239$
against $1/d = 0.0238$ at $d=42$. The best conceivable step along a random
direction therefore realizes a $1/d$ fraction of the decrease available to
steepest descent, whatever step size is chosen.

This is borne out by experiment. On the problem of section~\ref{sec:P4}, neither
a schedule, a line search, nor a rule rejecting poorly aligned directions
improves on a tuned constant rate by more than a constant factor. What remains is
the factor the analysis predicts.

\textbf{Time, not iterations.}
Iteration counts overstate the cost of the forward mode, whose steps are cheaper
because they write and read no tape. Consider the problem of
section~\ref{sec:P4}, measured serially with each method at its own best
configuration, taking medians over five independent direction sequences. Reverse
mode reaches the target in $250$ iterations at \SI{45.8}{\milli\second} each, RFG
in $2500$ at \SI{19.3}{\milli\second}. The iteration ratio is $10$. The ratio that matters is
\begin{equation*}
    \frac{2500 \times \SI{19.3}{\milli\second}}{250 \times \SI{45.8}{\milli\second}}
    = 4.2 .
\end{equation*}
We report wall-clock time throughout for this reason. The iteration count of RFG
also varies substantially with the direction sequence, ranging from $1000$ to
$15500$ over the five sequences. Single-run timings of a randomized estimator are
therefore not meaningful, and every figure quoted here is a median.

The conclusion is accordingly narrower than a claim of general efficiency. The
advantage of the forward mode in these experiments is one of memory: it is
large, it is deterministic, and it grows with the length of the computational
graph. Its cost is a proportionate increase in the number of iterations. RFG is
the method of choice when memory is the binding constraint, that is, when the
tape does not fit or when the recomputation demanded by checkpointing is itself
unaffordable, and not otherwise.

\begin{table}[htp!]
\begin{center}
\footnotesize
\setlength{\tabcolsep}{3.5pt}
\begin{tabular}{l|c|c|c|c}
 & $d$ & reverse & \,+\,ckpt & RFG \\ \hline
Burgers 1D, $T=2\times10^4$
    & $8$ & \SI{245.8}{\mega\byte} & \SI{81.9}{\mega\byte} & $\mathbf{\SI{0.0267}{\mega\byte}}$ \\
\quad factor over ckpt
    &     &  &  & $\mathbf{3067\times}$ \\ \hline
Diffusion 2D, $N=512$
    & $8$ & \SI{429.9}{\mega\byte} & \SI{432.0}{\mega\byte} & $\mathbf{\SI{12.6}{\mega\byte}}$ \\
\quad factor over ckpt
    &     &  &  & $\mathbf{34.3\times}$ \\ \hline
Stiff kinetics, $T=500$, $K=3$ Newton iterations
    & $8$ & \SI{28.5}{\mega\byte} & \SI{1.65}{\mega\byte} & $\mathbf{\SI{0.0426}{\mega\byte}}$ \\
\quad factor over ckpt
    &     &  &  & $\mathbf{38.8\times}$
\end{tabular}
\end{center}
\caption{Peak memory for the first three parameter identification problems, each with
$d=8$ unknown coefficients, measured in double precision on an NVIDIA A100.
The checkpointed column is the appropriate baseline. Against naive reverse mode
the factors are larger on the 1D and stiff problems ($9198\times$ and
$669\times$). On the 2D problem checkpointing costs marginally more than the
naive tape, so there the two baselines coincide to within $1\%$.
The factors span close to two orders of magnitude. The text explains why
checkpointing is effective on the stiff system but not on the one- and
two-dimensional problems.}
\label{table:inverse-efficiency}
\end{table}

\begin{figure}[!ht]
\centering
\includegraphics[width=\textwidth]{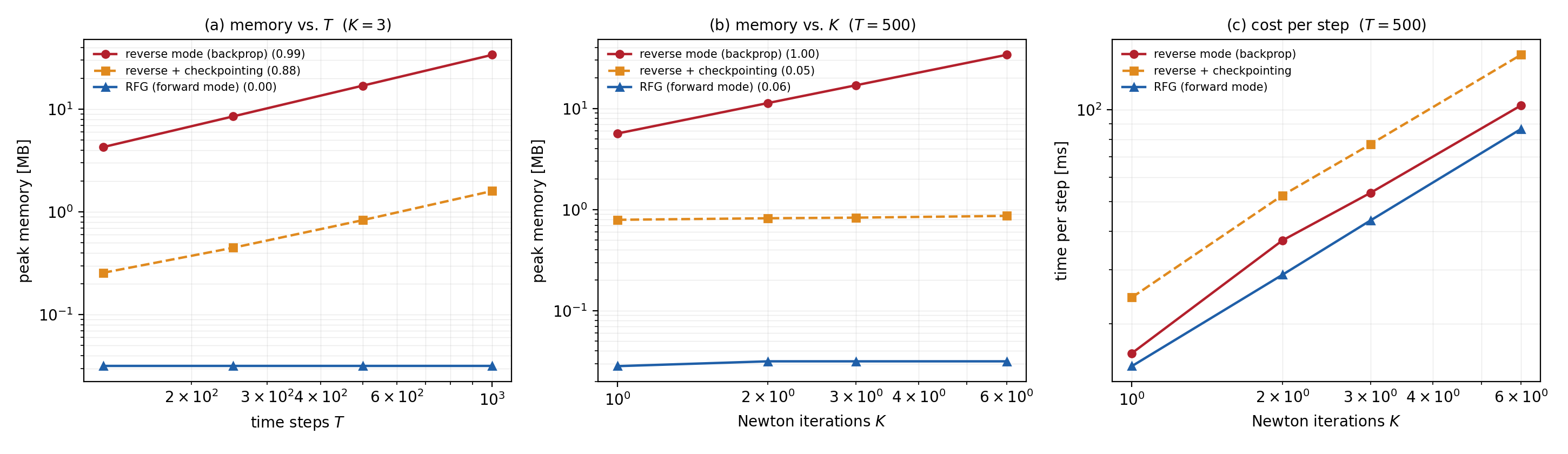}
\caption{The stiff kinetics problem, whose tape factorizes into an outer number
of time steps $T$ and an inner number of Newton iterations $K$.
Fitted memory scaling exponents are shown in the legends.
(a) In $T$, checkpointing remains essentially linear because it must retain one
snapshot per outer step, while the forward mode is flat.
(b) In $K$, checkpointing is flat because it discards every inner iterate,
as is the forward mode, but at a much smaller constant.
(c) Cost per step: the checkpointed method is the most expensive of the three,
the difference from reverse mode being the cost of recomputation.}
\label{fig:eff-kinetics}
\end{figure}

\begin{figure}[!ht]
\centering
\includegraphics[width=\textwidth]{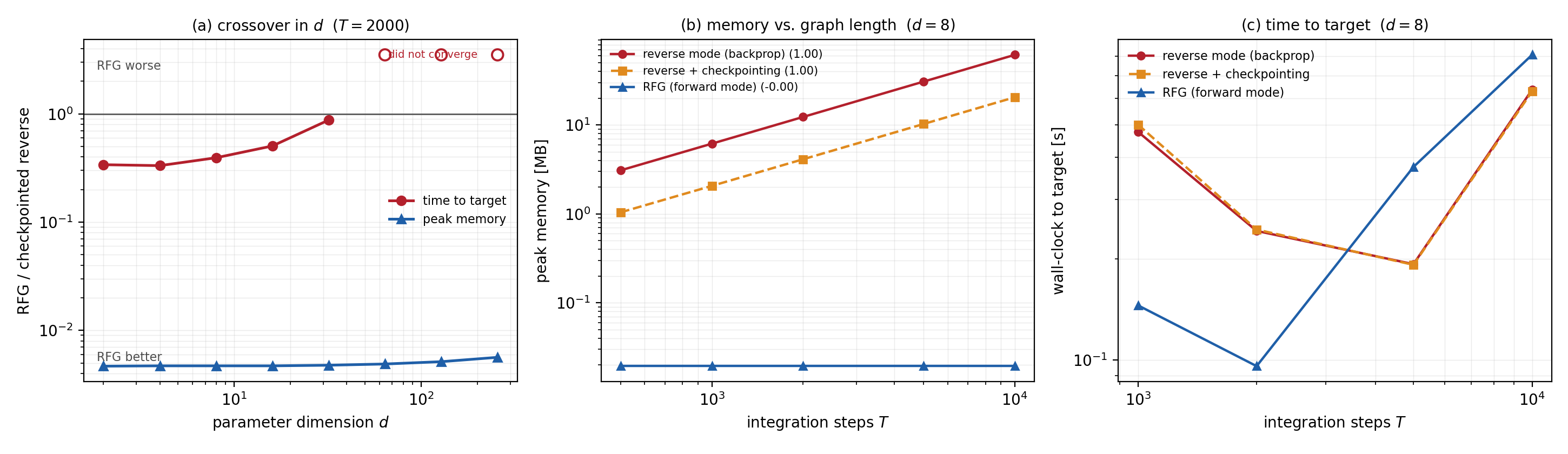}
\caption{(a) Time to target accuracy and peak memory of RFG relative to
checkpointed reverse mode, against the parameter dimension $d$, for the
one-dimensional inverse problem at $T=2000$. Values below unity favor RFG.
Open markers denote configurations at which the inverse problem is ill posed
and no method converges.
(b) Peak memory against the number of integration steps at $d=8$, with fitted
scaling exponents. The forward mode is flat.
(c) Cost per step on the same axis.
Time to solution is compared at fixed $T$ in panel (a) rather than along the
axis of panel (b), since each value of $T$ constitutes a different optimization
problem and iteration counts are not comparable across it.}
\label{fig:eff-regime}
\end{figure}

\begin{figure}[!ht]
\centering
\includegraphics[width=\textwidth]{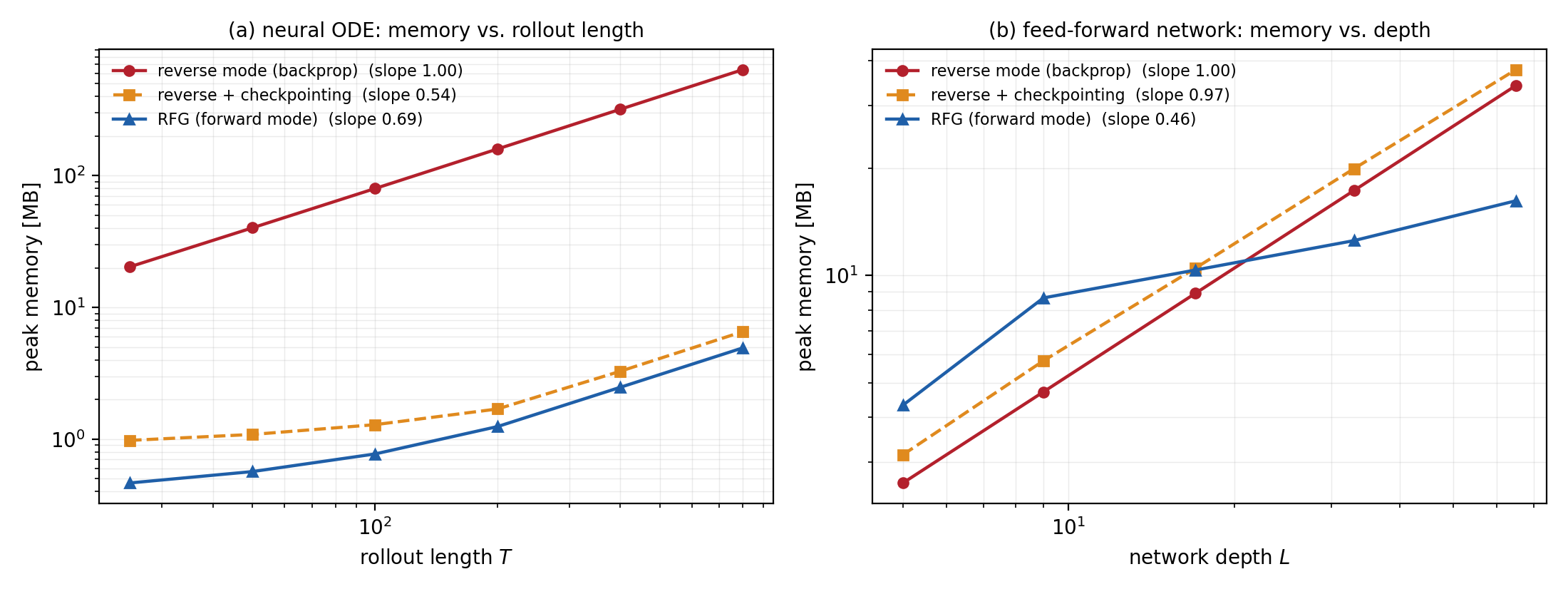}
\caption{Peak memory for the neural network problems.
(a) Neural ODE against rollout length; (b) feed-forward network against depth.
Fitted scaling exponents are shown in the legend.
The crossover in panel (b) near depth $17$ is the prediction of
section~\ref{sec:complexity}: the forward mode carries a tangent of parameter
size and is cheaper only once $\ell\,\mathcal{M}(f)$ dominates parameter
storage.}
\label{fig:eff-memory}
\end{figure}

\section{Summary}
In this paper, we studied a randomized forward mode AD-based gradient (RFG),
namely a family of random gradient estimators formed from the directional
derivative of the objective along a random vector. The family generalizes the forward gradient of
\cite{baydin2022gradients} to an arbitrary direction distribution and contains
randomized coordinate descent as a special case (Remark~\ref{rmk:RCD}).

The second-moment analysis of Theorems~\ref{thm:2nd-moment-general} and
\ref{thm:2nd-moment-RFG} gives the expected squared error of the estimator, and
identifies the optimal direction distribution as the one of least kurtosis, with
variance $\sigma^2 = (d+\kappa_4-1)^{-1}$. Among distributions with independent
components the Bernoulli distribution attains the minimum $\kappa_4 = 1$. The
resulting estimator is deliberately biased, being an unbiased estimate of a
fixed positive multiple of the gradient, and by Remark~\ref{rmk:RFG-scaling}
that multiple is absorbed into the learning rate. What governs performance is the
relative variance, which the minimum-kurtosis choice minimizes. For quadratic
objectives, Theorem~\ref{thm:RFG-GD-convg} establishes linear convergence of the
RFG-based gradient descent method with an explicit optimal learning rate and a
rate degraded by the factor $(\kappa_4+d-1)^{-1}$, and
Theorem~\ref{thm:RFG-PHB-convg} gives an exact operator recursion for the
heavy ball variant, which Proposition~\ref{prop:rate-PHB-k4=1} reduces to a
collection of $2\times2$ problems.

The computational experiments verify these predictions and delimit the regime in
which the method is advantageous. The kurtosis ordering predicted by the
analysis is reproduced on quadratic objectives and persists beyond the reach of
the theory, on the Rosenbrock and Ackley functions and in the training of
physics-informed and operator networks. The efficiency comparison is governed by
the ratio $\ell/d$ of graph length to parameter dimension. Measured against
reverse mode with checkpointing, RFG requires between one and three orders of
magnitude less memory on the parameter identification problems of
section~\ref{sec:inverse}, and its memory is flat in the length of the
computational graph where that of the reverse mode grows linearly. It is
correspondingly slower to reach a given accuracy, by a median factor of $3.2$,
at an iteration penalty of $6.5$ that agrees with the factor $d+\kappa_4-1=8$
predicted by the analysis. The same behavior holds when the unknown is a network
rather than a coefficient vector: on the neural right-hand side of
section~\ref{sec:P4} the memory is again flat in the horizon, reaching
$40.2\times$ that of a checkpointed reverse mode at $T=5\times10^4$, at a
wall-clock cost of $4.2$ against an unmodified reverse mode under momentum
gradient descent, or $3.21$ against a checkpointed one under Adam. For neural network training, where $d$ is large
and $\ell$ short, it is not competitive, and Table~\ref{table:adam-cost}
quantifies how far: the memory advantage falls to $2.3$ while the time cost rises
to a factor of $17$.

Two limits of the method deserve emphasis. The iteration penalty is not an
artifact of the step size but an information count, since one tangent sweep
returns a single scalar where an adjoint sweep returns all $d$ components of the
gradient, and no step-size rule removes it. And the memory advantage
that repays that penalty is not always present: it is the tape that the forward
mode declines to write, so when the graph is short relative to $d$ there is
little to decline, and the $d$-dimensional tangent the forward mode must carry
can cost as much as the tape it avoids. The RFG is the appropriate choice when
the tape does not fit, which is to say when $\ell/d$ is large, and not otherwise.

Several directions remain open. The convergence analysis is confined to
quadratic objectives, and its extension to strongly convex,
Polyak--{\L}ojasiewicz, or general nonconvex settings is the most immediate
theoretical question. Assumption~\ref{assumption:prob-dist} requires the
components of the random vector to be independent, yet the condition
$\mathbb{E}[\bz\bz^\top]=\sigma^2I$ admits correlated constructions, whose
analysis we leave open.
Numerical examples were provided to verify our theoretical findings and to
quantify the regime in which the forward mode is preferable, including the
recovery of a viscosity profile from a long time integration of a viscous
Burgers equation.

Because RFG is stochastic, it is natural to ask whether it escapes local minima
more readily than deterministic descent. We investigated this on two multimodal
test problems and found the outcome to depend on the tuning protocol to an
extent that prevented any firm conclusion, and a properly budget-matched study
is left to future work. Finally, the regime identified here suggests that the
natural applications of the method are large-scale inverse problems in which the
adjoint state cannot be stored, such as full-waveform inversion and variational
data assimilation, together with training on memory-constrained hardware.


\section*{Acknowledgments}
K. Shukla gratefully acknowledges the support from the Air Force Office of Science and Research (AFOSR) under OSD/AFOSR MURI Grant
FA9550-20-1-0358 and the
Office of Naval Research (ONR) Vannevar Bush grant N00014-22-1-2795.
Y. Shin was partially supported for this work by the NRF grant funded by the Ministry of Science and ICT of Korea (RS-2023-00219980).

\appendix \label{sec:appendix}

\section{Proof of Theorem~\ref{thm:2nd-moment-RFG}} \label{app:thm:2nd-moment-RFG}
\begin{proof}
    Since $f$ is convex and $L$-strongly smooth,
    we have 
    \begin{align*}
        \nabla_{\bm{z}} f(\bm{x}) \le 
        \nabla_{\bm{z},h} f(\bm{x}) \le  \nabla_{\bm{z}} f(\bm{x}) + \frac{hL}{2}\|\bm{z}\|^2.
    \end{align*}
    Let $\bm{z} = (z_1,\dots,z_d)^\top$
    and observe that for any $j \in \{1,\dots,d\}$,
    \begin{align*}
        |\nabla_{\bm{z},h} f(\bm{x}) z_j - \nabla_{\bm{z}} f(\bm{x}) z_j |^2
        \le \frac{h^2}{4}L^2\|\bm{z}\|^4 |z_j|^2.
    \end{align*}
    By summing over $j$, we obtain
    $\|\nabla_{\bm{z},h} f(\bm{x})\bm{z} - \nabla_{\bm{z}} f(\bm{x})\bm{z}\|^2
        \le \frac{h^2}{4}L^2\|\bm{z}\|^6$.
    Also, note that  
    \begin{align*}
        \|\nabla_{\bm{z}} f(\bm{x})\bm{z} - \nabla f(\bm{x})\|^2
        = (\nabla f(\bm{x}))^\top (I - (2-\|\bm{z}\|^2)\bm{z}\bm{z}^\top)\nabla f(\bm{x}),
    \end{align*}
    and $\|\nabla_{\bm{z}} f(\bm{x})\bm{z} - \nabla f(\bm{x})\| \le \|I - \bm{z}\bm{z}^\top\| \|\nabla f(\bm{x})\|$.
    Since $\|I-\bm{z}\bm{z}^\top\| = \max\{1, |\|\bm{z}\|^2-1|\}$, we have 
    \begin{align*}
        &\mathbb{E}\left[ \|\nabla_{\bm{z},h} f(\bm{x})\bm{z} - \nabla f(\bm{x})\|^2\right] 
        \\
        &=
        \mathbb{E}\left[ \|\nabla_{\bm{z},h} f(\bm{x})\bm{z} - \nabla_{\bm{z}} f(\bm{x})\bm{z}\|^2\right]
        +
        \mathbb{E}\left[ \|\nabla_{\bm{z}} f(\bm{x})\bm{z} - \nabla f(\bm{x})\|^2\right]
        \\
        &\qquad +2
        \mathbb{E}\left[\langle \nabla_{\bm{z},h} f(\bm{x})\bm{z} - \nabla_{\bm{z}} f(\bm{x})\bm{z}, \nabla_{\bm{z}} f(\bm{x})\bm{z} - \nabla f(\bm{x})\rangle 
        \right]
        \\
        &\le \frac{h^2}{4}L^2\cdot \mathbb{E}[\|\bm{z}\|^6]
        +
        \mathbb{E}\left[ \|\nabla_{\bm{z}} f(\bm{x})\bm{z} - \nabla f(\bm{x})\|^2\right]
        \\
        &\qquad +2\mathbb{E}\left[\|\nabla_{\bm{z},h} f(\bm{x})\bm{z} - \nabla_{\bm{z}} f(\bm{x})\bm{z}\|\cdot  \| \nabla_{\bm{z}} f(\bm{x})\bm{z} - \nabla f(\bm{x})\|
        \right]
        \\
        &\le \frac{h^2}{4}L^2\cdot \mathbb{E}[\|\bm{z}\|^6]
        +
        (\nabla f(\bm{x}))^\top \mathbb{E}[(I - (2-\|\bm{z}\|^2)\bm{z}\bm{z}^\top)]\nabla f(\bm{x})
        \\
        &\qquad
        +hL \|\nabla f(\bm{x})\| \cdot \mathbb{E}\left[\|\bm{z}\|^3\max\{1, |\|\bm{z}\|^2-1|\}
        \right].
    \end{align*}
    Therefore, we conclude that 
    \begin{align*}
        \mathbb{E}\left[ \|\nabla_{\bm{z},h} f(\bm{x})\bm{z} - \nabla f(\bm{x})\|^2\right] 
        &\le 
        \frac{h^2}{4}L^2\cdot \mathbb{E}[\|\bm{z}\|^6] \\
        &\qquad
        +hL \|\nabla f(\bm{x})\| \cdot \mathbb{E}\left[\|\bm{z}\|^3\max\{1, |\|\bm{z}\|^2-1|\}
        \right] \\
        &\qquad + 
        (\nabla f(\bm{x}))^\top \mathbb{E}[(I - (2-\|\bm{z}\|^2)\bm{z}\bm{z}^\top)]\nabla f(\bm{x}).
    \end{align*}

    \textbf{Case 1:} Suppose $z_j$'s are iid random variables 
    whose first and third moments are zeros.
    It then follows from Lemma~\ref{app:lemma-equality} that 
    \begin{align*}
        \mathbb{E}[\|\bm{z}\|^4\|\bm{z}\|^2] &= \sigma^6 d\left[ \kappa_6 + 3(d-1)\kappa_4 + (d-2)(d-1) \right], 
        \\
        \mathbb{E}[\bm{z}\bm{z}^\top] &= \sigma^2I, \\
        \mathbb{E}[\|\bm{z}\|^2\bm{z}\bm{z}^\top] &= \sigma^4 (d + \kappa_4-1)I,
    \end{align*}
    where 
    $\sigma^2 = \mathbb{E}[Z^2]$ and 
    $\kappa_k = \mathbb{E}[Z^k]/\sigma^k$.
    Hence, we have
    \begin{align*}
        \mathbb{E}\left[\|\nabla_{\bm{z},h}^\text{FM} f(\bm{x}) -\nabla f(\bm{x})\|^2\right]
        &\le 
        \frac{h^2}{4}L^2\sigma^6 \left(\kappa_6 + (d-2+3\kappa_4)(d-1) \right)d
        \\
        &\quad\quad + hL\|\nabla f(\bm{x})\| \cdot 
        \mathbb{E}[\|\bm{z}\|^3\max\{1, |\|\bm{z}\|^2-1|\}] \\
        &\quad\qquad+
        (\sigma^4 (d + \kappa_4 -1) + 1-2\sigma^2)\|\nabla f(\bm{x})\|^2.
    \end{align*}

    \textbf{Case 2:} Suppose $\bm{z}$ is a random vector from the uniform distribution on $\{\pm \alpha\bm{e}_i\}_{i=1}^d$ with $\alpha > 0$.
    It then follows from Lemma~\ref{app:lemma-equality-2} that 
    \begin{align*}
        \mathbb{E}[\|\bm{z}\|^6] &= \alpha^6, \quad 
        \mathbb{E}[\bm{z}\bm{z}^\top] = \frac{\alpha^2}{d}I, \quad 
        \mathbb{E}[\|\bm{z}\|^2\bm{z}\bm{z}^\top] = \frac{\alpha^4}{d}I,
    \end{align*}
    and $\mathbb{E}[\|\bm{z}\|^3\max\{1,|\|\bm{z}\|^2-1|\}] = \alpha^3\max\{1,|\alpha^2-1|\}$
    which gives
    \begin{align*}
        \mathbb{E}\left[\|\nabla_{\bm{z},h}^\text{FM} f(\bm{x}) -\nabla f(\bm{x})\|^2\right]
        &\le 
        \frac{h^2}{4}L^2 \alpha^6 + hL\|\nabla f(\bm{x})\|\alpha^3\max\{1,|\alpha^2-1|\}
        \\
        &\quad + \left(\frac{(\alpha^2-1)^2}{d}+1-\frac{1}{d}\right)\|\nabla f(\bm{x})\|^2.
    \end{align*}
\end{proof}

\section{Useful Equalities}

\begin{lemma} \label{app:lemma-equality}
    Let $a, b \in \mathbb{R}^d$
    and $A \in \mathbb{R}^{m\times d}$ whose $k,i$ component is denoted by 
    $A_{k,i}$.
    Let $B$ be a matrix.
    Suppose Assumption~\ref{assumption:prob-dist} holds.
    Then, 
    \begin{align*}
        &\mathbb{E}[ |\langle a, \bm{z}\rangle|^2 \|\bm{z}\|^2] = \sigma^4\left(\kappa_4 + d-1 \right)\|a\|^2, \\
        &\mathbb{E}[|\langle a,\bz\rangle|^2|\langle b,\bz\rangle|^2 \|\bz\|^2]
        = \sigma^6\left[\alpha_d(\sum_{i=1}^d a_i^2b_i^2)
        + 
        \beta_d
        (\sum_{i\ne j} (a_i^2b_j^2 + 2a_ib_ia_jb_j))
        \right], \\
        &\mathbb{E}[\|A\bz\|^4\|\bz\|^2] =
        \sigma^6
        \sum_{k,l}
        \left[\alpha_d (\sum_{i=1}^d A_{k,i}^2A_{l,i}^2)
        + 
        \beta_d
        (\sum_{i\ne j} (A_{k,i}^2A_{l,j}^2 + 2A_{k,i}A_{l,i}A_{k,j}A_{l,j}))
        \right], \\
        &\mathbb{E}[
        (\bz^\top M \bz) 
        \bm{z}\bm{z}^\top ] = 
        \sigma^4 T,
    \end{align*}
    where $\alpha_d= \kappa_6 + (d-1)\kappa_4$,
    $\beta_d = d-2+2\kappa_4$ and 
    \begin{align*}
        T_{ij} 
        &= 
        \begin{cases}
            \sum_{k} M_{kk} + M_{ii}(\kappa_4-1) & \text{if } i = j, \\
            M_{ij} + M_{ji} & \text{if } i \ne j
        \end{cases}.
    \end{align*}
\end{lemma}
\begin{proof}
    For the first equality, 
    note that $\mathbb{E}[z_i] = 0$ and $\mathbb{E}[z_iz_j] = \sigma^2 \delta_{ij}$.
    It then can be checked that 
    \begin{align*}
        \mathbb{E}[\|\bz\|^2|\langle a, \bz\rangle|^2]
        &= \mathbb{E}[ \sum_{k} z_k^2 \cdot \sum_{i,j} a_ia_j z_iz_j]
        =\sum_{k} \sum_{i,j} a_ia_j \mathbb{E}[z_iz_jz_k^2] \\
        &= \sum_{k} \left[ a_k^2(\mathbb{E}[z^4] - \sigma^4) + \sigma^4 \|a\|^2 \right]
        \\
        &= \left(\mathbb{E}[z^4]+(d-1)\sigma^4 \right)\|a\|^2.
    \end{align*}
    
    For the second equality,
    we categorize the index set $[d]^4 = \{1,\dots,d\}^4$ by
    the following three cases:
    \begin{align*}
        \Omega_1 &= \{(i,j,k,l) \in [d]^4: i=j=k=l \}, \\
        \Omega_2 &= \{(i,j,k,l) \in [d]^4 : i=j\ne k=l \text{ up to permutation}\}, \\
        \Omega_3 &= [d]^4 \backslash (\Omega_1 \cup \Omega_2).
    \end{align*}
    Note that $|\Omega_1| = d$
    and 
    $|\Omega_2| = 3d(d-1)$.
    It then can be checked that 
    \begin{align*}
        (i,j,k,l) \in \Omega_1 &\implies 
        \mathbb{E}[z_iz_jz_kz_lz_1^2]
        =\begin{cases}
            \sigma^6\kappa_4 & \text{if } i \ne 1 \\
            \sigma^6\kappa_6 & \text{if } i = 1
        \end{cases},
        \\
        (i,j,k,l) \in \Omega_2 &\implies 
        \mathbb{E}[z_iz_jz_kz_lz_1^2]
        =\begin{cases}
            \sigma^6 & \text{if } i \ne 1 \text{ and } k \ne 1 \\
            \sigma^6\kappa_4 &  \text{if } i = 1 \text{ and } k \ne 1 \\
        \end{cases},
        \\
        (i,j,k,l) \in \Omega_3 &\implies 
        \mathbb{E}[z_iz_jz_kz_lz_1^2]
        = 0.
    \end{align*}
    Observe that 
    since $|\langle a,\bz\rangle|^2|\langle b,\bz\rangle|^2
        = \sum_{i,j} a_ia_jz_iz_j \cdot \sum_{k,l} b_kb_l z_kz_l$, we have
    \begin{align*}
        &|\langle a,\bz\rangle|^2|\langle b,\bz\rangle|^2 \\
        &= \sum_{(i,j,k,l) \in \Omega_1} + \sum_{(i,j,k,l) \in \Omega_2} + \sum_{(i,j,k,l) \in \Omega_3} a_ia_jb_kb_lz_iz_jz_kz_l 
        \\
        &= \sum_{i=1}^d a_i^2b_i^2z_i^4 + \sum_{i\ne j} (a_i^2b_j^2 + 2a_ib_ia_jb_j) z_i^2z_j^2 +  \sum_{(i,j,k,l) \in \Omega_3} a_ia_jb_kb_lz_iz_jz_kz_l.
    \end{align*}
    Thus,
    \begin{align*}
        &\mathbb{E}[|\langle a,\bz\rangle|^2|\langle b,\bz\rangle|^2 z_1^2]
        \\
        &= \sigma^6\kappa_4 (\sum_{i=1}^d a_i^2b_i^2) + \sigma^6(\kappa_6 - \kappa_4) a_1^2b_1^2
        + \sigma^6 \sum_{i\ne j} (a_i^2b_j^2 + 2a_ib_ia_jb_j)
        \\
        &\quad
        + \sigma^6 (\kappa_4-1) \sum_{j\ne 1} (a_1^2b_j^2 + 2a_1b_1a_jb_j)
        + \sigma^6 (\kappa_4-1) \sum_{i\ne 1} (a_i^2b_1^2 + 2a_ib_ia_1b_1),
    \end{align*}
    \begin{align*}
        &\mathbb{E}[|\langle a,\bz\rangle|^2|\langle b,\bz\rangle|^2 z_k^2]
        \\
        &= \sigma^6\kappa_4 (\sum_{i=1}^d a_i^2b_i^2)
        + \sigma^6 \sum_{i\ne j} (a_i^2b_j^2 + 2a_ib_ia_jb_j)
        \\
        &\quad
        + \sigma^6 (\kappa_4-1) \sum_{j\ne k} (a_k^2b_j^2 + 2a_kb_ka_jb_j)
        + \sigma^6 (\kappa_4-1) \sum_{i\ne k} (a_i^2b_k^2 + 2a_ib_ia_kb_k)
        \\
        &\quad 
        -\sigma^6(\kappa_4 - \kappa_6) a_k^2b_k^2
        -2\sigma^6 (\kappa_4-1) (a_k^2b_k^2 + 2a_kb_ka_kb_k)
    \end{align*}
    which gives
    \begin{align*}
        &\mathbb{E}[|\langle a,\bz\rangle|^2|\langle b,\bz\rangle|^2 \|\bz\|^2]
        \\
        &= \sigma^6\left[(\kappa_6 + (d-1)\kappa_4) (\sum_{i=1}^d a_i^2b_i^2)
        + 
        (d-2+2\kappa_4)
        (\sum_{i\ne j} (a_i^2b_j^2 + 2a_ib_ia_jb_j))
        \right].
    \end{align*}

    For the third equality, 
    let $\alpha_d = \kappa_6 + (d-1)\kappa_4$
    and $\beta_d = d-2+2\kappa_4$.
    By applying the second equality, we obtain
    \begin{align*}
        &\mathbb{E}[\|A\bz\|^4 \|\bz\|^2]
        = \mathbb{E}[(\sum_{k=1}^m |\langle A_{k,:}, \bz\rangle|^2)^2 \|\bz\|^2] 
        = \sum_{k,l} 
        \mathbb{E}[|\langle A_{k,:}, \bz\rangle|^2 
        |\langle A_{l,:}, \bz\rangle|^2 \|\bz\|^2] \\
        &=
        \sigma^6
        \sum_{k,l}
        \left[\alpha_d (\sum_{i=1}^d A_{k,i}^2A_{l,i}^2)
        + 
        \beta_d
        (\sum_{i\ne j} (A_{k,i}^2A_{l,j}^2 + 2A_{k,i}A_{l,i}A_{k,j}A_{l,j}))
        \right].
    \end{align*}

    For the fourth equality, 
    observe that for any $k$, 
    \begin{align*}
        \mathbb{E}[(\bz^\top M\bz) z_k^2]
        &= \sum_{i,j}^d  M_{ij} 
        \mathbb{E}[z_iz_jz_k^2]
        = \sum_{i=1}^d  M_{ii} 
        \mathbb{E}[z_i^2z_k^2]
        \\
        &= \sigma^4(\kappa_4 - 1)M_{kk} + \sigma^4 \sum_{i=1}^d M_{ii}.
    \end{align*}
    Also, observe that for $k\ne l$, 
    \begin{align*}
        \mathbb{E}[(\bz^\top M \bz)z_kz_l]
        &= \sum_{i,j} M_{ij}\mathbb{E}[z_iz_jz_kz_l]
        = (M_{kl} + M_{lk})\sigma^4,
    \end{align*}
    which gives $\mathbb{E}[(\bz^\top M\bz) \bz \bz^\top] = \sigma^4 T$ where
    \begin{align*}
        T_{ij} 
        &= 
        \begin{cases}
            \sum_{k} M_{kk} + M_{ii}(\kappa_4-1) & \text{if } i = j, \\
            M_{ij} + M_{ji} & \text{if } i \ne j
        \end{cases}.
    \end{align*}
\end{proof}

\begin{lemma} \label{app:lemma-equality-2}
    Let $a, b \in \mathbb{R}^d$, let
    $A \in \mathbb{R}^{m\times d}$ have $k,i$ component $A_{k,i}$,
    and let $M \in \mathbb{R}^{d\times d}$.
    Let $\bm{z} \in \mathbb{R}^d$ be a random vector from the uniform distribution on 
    $\{\pm \alpha\bm{e}_i\}_{i=1}^d$ with $\alpha>0$. 
    Then $\|\bz\|^2 = \alpha^2$ almost surely, and
    \begin{align*}
        &\mathbb{E}[\|\bz\|^6] = \alpha^6, \quad
        \mathbb{E}[\bz\bz^\top] = \frac{\alpha^2}{d}I, \quad
        \mathbb{E}[\|\bz\|^2\bz\bz^\top] = \frac{\alpha^4}{d}I, \\
        &\mathbb{E}[ |\langle a, \bm{z}\rangle|^2 \|\bm{z}\|^2] = \frac{\alpha^4}{d} \|a\|^2, \\
        &\mathbb{E}[|\langle a,\bz\rangle|^2|\langle b,\bz\rangle|^2 \|\bz\|^2]
        = \frac{\alpha^6}{d} \sum_{i=1}^d a_i^2b_i^2, \\
        &\mathbb{E}[\|A\bz\|^4\|\bz\|^2] =
        \frac{\alpha^6}{d}\sum_{i=1}^d \|A\bm{e}_i\|^4, \\
        &\mathbb{E}[(\bz^\top M \bz)\bm{z} 
        \bm{z}^\top ] = 
        \frac{\alpha^4}{d}\,\emph{\text{diag}}(M_{ii}).
    \end{align*}
\end{lemma}
\begin{proof}
    Write $\bz = \varepsilon\alpha\bm{e}_I$ with $I$ uniform on $\{1,\dots,d\}$ and
    $\varepsilon$ a symmetric sign independent of $I$, so that
    $\|\bz\|^2 = \alpha^2$ deterministically. Then
    $\langle a,\bz\rangle = \varepsilon\alpha a_I$, whence
    $|\langle a,\bz\rangle|^2 = \alpha^2a_I^2$ and
    $|\langle a,\bz\rangle|^2|\langle b,\bz\rangle|^2 = \alpha^4 a_I^2b_I^2$,
    the two factors sharing the \emph{same} index $I$. Taking
    $\mathbb{E}$ over $I$ multiplies by $d^{-1}\sum_i$. The third display follows
    by summing the second over the rows $a = A_{k,:}$, $b = A_{l,:}$, since
    $\sum_{k,l}\sum_i A_{k,i}^2A_{l,i}^2 = \sum_i\|A\bm{e}_i\|^4$. For the last,
    $\bz^\top M\bz = \alpha^2M_{II}$ and $\bz\bz^\top = \alpha^2\bm{e}_I\bm{e}_I^\top$.
\end{proof}

\section{Proof of Theorem~\ref{thm:RFG-GD-convg}} \label{app:thm:RFG-GD-convg}

\begin{proof}
Observe that 
\begin{align*}
    \bx^{(k+1)} &= \bx^{(k)} - \eta \left( \frac{f(\bx^{(k)}+h\bz) - f(\bx^{(k)})}{h} \right) \bz \\
    &= \bx^{(k)} - \eta \bz \bz^\top \left( A^\top(A\bx^{(k)}-b) \right) - \eta\frac{h}{2}\|A\bz\|^2 \bz.
\end{align*}
Let $\epsilon^{(k)} = \bx^{(k)} - \bx^*$ where $\bx^* = (A^\top A)^{-1}A^\top b$. Then,
\begin{align*}
    \epsilon^{(k+1)} &= (I - \eta \bz\bz^\top A^\top A)\epsilon^{(k)} - \frac{h\eta}{2}\|A\bz\|^2 \bz, \\
    \|\epsilon^{(k+1)}\|^2 &= \|P_{\bz}^A(\eta) \epsilon^{(k)} \|^2 + \frac{h^2\eta^2}{4}\|A\bz\|^4 \|\bz\|^2
    - h\eta \|A\bz\|^2 \langle \bz, P_{\bz}^A(\eta)\epsilon^{(k)} \rangle,
\end{align*}
where $P_{\bz}^A(\eta):= I - \eta \bz\bz^\top A^\top A$.
Let $\mathbb{E}[\bz\bz^\top] =\beta^2 I$
and 
$\mathbb{E}[\|\bz\|^2\bz\bz^\top] = \gamma_dI$.
It then can be checked that 
\begin{align*}
    \mathbb{E}[P_{\bz}^A(\eta)^\top P_{\bz}^A(\eta)]
    &= I - 2\eta \beta^2 A^\top A + \eta^2 \gamma_d A^\top AA^\top A
    \\
    &= (1-\frac{\beta^4}{\gamma_d})I
    + \big(\frac{\beta^2}{\sqrt{\gamma_d}} I - \eta \sqrt{\gamma_d} A^\top A\big)^2.
\end{align*}
Suppose $\bz$ satisfies $\mathbb{E}[\|A\bz\|^2 \langle \bz, P_{\bz}^A(\eta)\epsilon^{(k)} \rangle]= 0$, which is equivalent to 
\begin{align*}
    \mathbb{E}[\|A\bz\|^2 \langle \bz, c \rangle] &= 0, \quad
    \mathbb{E}[\|A\bz\|^2 \|\bz\|^2 \langle \bz, c \rangle] = 0 \quad  \forall c \in \mathbb{R}^d.
\end{align*}
Let $\lambda_{\min}=\lambda_1 \le \lambda_2 \le \cdots \le \lambda_d = \lambda_{\max}$ be eigenvalues of $A^\top A$ and $q_1,\dots,q_d$ be the corresponding eigenvectors. 
Then,
\begin{align*}
    &\mathbb{E}[\|\epsilon^{(k+1)}\|^2] - \frac{\eta^2 h^2}{4} \mathbb{E}[\|A\bz\|^4\|\bz\|^2]
    \\
    &= 
    \sum_{i=1}^d \left(1 - \frac{\beta^4}{\gamma_d} + (\frac{\beta^2}{\sqrt{\gamma_d}} - \eta \sqrt{\gamma_d}\lambda_i)^2 \right)
    |\langle \epsilon^{(k)}, q_i\rangle|^2.
\end{align*}
Bounding each coefficient by its maximum over $i$ and using
$\sum_i|\langle \epsilon^{(k)},q_i\rangle|^2 = \|\epsilon^{(k)}\|^2$ gives a
contraction factor $\max_i\big(1-\frac{\beta^4}{\gamma_d}+(\frac{\beta^2}{\sqrt{\gamma_d}}-\eta\sqrt{\gamma_d}\lambda_i)^2\big)$
valid for every $\epsilon^{(k)}$. That factor is minimized over $\eta$ when the
residuals at $\lambda_{\min}$ and $\lambda_{\max}$ are equal in magnitude and
opposite in sign, i.e. when
$\eta = \frac{\beta^2}{\frac{\lambda_{\max}+\lambda_{\min}}{2}\gamma_d}$.
With this choice,
\begin{align*}
    \mathbb{E}[\|\epsilon^{(k+1)}\|^2]
    \le r_\text{rate}\,\mathbb{E}[\|\epsilon^{(k)}\|^2]
    + 
    \frac{h^2}{4}\frac{\beta^4\,\mathbb{E}[\|A\bz\|^4\|\bz\|^2]}{(\frac{\lambda_{\max}+\lambda_{\min}}{2})^2\gamma_d^2}
\end{align*}
where $\kappa_A = \frac{\lambda_{\max}}{\lambda_{\min}}$
and $r_\text{rate} = 1 - \frac{\beta^4}{\gamma_d}[1-(\frac{\kappa_A - 1}{\kappa_A +1})^2]$.
Unrolling this recursion and summing the resulting geometric series,
whose ratio is $r_\text{rate} \in [0,1)$, gives
\begin{align*}
    \mathbb{E}[\|\epsilon^{(k)}\|^2]
    \le r_\text{rate}^k \|\epsilon^{(0)}\|^2
    +
    \frac{h^2}{4}\frac{\mathbb{E}[\|A\bz\|^4\|\bz\|^2]}{(\frac{\lambda_{\max}+\lambda_{\min}}{2})^2\gamma_d\left[1-(\frac{\kappa_A-1}{\kappa_A+1})^2\right]},
\end{align*}
where the factor $\beta^4$ cancels against $1-r_\text{rate}$.

\textbf{Case 1:}
Let $\bz = (z_1,\dots,z_d)^\top$
be a random vector satisfying Assumption~\ref{assumption:prob-dist}.
It then can be checked (also from Lemma~\ref{app:lemma-equality}) that 
$\mathbb{E}[\bz\bz^\top] =\sigma^2 I$
and 
$\mathbb{E}[\|\bz\|^2\bz\bz^\top] = \gamma_dI$
where $\gamma_d= \sigma^4(\kappa_4 + (d-1))$.
Since the 1st, 3rd, and 5th moments are zeros, 
$\mathbb{E}[\|A\bz\|^2 \langle \bz, P_{\bz}^A(\eta)\epsilon^{(k)} \rangle] = 0$.
%
Since $\sigma^4/\gamma_d = 1/(\kappa_4+d-1)$, 
if $\eta = \frac{1}{\frac{\lambda_{\max}+\lambda_{\min}}{2}(\kappa_4 + d - 1)\sigma^2}$, 
we have
\begin{align*}
    &\mathbb{E}[\|\epsilon^{(k+1)}\|^2]
    \\
    &\le \left(1 - \frac{1}{\kappa_4 + d-1}\left[1-(\frac{\kappa_A - 1}{\kappa_A +1})^2\right]\right)\|\epsilon^{(k)}\|^2
    + 
    \frac{h^2\sigma^2\mathcal{F}(d,\kappa_4,\kappa_6,A)}{(\frac{\lambda_{\max}+\lambda_{\min}}{2})^2(\kappa_4 + d - 1)^2},
\end{align*}
where $\mathcal{F}(d,\kappa_4,\kappa_6,A)$ is defined in the remark following Proposition~\ref{prop:RFGquadratic}.
By letting
\begin{equation*}
    r_\text{rate}:=1 - \frac{1}{\kappa_4 + d-1}\left[1-\left(\frac{\kappa_A - 1}{\kappa_A +1}\right)^2\right]
\end{equation*}
and recursively applying the above inequality,
the desired result is obtained.

\textbf{Case 2:}
Let $\bz = (z_1,\dots,z_d)^\top$ be drawn uniformly from
$\{\pm \alpha \bm{e}_i\}_{i=1}^d$.
Direct calculations give $\beta^2 = \frac{\alpha^2}{d}$, $\gamma_d = \frac{\alpha^4}{d}$ and 
\begin{align*}
    \mathbb{E}[\|A\bz\|^2 \langle \bz, P_{\bz}^A(\eta)\epsilon^{(k)}\rangle] &= 0,
    \quad 
    \mathbb{E}[\|A\bz\|^4\|\bz\|^2] = \alpha^6 \sum_{i=1}^n \|A\bm{e}_i\|^4,
\end{align*}
which lead to
\begin{align*}
    \mathbb{E}[\|\epsilon^{(k+1)}\|^2] - 
    \eta^2 \frac{h^2\alpha^6}{4}\sum_{i=1}^n \|A\bm{e}_i\|^4
    = 
    \sum_{i=1}^d \left(1 - \frac{1}{d} + (\frac{1}{\sqrt{d}} - \eta \sqrt{d}\sigma^2\lambda_i)^2 \right)
    |\langle \epsilon^{(k)}, q_i\rangle|^2.
\end{align*}
With $\eta = \frac{1}{\frac{\lambda_{\max}+\lambda_{\min}}{2}d\alpha^2}$, the expected squared error is bounded by
\begin{align*}
    \mathbb{E}[\|\epsilon^{(k+1)}\|^2]  
    &\le \left(1 - \frac{1}{d}\left[1-(\frac{\kappa_A-1}{\kappa_A+1})^2\right] \right)
    \|\epsilon^{(k)}\|^2
    +
     \frac{h^2\alpha^2}{4d^2(\frac{\lambda_{\max}+\lambda_{\min}}{2})^2}
     \sum_{i=1}^n \|A\bm{e}_i\|^4,
\end{align*}
which completes the proof
with $r_\text{rate}=1 - \frac{1}{d}\left[1-(\frac{\kappa_A-1}{\kappa_A+1})^2\right]$.
\end{proof}

\section{Proof of Theorem~\ref{thm:RFG-PHB-convg}} \label{app:thm:RFG-PHB-convg}

\begin{proof}
    Let $\epsilon_{k+1} = \bx^{(k+1)} - \bx^{*}$.
    It then follows from the update rule \eqref{def:RFG-PHB} of the RFG-based PHB method that 
    \begin{align*}
        \epsilon_{k+1} &= \epsilon_{k} - \eta \nabla_{\bz,h}^\text{FM} f(\bx^{(k)}) + \mu (\epsilon_k - \epsilon_{k-1})
        \\
        &= ((1+\mu)I - \eta \bz\bz^\top A^\top A)\epsilon_k  - \mu  \epsilon_{k-1} - \eta \frac{h}{2}\|A\bz\|^2\bz,
    \end{align*}
    which can be equivalently written as
    $\mathcal{E}_{k+1} = \bm{M} \mathcal{E}_k + \bm{\alpha}$
    where 
    \begin{align*}
        \bm{M} = \begin{pmatrix}
            (1+\mu)I- \eta \bz\bz^\top A^\top A & -\mu I \\
            I & 0
        \end{pmatrix},
        \quad 
        \bm{\alpha}
        =-\eta\frac{h}{2}\|A\bz\|^2\begin{pmatrix}
            \bz \\
            0
        \end{pmatrix}.
    \end{align*}
    Note that it can be checked that $\bm{M} = \bm{U}\overline{\bm{M}} \bm{U}^\top$
    where 
    \begin{equation} \label{def:Moverline}
        \overline{\bm{M}}= \begin{pmatrix}
        (1+\mu)I - \eta U_A^\top \bz\bz^\top U_A \Sigma_A & -\mu I \\ I & 0
    \end{pmatrix}.
    \end{equation}
    Let $\tilde{\mathcal{E}}_{k} = \bm{U}^\top \mathcal{E}_k$
    and $\tilde{\bm{\alpha}} = \bm{U}^\top\bm{\alpha}$.
    Then, we have 
    \begin{equation} \label{eqn-PHB-recursion}
        \tilde{\mathcal{E}}_{k+1} = \overline{\bm{M}}\tilde{\mathcal{E}}_{k} + \tilde{\bm{\alpha}}.
    \end{equation} 
    Also, observe that 
    $\|\mathcal{E}_{k+1}\|^2  = \|\overline{\bm{M}} \tilde{\mathcal{E}}_{k}\|^2 + 2\langle \overline{\bm{M}} \tilde{\mathcal{E}}_{k}, \tilde{\bm{\alpha}}\rangle + \|\tilde{\bm{\alpha}}\|^2$.
    From the assumption, $\bz$ satisfies 
    $\mathbb{E}[\langle \overline{\bm{M}}\tilde{\mathcal{E}}_{k}, \tilde{\bm{\alpha}} \rangle ] = 0$
    where the expectation is taken with respect to $\bm{z}$.
    From this, we obtain
    $\mathbb{E}[\|\mathcal{E}_{k+1}\|^2]
        = \mathbb{E}[\|\overline{\bm{M}} \tilde{\mathcal{E}}_{k}\|^2]
        + \mathbb{E}[\|\tilde{\bm{\alpha}}\|^2]$.
    It then follows from Lemma~\ref{lemma:PHB-map} that
    we have 
    $\mathbb{E}[\|\mathcal{E}_{k+1}\|^2]
        = \tilde{\mathcal{E}}_{k}^\top
        \Phi(I_{2d})
        \tilde{\mathcal{E}}_{k}
        + \mathbb{E}[\|\tilde{\bm{\alpha}}\|^2]$.
    With \eqref{eqn-PHB-recursion},
    we have
    \begin{align*}
        \mathbb{E}[\|\mathcal{E}_{k+1}\|^2]
        &= \tilde{\mathcal{E}}_{k-1}^\top
        \mathbb{E}_{k-1}[\overline{\bm{M}}^\top
        \Phi(I_{2d})
        \overline{\bm{M}}]\tilde{\mathcal{E}}_{k-1}
        + \mathbb{E}[\|\tilde{\bm{\alpha}}\|^2_{\Phi}]
        + \mathbb{E}[\|\tilde{\bm{\alpha}}\|^2]
        \\
        &= \tilde{\mathcal{E}}_{k-1}^\top
        \Phi^{2}(I_{2d})\tilde{\mathcal{E}}_{k-1}
        + \mathbb{E}[\|\tilde{\bm{\alpha}}\|^2_{\Phi}]
        + \mathbb{E}[\|\tilde{\bm{\alpha}}\|^2]
    \end{align*}
    where the expectation over the random variables used in the $(k-1)$th iteration.
    By repeating the above recursion, we have
    $\mathbb{E}[\|\tilde{\mathcal{E}}_{k}\|^2]
        = \tilde{\mathcal{E}}_{0}^\top
        \Phi^{k}(I_{2d})\tilde{\mathcal{E}}_{0}
        + \sum_{i=0}^{k-1} \mathbb{E}[\|\tilde{\bm{\alpha}}\|^2_{\Phi^i}]$
    and the proof is completed.
\end{proof}

\begin{lemma} \label{lemma:PHB-map}
        Let $\bm{S} = \begin{bmatrix}
            S_1 & S_2^\top \\ S_2 & S_3
        \end{bmatrix} \in \text{Sym}_{2d}$.
        Then $\Phi(\bm{S})=\mathbb{E}[\overline{\bm{M}}^\top \bm{S} \overline{\bm{M}}]$
        where $\Phi$ is defined in \eqref{def:PBH-mapping}
        and $\overline{\bm{M}}$ is defined in \eqref{def:Moverline}.
\end{lemma}
\begin{proof}[Proof of Lemma~\ref{lemma:PHB-map}]
        Let $J = (1+\mu)I - \eta U_A^\top \bz\bz^\top U_A \Sigma_A$.
        Since $\mathbb{E}[\bz \bz^\top] =\beta^2 I$, we have 
        $V:=\mathbb{E}[J] = (1+\mu)I - \eta \beta^2 \Sigma_A$, which is diagonal.
        We then have
        \begin{align*}
            \mathbb{E}[\overline{\bm{M}}^\top \bm{S} \overline{\bm{M}}]
            &=
            \mathbb{E}[\begin{bmatrix}
        J^\top  & I \\ -\mu I & 0
    \end{bmatrix}
        \begin{bmatrix}
        S_1 & S_2^\top \\ S_2 & S_3
    \end{bmatrix}\begin{bmatrix}
        J & -\mu I \\ I & 0
    \end{bmatrix}]
    \\
    &=
    \begin{bmatrix}
        \mathbb{E}[J^\top S_1J] + (S_2V)^\top + S_2V + S_3 & -\mu (S_1V+S_2^\top)^\top \\ -\mu (S_1V+S_2^\top) & 
        \mu^2 S_1
    \end{bmatrix}.
    \end{align*}
    Observe that 
    \begin{align*}
        \mathbb{E}[J^\top S_1J]
        &=
        \mathbb{E}[((1+\mu)I - \eta \Sigma_AU_A^\top \bz\bz^\top U_A ) S_1((1+\mu)I - \eta U_A^\top \bz\bz^\top U_A \Sigma_A)]
        \\
        &= (1+\mu)^2 S_1 -(1+\mu)\eta
        (\Sigma_A U_A^\top \mathbb{E}[\bz\bz^\top] U_A S_1+ S_1 U_A^\top \mathbb{E}[\bz\bz^\top]U_A\Sigma_A)
        \\
        &\qquad+ \eta^2
        \Sigma_A U_A^\top 
        \mathbb{E}[
        (\bz^\top U_AS_1U_A^\top \bz)\bz
        \bz^\top ]U_A \Sigma_A.
    \end{align*}
    Substituting $\mathbb{E}[\bz\bz^\top] = \beta^2I$ in the middle term and
    the definition of $H_4$ in the last, we obtain
    \begin{align*}
        \mathbb{E}[J^\top S_1J]
        = H_1 - S_2V - (S_2V)^\top - S_3,
    \end{align*}
    which together with the $(1,2)$, $(2,1)$ and $(2,2)$ blocks computed above
    establishes the lemma.

    If in addition $S_1$ is diagonal, so that $S_1$ and $\Sigma_A$ commute, the
    first two terms may be written more compactly as
    \begin{align*}
        \mathbb{E}[J^\top S_1J]
        &= (V^\top V - \eta^2\beta^4\Sigma_A^2)S_1
        + \eta^2
        \Sigma_A U_A^\top 
        \mathbb{E}[
        (\bz^\top U_AS_1U_A^\top \bz)\bz
        \bz^\top ]U_A \Sigma_A,
    \end{align*}
    which is the form used in Proposition~\ref{prop:rate-PHB-k4=1}.
\end{proof}

\bibliographystyle{siamplain}
\bibliography{main}
\end{document}